\documentclass[12pt]{article}
\usepackage[utf8]{inputenc}

\usepackage{svgcolor}
\usepackage{csquotes}
\def\showauthornotes{0}

\def \R {\mathbb{R}}

\def \E {\mathbb{E}}

\def \eps {\epsilon}
\def \al {\alpha}
\renewcommand{\Pr}{\mathop{\bf Pr\/}}
\newcommand{\wt}[1]{\widetilde{#1}}

\newcommand{\Var}{\mathrm{Var}}

 \def\1{\bm{1}}

\usepackage{amsmath,amssymb,amsthm,amsfonts,latexsym,bm,bbm,xspace,graphicx,float,mathtools,mathdots,physics}
\usepackage{braket,caption,subcaption,ellipsis,xcolor,textcomp,hhline,pifont,combelow,booktabs}
\usepackage[colorlinks=true, allcolors=blue]{hyperref}
\usepackage{color}
\usepackage{times}
\usepackage{fullpage}
\usepackage{tikz-cd}
\usepackage[shortlabels]{enumitem}
\usepackage{thm-restate}
\usepackage{cleveref}
\usepackage{mdframed}

\usepackage{mathrsfs}

\newtheorem{theorem}{Theorem}[section]
\newtheorem{lemma}[theorem]{Lemma}
\newtheorem{claim}[theorem]{Claim}
\newtheorem{proposition}[theorem]{Proposition}
\newtheorem{fact}[theorem]{Fact}
\newtheorem{corollary}[theorem]{Corollary}

\newtheorem{definition}[theorem]{Definition}
\newtheorem{remark}[theorem]{Remark}

\newtheorem{observation}[theorem]{Observation}

\newcommand{\poly}{\operatorname{poly}}
\newcommand{\polylog}{\operatorname{polylog}}

\DeclareMathAlphabet{\mathsfit}{\encodingdefault}{\sfdefault}{m}{sl}
\SetMathAlphabet{\mathsfit}{bold}{\encodingdefault}{\sfdefault}{bx}{n}

\renewcommand{\wt}{\mathsf{weight}}

\mathchardef\mhyphen="2D

\newcommand{\calP}{\mathcal{P}}

\newcommand{\calB}{\mathcal{B}}

\newcommand{\calQ}{\mathcal{Q}}
\newcommand{\calR}{\mathcal{R}}

\newcommand{\calS}{\mathcal{S}}
\newcommand{\calT}{\mathcal{T}}

\newcommand\numberthis{\addtocounter{equation}{1}\tag{\theequation}}

\newcommand{\gam}{\gamma}

\newcommand{\cm}{\mathsf{M} }

\newcommand{\corr}{\mathsf{correct}}

\newcommand{\fp}{\mathsf{P}}
\newcommand{\bti}{\mathsf{BacktrackingInt}}

\newcommand{\Lovasz}{Lov\'asz\xspace}

\renewcommand*{\circle}[1]{\scalebox{0.85}{\footnotesize
    \tikz[baseline=(char.base)]{
        \node[shape=circle,draw,inner sep=0.5pt, minimum size=14pt](char) {   \ifx&#1&
        \color{white} $i$
        \else
        $#1$
        \fi};
    }}}

\renewcommand*{\square}[1]{\scalebox{0.85}{\footnotesize
    \tikz[baseline=(char.base), square/.style={regular polygon,regular polygon sides=4}]{
        \node[draw,square, inner sep=0pt, minimum size=18pt](char) {
        \ifx&#1&
        \color{white} $t$
        \else
        $#1$
        \fi};
    }}}
    
\newcommand*{\hexagon}[1]{\scalebox{0.85}{\footnotesize
    \tikz[baseline=(char.base), square/.style={regular polygon,regular polygon sides=6}]{
        \node[draw,square, inner sep=-3pt, minimum size=18pt](char) {
        \ifx&#1&
        \color{white} $t$
        \else
        $#1$
        \fi};
    }}}

\usepackage[notes=true,later=false,camera=false]{dtrt}
\usepackage{algorithm}
\usepackage{algorithmic}
\usepackage{amsmath}
\DeclareMathOperator{\arccosh}{arccosh}

\usepackage{tikz}

\newcommand{\todo}[1]{\textcolor{red}{(#1)}}

\allowdisplaybreaks
\newcommand{\calL}{\mathcal{L}}
\newcommand{\q}{\mathfrak q}

\usepackage{comment}

\usepackage{mathpazo}
\usepackage{microtype}

\ifnum\showauthornotes=1
\newcommand{\Authornote}[3]{{\sf\footnotesize\color{#3}{[#1: #2]}}}
\else
\newcommand{\Authornote}[3]{}
\fi

\newcommand{\jnote}[1]{\Authornote{Jeff}{#1}{blue}}
\newcommand{\anote}[1]{\Authornote{Aaron}{#1}{cyan}}

\newcommand{\snote}[1]{\Authornote{Sofia}{#1}{orange}}

\begin{document}
\title{Sharp Phase Transition for Ellipsoid Fitting}%
\author{
Sofia de la Cerda
\thanks{{University of Chicago}. \textit{sofiasl@uchicago.edu} }
\and 
Aaron Potechin 
\thanks{{University of Chicago}. \textit{potechin@uchicago.edu}}
\and
Madhur Tulsiani
\thanks{{Toyota Technological Institute at Chicago}. \textit{madhurt@ttic.edu}}
\and
Jeff Xu
\thanks{{Toyota Technological Institute at Chicago}.\textit{jeffxusichao@ttic.edu}}
}
\maketitle
\begin{abstract}
	We resolve the ellipsoid fitting conjecture of Saunderson,
Chandrasekaran, Parrilo, and Willsky~\cite{SCPW12} up to a vanishing
factor.
Concretely, for $m$ independent Gaussian points in dimension $d$, we show that with high probability,\begin{enumerate}
	\item for $m \leq (1-o_d(1)) \cdot d^2/4$, there exists a centered ellipsoid passing through all $m$ points;
	\item for $m\geq (1+o_d(1) )\cdot d^2/4$, no such ellipsoid exists.
\end{enumerate}
This confirms that the ellipsoid fitting problem has a sharp phase transition at $d^2/4$.

\end{abstract} 
\clearpage
\newpage
\thispagestyle{empty}
\setcounter{page}{0}

\thispagestyle{empty}
\setcounter{page}{0}
\thispagestyle{empty}
\setcounter{page}{0}
\tableofcontents
\clearpage \newpage

\section{Introduction}

A basic question at the intersection of convex geometry, random matrix theory,
and semidefinite programming is the following: given a sample of size $m$ from the standard normal distribution \(v_1,\ldots,v_m\in \mathbb R^d\),
when does there exist an (origin-centered, possibly degenerate) ellipsoid passing
through all of them?  Equivalently, for what values of \(m=m(d)\) is there,
with high probability, a positive semidefinite matrix \(\Lambda\succeq 0\) such that  $v_i^\top \Lambda v_i = 1$ for all $i\in [m]$?

The ellipsoid fitting problem was introduced by Saunderson \cite{Saunderson11}
and studied by Saunderson, Chandrasekaran,
Parrilo, and Willsky \cite{SCPW12,SPW13}. The latter authors conjectured that the
problem exhibits a sharp phase transition at
\[
    m \approx \frac{d^2}{4}.
\]

The natural dimension-counting argument 
shows that this is impossible when $m >  \binom{d+1}{2} \sim d^2/2$, the dimension of the space of $d\times d$ symmetric matrices. From this perspective, the conjecture predicts that the positive semidefiniteness constraint forces an additional factor-two loss, lowering the feasibility threshold from \(d^2/2\) to \(d^2/4\). To the best of our knowledge, this na\"{i}ve dimension-counting argument remains the only known result—and hence the state of the art—for the impossibility direction of ellipsoid fitting prior to this work.


A long line of work has progressively improved the lower bound on the feasible
regime.  Early results established feasibility for
\(m\le O(d^{6/5-\varepsilon})\) \cite{SPW13}. 
In the past few years, there has been a revival of interest in this problem, with connections to degree-two Sum-of-Squares and random matrix constructions,
obtaining bounds of the form \(m\le O(d^{3/2-\varepsilon})\) \cite{GJJPR20} and then
\(m\le d^2/\operatorname{polylog}(d)\) \cite{PTVW22,KD22}. Most recently, a line of independent works \cite{HKPX23,BMMP24,TW25} has obtained bounds that are tight up to an absolute constant factor, establishing feasibility for \(m\le c d^2\) for some universal constant \(c>0\).\footnote{The coefficient is not explicitly tracked in the existing analysis, though we believe a safe estimate of $c=1/1000$ would have sufficed in \cite{HKPX23}, albeit with a substantial gap from the conjectured threshold of $1/4$.} These works use substantially different techniques, but the underlying construction is essentially the same identity-perturbation ansatz analyzed in prior works. In parallel with this line of progress, \cite{MD24, MB25} provided non-rigorous statistical-physics evidence pointing to \(d^2/4\) as the transition threshold.



Despite this active progress, the sharp constant \(1/4\) remains open. We emphasize that the
remaining gap is not merely technical in the analysis, but more importantly, conceptual: empirical evidence,
together with heuristic predictions from statistical physics, suggests that the
previously studied constructions already fail well below the conjectured
threshold, around \(m\approx d^2/10\). Thus, even setting aside the difficulty of
rigorously analyzing a construction near the threshold, a central challenge is to
find, or even identify, a plausible candidate construction capable of reaching
the \(d^2/4\) regime.

\paragraph{Our Results}
In this work, we resolve both directions of the ellipsoid fitting conjecture and establish a sharp phase transition at \(m=d^2/4\), up to lower-order terms:
\begin{enumerate}
    \item \textbf{Feasibility below the threshold:} an ellipsoid can be fitted through the points throughout the regime
    \(
        m < \frac{d^2}{4}\,;
    \)
    
    \item \textbf{Infeasibility above the threshold:} no such ellipsoid exists once
   \(
        m > \frac{d^2}{4}\,.
    \)
  
\end{enumerate}

More concretely, we prove the following two theorems.

\begin{theorem}[Main theorem for the positive side]
\label{thm:ellipsoid-thm}
Let
\(
    m \leq \bigl(1-o_d(1)\bigr) \cdot d^2/4\)
and let \(v_1,\dots,v_m\sim \mathcal N(0,I_d/d)\) be independent \(d\)-dimensional Gaussian vectors. Then, with probability \(1-o_d(1)\), there exists an ellipsoid passing through all \(m\) points.
\end{theorem}
\begin{theorem}[Main theorem for the negative side]
\label{thm:ellipsoid-refutation-thm}
Let
\(
    m \geq \bigl(1+o_d(1)\bigr) \cdot d^2/4
\)
and let \(v_1,\dots,v_m\sim \mathcal N(0,I_d/d)\) be independent \(d\)-dimensional Gaussian vectors. Then, with probability \(1-o_d(1)\), there does not exist any ellipsoid passing through all \(m\) points.
\end{theorem}

\begin{remark}
    These theorems establish the threshold at $\frac{d^2}{4}$ and hold with $o_d(1)$ instantiated as $\varepsilon = 1/\poly(\log \log d)$. We make no attempt to optimize this factor.
\end{remark}

We prove both theorems by exhibiting and analyzing explicit SDP witnesses. The two solutions are drastically different in their constructions while their analyses share a strikingly similar structure.  Among prior works on the ellipsoid fitting conjecture, the closest to ours is~\cite{HKPX23}: we also make use of the language and techniques of graph matrices to construct and analyze the correlated random matrices arising in the problem. A key new ingredient is the application and extension of techniques from the recent work~\cite{PotechinXu2026Theta}, which establishes a sharp bound for the \Lovasz-theta function of \(G(n,1/2)\). In the following section, we explain the connection between these results and describe the main technical ideas underlying our proof.

\paragraph{Roadmap of Our Work}

%
%
%

Both our feasibility and refutation results are established via explicit SDP solutions.

In the next section, we provide an overview of our techniques and highlight the main technical challenges that arise in the construction and the analysis. We focus exclusively on the positive side of the conjecture, namely, the existence of a fitting ellipsoid below the threshold. We then explain in \cref{sec:refutation} how these arguments can be adapted to establish the corresponding refutation result above the threshold.  As a preview, all of the technical ingredients developed for the primal program will carry over to the refutation argument with only minor modifications in the setup of the iterative procedure. 

The remainder of the paper is organized as follows.

\begin{enumerate}
	\item In~\cref{sec:refutation}, we show that an analogous procedure produces an SDP dual witness that refutes the existence of a fitting ellipsoid for $m>d^2/4$.
	\item The Gram matrix $M=\calL\calL^\top$ and its inverse play a key role in our construction. In~\cref{sec:MP-sec}, we give an explicit inversion via orthogonal polynomials and graph matrices.
This justifies the correction mechanism of our iterative scheme and, more importantly, enables a decomposition of $M^{-1}$ in the graph-matrix basis that is additionally amenable to a stability analysis.  \item In~\cref{sec:inner-matrix-Q}, we establish the free independence of backbone-dangling shapes as well as the variance of the inner matrix $Q$ in the limit.
\item Finally, in~\cref{sec:formal-truncation}, we incorporate details for the truncation procedure to instantiate the construction,
    and complete the proof of the main theorem.
\end{enumerate} 

\paragraph{Notation.}
Throughout this work, all probabilistic statements hold with probability
\(1-o_d(1)\).  We write \(\|\cdot \|_{sp}\) for the spectral norm of a matrix. 
We use \(\tau\) as a generic placeholder for an arbitrary shape; accordingly, \(\cm_\tau\) denotes the graph matrix associated with \(\tau\).  The symbols \(\alpha\) and \(\beta\)
are reserved for specific shapes that play distinguished roles in our analysis. We additionally use the following notations:
\begin{enumerate}
	\item $I_d$ denotes the identity matrix of dimension $d\times d$. We ignore the dependence on $d$ when the dimension is clear;
	\item  $\gam = \frac{2m}{d^2}$ is a proxy for the threshold;
	\item $\calL$ evaluates the constraint values, $\calL^*$ lifts coefficient vectors to linear combinations of the rank-one forms $v_tv_t^\top$, and $M =\calL\calL^\top$
is their Gram matrix. \[
\calL(X)\coloneqq (v_t^\top Xv_t)_{t\in[m]},
\qquad
\calL^*(w)\coloneqq \sum_{t\in[m]} w[t]v_tv_t^\top,
\qquad
M\coloneqq \calL\calL^*.
\]
\item  
Let $\calS:=\operatorname{Im}(\mathcal L^*)
=\operatorname{span}\{v_iv_i^\top:i\in[m]\}$ and let  \[ \Pi \coloneqq \Pi_S  = \calL^*M^{-1}\calL \] denote the orthogonal projection to $\calS$, and $\Pi_S^\perp$ the projection to its orthogonal complement.

\item   $\eta_\tau$ denotes the deviation vector $\eta_\tau = \calL(\cm_\tau)$ for a shape $\tau$. When the subscript is left out, we use $\eta = \eta_\emptyset \coloneqq \calL(I_d) - \1_m $ as well as $\eta_0$ interchangeably to denote the raw deviation of the identity perturbation.
\item We use the letter $\calR$ to denote the perturbation matrix from the identity-perturbation construction, \[ \calR \coloneqq  \calL^* M^{-1} \eta =  \calL^*M^{-1} (\calL (I_d) -\1_m)\,.   \] In other words, $I_d -\calR $ satisfies the linear constraints. 
\item We define the variance of a $d\times d$ matrix $X$  as \(
\Var(X) \coloneqq \frac{1}{d }\cdot\E \Tr[ XX^T ]\,. \) With some abuse of notation,  we also write $\Var(\tau) = \Var(\cm_\tau)$ for the graph matrix of shape $\tau$.
\end{enumerate}
%

\paragraph{Concurrent and Independent Work} While preparing this manuscript, we became aware of two concurrent and independent works~\cite{KS26} and~\cite{MW26} that also obtain comparable results.

\paragraph{Acknowledgments} We are grateful to Sidhanth Mohanty for bringing the independent, concurrent work to our attention. We also thank Frederic Koehler and Youngtak Sohn for coordinating the preparation of the manuscripts.

\paragraph{AI Disclosure}  We used generative AI tools to help polish the writing. All the conceptual contributions are our own, and we take full responsibility for any errors. We also note that the connection to the prior work of \cite{kogan2025extremal} was brought to the attention of one of the authors by generative AI after we completed our proof via graph matrices. 

\section{Technical Overview of Ellipsoid Fitting}
%

In this section, we provide an overview of our techniques and highlight the main technical challenges for the positive side of the conjecture.
\label{sec:connection-theta} 

Our starting point is the recent breakthrough of Potechin and Xu \cite{PotechinXu2026Theta} which showed that with high probability, the \Lovasz-theta function of a random $G(n,1/2)$ graph has value $(1 \pm{o_n(1)})\sqrt{n}$. This paper gives an intricate iterative construction of a solution to the SDP for the \Lovasz-theta function and analyzes it by exploiting connections between graph matrices, orthogonal polynomials, and free independence. This is promising as it gives a novel solution for this SDP which is explicit and optimal. Moreover, the construction of \cite{PotechinXu2026Theta} matched a distinctive feature of the statistical physics predictions of \cite{MD24} for ellipsoid fitting. In particular, \cite{MD24} predicts that for ellipsoid fitting, nuclear norm minimization gives a solution where half of the eigenvalues are zero. For the construction of \cite{PotechinXu2026Theta}, it also turns out that half of the eigenvalues are zero!

\subsection{Overview of the Iterative Construction }
At a high level, there are two components in the construction:
\begin{enumerate}
    \item choose a nonnegative function \(F\), which will be applied spectrally;
    \item construct an inner matrix \(Q\) such that the spectrally transformed
    matrix satisfies the affine constraints of the SDP. Moreover, the inner matrix $Q$ is constructed by an iterative process.
\end{enumerate}

In this section we describe the iterative construction of \(Q\). We start by introducing the non-negative function $F$ and its polynomial expansion. This is exactly the same choice as that for the prior work of \cite{PotechinXu2026Theta} for the \Lovasz-theta function.
\begin{definition}[Non-negative Function \(F\) and its Chebyshev Expansion]
\label{def:F-function}
Let \(\mu_{\mathrm{sc}}\) be the semicircle distribution on \([-2,2]\), with density
\(
    f_{\mathrm{sc}}(x)
    =
    \frac{\sqrt{4-x^2}}{2\pi}.
\)
Define
\[F(x)=\textup{max}(2x,0)=\begin{cases}
        2x, & x>0,\\
        0,  & x\leq 0.
    \end{cases}\]
Let
\(
    P_j(x):=U_j\left(\frac{x}{2}\right),
\)
where \(U_j\) is the \(j\)-th Chebyshev polynomial of the second kind. Then
\(\{P_j\}_{j\ge0}\) is an orthonormal basis for
\(L^2([-2,2],\mu_{\mathrm{sc}})\), with \(P_0=1\) and \(P_1=x\).

We write
\[
    F(x)
    =
    C_F+x+\sum_{j\ge2} b_j P_j(x),
\]
where
\(
    C_F
    =
    \mathbb E_{x\sim\mu_{\mathrm{sc}}}[F(x)]
    =
    \frac{8}{3\pi},
    \mathbb E_{x\sim\mu_{\mathrm{sc}}}[F(x)P_1(x)]=1,
\)
and, for \(j\ge2\),
\(
    b_j\coloneqq
    \E_{x\sim\mu_{\mathrm{sc}}}[F(x)P_j(x)].
\)
\end{definition}

  We use the following normalization to ensure the identity term has coefficient $1$. Once the $F$ function is fixed, the main difficulty is to find an inner matrix \(Q\) for which the spectral
matrix \(\frac1{C_F}F(  Q)\) satisfies the affine constraints. 
%
\[
    \frac{1}{C_F}F( Q)
    =
    I+\frac{1}{C_F}\cdot Q +\sum_{j\ge2}   \frac{b_j}{C_F}\cdot  P_j(  Q).
\]

 Towards that end, we introduce an iterative process to construct \(Q\), following the same recipe 
 as \cite{PotechinXu2026Theta}:  
\begin{enumerate}
	\item At level \(i\), we maintain an inner matrix \(Q_i\);
	\item Apply the non-negative spectral map to obtain a PSD matrix $X_i \leftarrow \frac{1}{C_F} F(Q_i) $;
	\item $X_i$ may not satisfy the affine constraints exactly, and we update the inner matrix to $Q_{i+1}$ to correct for the deviation. \end{enumerate} 

\paragraph{Identity Perturbation as Initialization}
We now briefly recap the prior approaches based on identity-perturbation as that in turn gives us an initialization for the inner matrix $Q_0$.
The identity-perturbation construction searches for a witness of the
form
\(
    \Lambda
    =
    I_d - \calR  \) with\[
    \eta \coloneqq \calL(I)-\1_m ,
    \qquad
    \calR \coloneqq \calL M^{-1} \eta\,.
    \]
The linear constraints are then automatically satisfied, and the entire problem
reduces to proving
\(
    \|\calR\|_{\mathrm{sp}} < 1.
\)
This is precisely the route taken in \cite{KD22,PTVW22,HKPX23}. The matrix \(\calR\) depends on the inverse of the highly correlated
random matrix \(M\), and the positivity of \(\Lambda=I_d-\calR\) hinges on
controlling the operator norm of this correction. 

However, it is known that $\calR$ is not of small norm close to the threshold, and this is where the non-negative spectral map $F$ comes to our rescue as it produces a positive albeit large norm matrix as a result. The remaining challenge is then to correct its affine constraint violations 
without destroying positivity.

\jnote{im debating on whether to elaborate on the specific correction term here or simply keep it abstract. i'm also ok with the current version in which we fully specify the correction primitive, and mention we are gonna tell u why later.}
\paragraph{Iteration in Action}
Motivated by the prior approaches via identity perturbation, our starting point is 
\(
    Q_0
    \coloneqq - C_F \cdot \calR\,,
\)
which is precisely the perturbation component from the identity-perturbation construction with the scaling chosen so that
\(
    I+   Q_0/C_F \)
satisfies the affine constraints. The particular choice of scaling comes from the Chebyshev expansion of $F$: 
\[
    \frac1{C_F}F(  Q_0)  = I+ \frac{1}{C_F} Q_0 + \sum_{j>1} \frac{b_j}{C_F} \cdot P_j(Q_0) \,.
   \]
Next, observe that the first two terms in the equation already satisfy the affine constraints, and the only affine deviation
comes from the unscaled higher-order terms
\(
 H_{0} \coloneqq \sum_{j\ge2} b_j \cdot   P_j( Q_0)\) crucially without the $1/C_F$ scaing. 
Define the (unscaled) first residual vector by \[ \eta_1= \calL(\sum_{j\ge2}  b_j \cdot   P_j(  Q_0)  ) = \calL(H_{0}  ) \in \R^m\]
Therefore, the higher-degree component $ \sum_{j>1} \frac{b_j}{C_F} \cdot P_j(Q) 
$ in the resulting function of $ \frac{1}{C_F} F(Q_0) $ would incur a deviation of $\frac{\eta_1}{C_F}$ when measured with respect to the affine constraints in this notation. 

\paragraph{Picking a Freely Independent Correction Mechanism}
Analogous to the identity perturbation construction, one straightforward choice 
is to pick the correction term as \(
\mathcal{L}^{*}(M^{-1}(-\eta_1)) 
\)
as this would then allow us to group the correction term with the previously troublesome term as \[ 
H_{0} +  \mathcal{L}^*(M^{-1}(-\eta_1) )
\]
to correct for the prior deviation term incurred by $H_0$ since the affine constraints evaluated with these two terms are simply \[ 
\calL\left(H_{0} + \mathcal{L}^*(M^{-1}(-\eta_1) )\right) = \eta_{1} + \calL^*\calL(M^{-1}(-\eta_1)) = \eta_1-\eta_1 = 0\,
\] 
since $\calL^*\calL = M$. \jnote{add a remark saying how this can be viewed as a projection to the image of rank-$1$ forms as \[ 
\calL^*M^{-1} \calL 
\]
is a projection matrix 
}
 However, as we show in~\cref{sec:cooking-correction}, surprisingly this is not the ``right'' correction mechanism to be employed. For our machinery to apply, it is crucial that we can view each summand in the inner matrix as a ``free'' matrix with semicircular spectrum. The naive choice above does not suffice for this purpose: it is indeed true that it contains various matrices with semicircular spectrum; however, it also contains an extra copy of $\gam \cdot H_{0} $ once we examine its polynomial expansion.
 
 To salvage this correction machinery, we consider an alternative correction mechanism given by 
 \[ 
 \Delta_0  = \corr(H_0) \coloneqq \frac{1}{1-\gamma}\left(\calL^*M^{-1} \calL (-H_0) + \gam \cdot  H_{0}  \right)
 \]
It is straightforward to verify that $\Delta_1$ achieves the same task of removing prior deviations when combined with $ H_{0}$:
\[
\calL(\Delta_0) = \frac{1}{1-\gamma}  \calL\calL^*M^{-1}(-\calL(H_0) )+\frac{\gam }{1-\gamma} \calL( H_{0})\, = -\frac{1}{1-\gamma}\eta_1  + \frac{\gamma}{1-\gamma} \eta_1 =   -\eta_1\,.
 \]

Notice that the scaling is picked so 
that when we apply $\frac{1}{C_F} F(Q_1)$ to $Q_1 = Q_0 + H_0$, the degree-$1$ term from the update gives exactly \[
\frac{1}{C_F} (Q_1-Q_0) = \frac{\Delta_0}{C_F} =  \frac{\corr(H_0) }{C_F} \,, \] the desired correction term for $H_0$ in the polynomial expansion for $F(Q_0)/C_F$.

We emphasize that in general, especially for $i\ge 1$, \(H_i\) should not be confused with the seemingly more natural quantity 
\[ 
\sum_{j\ge 2} b_j P_j(Q_i)
\]
The difference is important: the contribution \(P_j(Q_i)\) from the previous level has already been absorbed into the newly added linear correction term \(P_1(Q_i-Q_{i-1})\).  Concretely, in the case of $i=1$, \(H_1\) is therefore defined to include only the genuinely new higher-order deviation created by the update from \(Q_0\) to \(Q_1\).

At this point, one may hope to show that the iterative scheme works by showing
that the newly created deviations become progressively smaller across subsequent
iterations. Before instantiating the precise correction term to be chosen and turning to the analysis, we first give a full summary of the iterative process. We caution, however, that even though the following summary completely specifies the construction, the resulting object is arguably quite opaque. 
To make the analysis tractable and transparent, we will later introduce the graph-matrix viewpoint. 
\begin{mdframed}[linewidth=0.2pt]\textbf{Summary of the Initialization and Iterative Procedure.} 
\\\\
\textbf{Initialization:} Pick \(Q_0 = -C_F \cdot \calR\) so that \(I+Q_0/C_F\) satisfies the affine constraints.
\\\\
\textbf{Iterative Update $i\rightarrow i+1$:}
 Given \(Q_i\), 
  define \[ 
  H_{i} \coloneqq \sum_{j\ge 2} b_j \cdot \left(P_j(Q_{i}) - P_j(Q_{i-1})  \right) \numberthis \label{eq:primal-corr}
  \]
  for $i\ge 1$, with \( 
  H_{0} \coloneqq \sum_{j\ge 2} b_j \cdot P_j(Q_0)\,.
  \) 
  
  Set the update as 
  \[ 
  \Delta_{i}  = \corr(H_i) \coloneqq \frac{1}{1-\gamma} \left(\calL^*M^{-1} \calL(-H_{i} ) + \gam \cdot H_{i}     \right)
   \]
%
 
 Finally, we set $Q_{i+1} \coloneqq Q_i+ \Delta_{i}$.
\end{mdframed}

	Geometrically, our update defined above can be viewed as \[ 
	\corr(H_i) =   \frac{1}{1-\gamma} (-\Pi_S +\gam I) H_i =  \frac{1}{1-\gamma} \Pi_S^\perp H_i - H_i \,,	\]
	by observing $ \Pi_S = \calL^*M^{-1} \calL  $ and  \[ 
	-\Pi_S + \gam I =   (I-\Pi_S)- (1-\gam )I = \Pi_S^\perp - (1-\gamma) \cdot  I\,.
	\]
	where we remind the reader $\Pi_S=\calL^*M^{-1}\calL$ and $\Pi_S^{\perp}=I-\Pi_S$ are orthogonal projections to $\calS = \operatorname{Im}(\mathcal L^*)$ and its orthogonal complement respectively. 

\begin{remark}
    The update admits a closed form
    \begin{align*}
        Q_{i+1} &= Q_0+\frac{1}{1-\gamma}\Pi_S^\perp(F(Q_i)-Q_i-C_FI)-(F(Q_i)-Q_i-C_FI)\,.
    \end{align*}
  \end{remark}
\jnote{i think your current version is correct but it might require a few lines of proof...add a proof to the appendix? its likely commented out below somewhere..}

At the heart of the iterative process is the following invariant with an upshot that the needed correction term becomes progressively smaller which 
ultimately allows us to terminate the iteration.

\begin{proposition}[Invariant for Iterative Process] \label{prop:invariant-primal}
	For any $i \ge 1$, we have
	\[
 \calL( \frac{1}{C_F } F(Q_i) )  - \1_m = \frac{1}{C_F} \cdot  \calL(H_{i}).
\]
In words, at the end of each iteration $i$, the entire affine deviation is incurred by the new
higher-order term \(H_{i}\).
\end{proposition}
In particular, this invariant is agnostic to the specific correction primitive: it holds for any update rule \(\Delta(H_i)\) satisfying
\(
    \calL\bigl(\Delta(H_i)+H_i\bigr)=0.
\)
The subsequent analysis, however, relies crucially on the particular choice of \(\corr(\cdot )\) defined above.

\begin{proof}[Proof of \cref{prop:invariant-primal}]

For the base case $i=0$, initialization gives
\(
    \calL\!\left(I+\frac{Q_0}{C_F}\right)=\1_m.
\)
%
Now suppose the claim holds at level \(i\). Since
\(
    Q_{i+1}-Q_i=\Delta_i
\)
and
\(
    F(Q_{i+1})-F(Q_i)=\Delta_i+H_{i+1},
\)
we have
\[
\begin{aligned}
    \calL\!\left(\frac{1}{C_F}F(Q_{i+1})\right)-\1_m
    &=
    \frac{1}{C_F}\calL(H_i)
    +\frac{1}{C_F}\calL(\Delta_i)
    +\frac{1}{C_F}\calL(H_{i+1})\\
    &=
    \frac{1}{C_F}\calL(H_{i+1}).
\end{aligned}
\]
This completes the induction.
\end{proof}

One should view the iterative process described here as an ``idealized'' process. In~\cref{sec:concrete-primal-graph-mat}, we give a concrete procedure that simplifies the above process by additionally discarding negligible norm components throughout the iterative process to streamline the analysis.

\subsection{Graph Matrices:
 Visualizing Orthogonal Polynomials} \label{sec:graph-matrix}

 To demystify the above construction and its analysis, we now introduce our main tool---graph matrices.
 
\paragraph{Overview of Graph Matrices.}
Graph matrices provide a versatile framework for analyzing structured random matrices
arising in average-case problems. They have played a central role in proving
Sum-of-Squares lower bounds for a wide range of average-case inference problems
\cite{BHKKMP16, PR20, pang:LIPIcs.CCC.2021.26, JPRTX, JPRX23, KPX24, NGCA24, Xu25, PX25,
PotechinXu2026Theta}. Beyond the SoS literature, graph matrices have also been used
to analyze power-sum decompositions of polynomials \cite{BHKX22}, the ellipsoid
fitting conjecture \cite{PTVW22, HKPX23}, and first-order iterative algorithms,
including belief propagation and approximate message passing \cite{JP24}.

For completeness, we begin with a brief introduction to graph-matrix techniques specialized to our setting, with particular emphasis on their nontrivial connection to the orthogonal polynomials that will appear later. Readers familiar with graph matrices may skip the discussion through the subsection ``Visualizing Orthogonal Polynomials.'' 

Throughout this section, the underlying random input is a collection of
Gaussian vectors
\(G =
  \{v_1,\ldots,v_m\} 
\) in dimension $d$ 
drawn independently from the distribution specified above. Equivalently, we write
\(G\in \mathbb R^{m\times d}\) for the matrix whose \(i\)-th row is \(v_i\).
The graph-matrix framework applies much more generally, for instance to random
graphs and  non-product distributions; we refer the reader to the references above
for a more thorough introduction.

\begin{definition}[Shape]
A shape \(\tau\) is a tuple
\(
    \tau=(V(\tau),U_\tau,V_\tau,E(\tau))
\)
associated with a multigraph \((V(\tau),E(\tau))\). Each vertex in \(V(\tau)\) is
assigned a vertex type, which determines the range from which its labels are drawn.
Each edge \(e\in E(\tau)\) is assigned a Fourier index \(t(e)\in \mathbb N\).
The subsets \(U_\tau,V_\tau\subseteq V(\tau)\) are called the left and right
boundaries of the shape, respectively.
\end{definition}

We emphasize that \(V_\tau\) denotes the right boundary of the shape, whereas
\(V(\tau)\) denotes the full vertex set of the underlying graph.

\begin{definition}[Mapping of a shape]
Given a shape \(\tau\), a mapping of \(\tau\) is a function
\(
    \sigma:V(\tau)\to \mathbb N
\)
such that:
\begin{enumerate}
    \item each vertex is assigned a label from the range specified by its vertex type;
    \item \(\sigma\) is injective on vertices of the same type.
\end{enumerate}
\end{definition}

In \cref{fig:M_alpha_beta}, we illustrate the shapes corresponding to the matrices
\(M_\alpha\) and \(M_\beta\) defined in \cref{eq:M-decomposition}, which will be a
central focus of our analysis. These shapes have two vertex types: square vertices
take labels in \([m]\), while circle vertices take labels in \([d]\). The two ovals
indicate the left and right boundaries \(U_\tau\) and \(V_\tau\).

\begin{figure}[ht!]
    \centering
    \begin{subfigure}[b]{0.31\textwidth}
        \centering
        \includegraphics[width=0.85\textwidth]{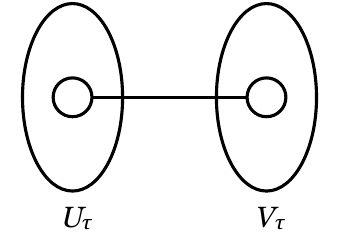}
        \caption{\(\mathsf{GOE}\), zero diagonal.}
        \label{fig:GOE}
    \end{subfigure}
    \hfill
    \begin{subfigure}[b]{0.31\textwidth}
        \centering
        \includegraphics[width=1.2\textwidth]{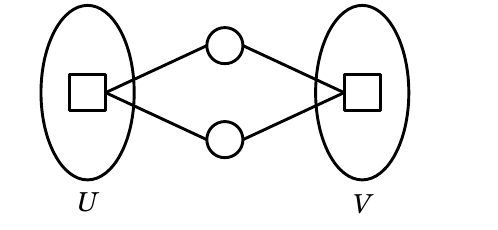}
        \caption{\(M_{\alpha}\).}
        \label{fig:M_alpha}
    \end{subfigure}
    \hfill
    \begin{subfigure}[b]{0.31\textwidth}
        \centering
        \includegraphics[width=\textwidth]{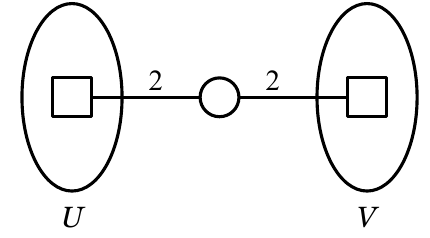}
        \caption{\(M_{\beta}\).}
        \label{fig:M_beta}
    \end{subfigure}

    \captionsetup{width=.9\linewidth}
    \caption{Graph-matrix representations of a \(d\times d\) \(\mathsf{GOE}\)
    matrix with zero diagonal, and of the \(m\times m\) matrices \(M_\alpha\) and
    \(M_\beta\) from \cref{eq:M-decomposition}. Square vertices take labels in
    \([m]\), and circle vertices take labels in \([d]\). The two ovals indicate the
    left and right boundaries \(U_\tau,V_\tau\). If an edge \(e\) is not explicitly
    labeled with an index, then \(t(e)=1\) by default.}
    \label{fig:M_alpha_beta}
\end{figure}

We now describe how a shape gives rise to a matrix once the underlying random input
\(G\) is fixed.

\begin{definition}[Graph matrix of a shape]
\label{def:graph-matrix-for-shape}
Given a shape \(\tau\), its graph matrix \(M_\tau\) is indexed by boundary labelings.
For boundary labelings \(S\) of \(U_\tau\) and \(T\) of \(V_\tau\), define
\[
    \cm_{\tau}[S,T]
    =
    \sum_{\substack{
        \sigma:\,V(\tau)\to \mathbb N \\
        \sigma \text{ is a mapping of } \tau \\
        \sigma(U_\tau)=S,\ \sigma(V_\tau)=T
    }}
    \prod_{e\in E(\tau)}
    \chi_{t(e)}\bigl(G[\sigma(e)]\bigr).
\]
Here, if \(e=\{x,y\}\) is an edge between a square vertex and a circle vertex, then
\(G[\sigma(e)]\) denotes the corresponding entry \(G[\sigma(x),\sigma(y)]\).
\end{definition}

For each entry \(\cm_\tau[S,T]\), the boundary labels are fixed by the constraints
\(\sigma(U_\tau)=S\) and \(\sigma(V_\tau)=T\). Thus the entry is a sum over labelings
of the internal vertices
\(
    V(\tau)\setminus (U_\tau\cup V_\tau).
\)
For example, consider the shapes in \cref{fig:M_alpha_beta}. Suppose
\(G\in \mathbb R^{m\times d}\), with square vertices labeled by elements of \([m]\)
and circle vertices labeled by elements of \([d]\). Then for \(i\neq j\in[m]\),
\begin{align*}
    \cm_{\alpha}[i,j]
    &=
    \sum_{a\neq b\in[d]}
    \chi_1(G[i,a])\chi_1(G[i,b])
    \chi_1(G[j,a])\chi_1(G[j,b]), \\
    \cm_{\beta}[i,j]
    &=
    \sum_{a\in[d]}
    \chi_2(G[i,a])\chi_2(G[j,a]).
\end{align*}
These two matrices will come up again when we analyze the system of affine constraints.
\begin{remark}
The nonzero entries of \(M_\alpha\) may equivalently be written as
\[
    \cm_{\alpha}[i,j]
    =
    2\sum_{a<b\in[d]}
    \chi_1(G[i,a])\chi_1(G[i,b])
    \chi_1(G[j,a])\chi_1(G[j,b]).
\]
The factor of \(2\) will be important in our later analysis, though it can be
safely ignored on a first reading.
\end{remark}

Since the mapping \(\sigma\) is injective on vertices of the same type, and since
\(U_\tau\neq V_\tau\) in both examples in \Cref{fig:M_alpha_beta}, there is no valid
mapping with \(\sigma(U_\tau)=\sigma(V_\tau)\). Consequently, both \(M_\alpha\) and
\(M_\beta\) have zero diagonal.

Following the notation of \cite{HKPX23}, we specialize the graph-matrix framework to
the following setting:
\begin{itemize}
    \item \(G\in\mathbb R^{m\times d}\) is a random Gaussian matrix whose rows are
    \(v_1,\ldots,v_m\sim N(0,\frac1d I_d)\).
    \item The Fourier characters \(\{\chi_t\}_{t\in\mathbb N}\) are the appropriately
    scaled Hermite polynomials.
    \item For all graph matrices appearing in our analysis:
    \begin{itemize}
        \item the left and right boundary labelings satisfy \(|S|=|T|=1\);
        \item there are two vertex types: square vertices take labels in \([m]\), and
        circle vertices take labels in \([d]\).
    \end{itemize}
\end{itemize}

\paragraph{Shape Concatenation and Intersection Terms}

We give a brief overview of the multiplication of graph matrices to shed light on the non-trivial connection with orthogonal polynomials to be introduced. To start with, given shapes $\tau_1, \tau_2$, we call them \emph{composable} if $V_{\tau_1}$ and $U_{\tau_2}$ have matching boundaries, specifically that they have the same number of both square and circle vertices. For any composable pair $\tau_1,\tau_2$, we define their properly concatenated shape as simply the underlying shape obtained by gluing $\tau_1,\tau_2$ together along the boundary (specifically, the right boundary of $\tau_1$ and equivalently the left boundary of $\tau_2$). Formally it is defined as follows, 

\begin{definition}[Proper Shape Concatenation]\label{def:shape-concate}
Given shapes $\tau_1$ and $\tau_2$ that are composable, we take the concatenation $\tau_1 \circ \tau_2$ of $\tau_1$ and $\tau_2$ to be the shape such that:
\begin{enumerate}
		\item $V(\tau_1 \circ \tau_2) = V(\tau_1) \cup V(\tau_2)$ where the vertices in $V_{\tau_1}$ are identified with the vertices in $U_{\tau_2}$.
        \item $U_{\tau_1 \circ \tau_2} = U_{\tau_1}$ and $V_{\tau_1 \circ \tau_2} = V_{\tau_2}$.
		\item $E(\tau_1 \circ \tau_2) = E(\tau_1) \cup E(\tau_2)$. 
\end{enumerate}
	
This definition extends naturally to a sequence of $k$ shapes $\tau_1,\ldots,\tau_k$ such that for each $j \in [k-1]$, $|V_{\tau_j}| = |U_{\tau_{j+1}}|$. We call this a $k$-way (or $k$-fold) concatenation.  We can obtain the concatenation $\tau_1 \circ \cdots \circ \tau_k$ by concatenating these shapes together one by one (it is not hard to check that the order of the concatenations does not matter).
\end{definition}

We illustrate the concatenation of $\al\circ \beta$ in the following diagram, and we use the dotted oval to denote the concatenation boundary of $V_\al = U_\beta$.

\begin{figure}[ht!]
    \centering
    \begin{subfigure}[b]{0.45\textwidth}
        \centering
        \includegraphics[width=\textwidth]{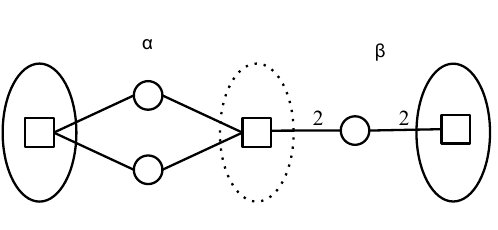}
        \caption{Proper concatenation \(\alpha\circ\beta\).}
        \label{fig:concat-proper}
    \end{subfigure}
    \quad
    \begin{subfigure}[b]{0.45\textwidth}
        \centering
        \includegraphics[width=\textwidth]{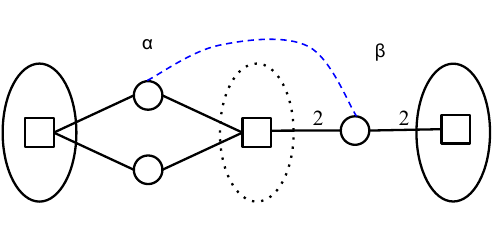}
        \caption{Concatenation with an intersection.}
        \label{fig:concat-intersection}
    \end{subfigure}
    \centering
	 \begin{subfigure}[b]{0.8\textwidth}
        \centering
        \includegraphics[width=\textwidth]{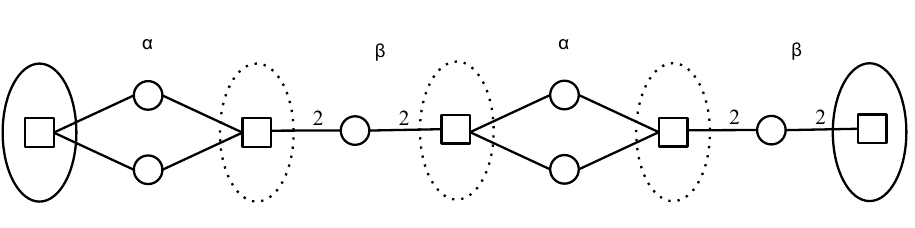}
        \caption{Proper $4$-fold concatenation \(\alpha\circ\beta\circ \al \circ \beta \).}
    \end{subfigure}
    \caption{Examples of shape concatenations.}
    \label{fig:concat}
\end{figure}

With the proper concatenation introduced, we now return to the motivation for
introducing this operation. For concreteness, consider two shapes \(\alpha\) and
\(\beta\) with compatible boundaries. The reason concatenation is central is that
it captures the leading contribution in the product of the corresponding graph
matrices. Namely, when we multiply \(\cm_\alpha\) and \(\cm_\beta\) using ordinary
matrix multiplication, one might hope in an idealized setting that
\[
    \cm_{\alpha}\cdot \cm_{\beta}
    =
    \cm_{\alpha\circ\beta}.
\]
This identity, however, is not exact. The product on the left generally produces
additional terms, known in the graph-matrix literature as \emph{intersection
terms}. 

The source of the discrepancy is the difference between global and local
injectivity constraints. The graph matrix \(\cm_{\alpha\circ\beta}\) is defined
using labelings that are injective across the entire concatenated shape
\(\alpha\circ\beta\). By contrast, the product \(\cm_\alpha \cdot \cm_\beta\) only
enforces injectivity separately within the copy of \(\alpha\) and within the copy
of \(\beta\). As a result, vertices from \(\alpha\) that are not identified
through the boundary may nevertheless collide with vertices from \(\beta\).

We now formalize this notion through intersection patterns, which record which vertices from the two shapes collide during the matrix product.
\begin{definition}[Intersection patterns/shapes]
Given two shapes $\tau_1$ and $\tau_2$ that are composable, we define an intersection pattern for $\tau_1$ and $\tau_2$ to be a non-empty set $I \subseteq \left(V(\tau_1) \setminus V_{\tau_1}\right) \times \left(V(\tau_2) \setminus U_{\tau_2}\right)$ of intersection edges where each vertex is incident to at most one edge. Whenever we have an intersection edge $e = \{a,b\}$ (between vertices of the same type), this means that when we consider maps $\sigma: V(\tau_1) \cup V(\tau_2) \to [n]\cup [d] $, we have the restriction that $\sigma(a) = \sigma(b)$.

We define $\tau_{I} = \tau_{I}(\tau_1\cdot \tau_2)$ to be the shape obtained from $\tau_1$ and $\tau_2$ as follows:
\begin{enumerate}
\item Start with the properly concatenated shape $\tau_1\circ \tau_2$;
\item For each intersection edge $e = \{a,b\}$, merge $a$ and $b$ together into one vertex.
\end{enumerate}
\end{definition}
To make our analysis more convenient, instead of merging the collided vertices, we will primarily view intersection patterns as shapes except that they additionally come with intersection edges that specify vertex collisions. See the blue dotted edge in \cref{fig:concat-intersection} for an example of an intersection edge, together with the resulting shape corresponding to this
intersection pattern.
\begin{definition}
We define $Int_{\alpha,\beta}$ to be the set of possible intersection patterns for $\alpha$ and $\beta$.
\end{definition}

\begin{definition}[$\circ$ and $\cdot$ notation]
	Given two shapes $\al_1$ and $\al_2$, we use $\circ$ to denote proper concatenation, and $\cdot $ to denote (possibly) improper concatenations. In other words, $\al_1\circ  \al_2$ denotes a single shape obtained by proper concatenation, while $\al_1\cdot \al_2$ denotes a collection of shapes that include both proper concatenations and intersection shapes.
\end{definition}
With the intersection terms introduced, we can deduce the following analog of ordinary matrix multiplication of graph matrices in the diagram language.
\begin{proposition}
For all shapes $\alpha$ and $\beta$ that can be concatenated, \[
{\cm}_{\alpha}{\cm}_{\beta}  = \cm_{\al\cdot \beta}   = {\cm}_{\alpha \circ \beta} + \sum_{I \in Int_{\alpha,\beta}}{{\cm_{\tau_I}}}\,.\]
\end{proposition}

\paragraph{Orthogonal Polynomials as Horizontal Concatenation}
We now introduce another important ingredient of our work: orthogonal polynomials.
The connection between graph matrices and orthogonal polynomials was first made
explicit in \cite{PotechinXu2026Theta}. In their setting, one considers the
orthogonal polynomials \(\{P_t(x)\}_{t\ge 0}\) with respect to the semicircle
distribution \(\mu_{\mathrm{sc}}\) on \([-2,2]\). They show that if \(Q\) is a
linear combination of graph matrices from a specific family called \emph{backbone-correction shapes}, then these individual matrices are all freely independent of each other, and \(P_t(Q)\) admits a clean decomposition via the \(t\)-fold proper concatenations of the shapes appearing in \(Q\) additionally provided that the matrix $Q$ is appropriately normalized.

Returning to our setting, the relevant shapes are no longer the
backbone-correction shapes studied in \cite{PotechinXu2026Theta}. Nevertheless,
it is natural to expect that the underlying phenomenon is not specific to this
particular family. Rather, one may hope that the same graph-matrix
interpretation of orthogonal polynomials applies more broadly, with the
essential condition being an appropriate form of free independence among the
underlying graph matrices.

In our setting, we call the shapes that arise in the analysis
\emph{backbone-dangling shapes}. At a high level, these shapes may be viewed as
generalizations of the graph matrices studied in \cite{HKPX23}; indeed, our
initialization \(Q_0\) is itself based on the identity-perturbation construction
considered therein. For intuition, we now give a few examples of backbone-dangling shapes that arise
in our analysis, deferring their formal definition and classification to later
sections.

 Analogous to \cite{PotechinXu2026Theta}, a key component of our
analysis is to show that the graph matrices associated with backbone-dangling
shapes remain freely independent. Crucially, the free independence of the underlying graph matrices ensures that the orthogonal-polynomial interpretation
via graph matrices continues to hold: applying an orthogonal
polynomial to a linear combination of these graph matrices is captured, up to negligible error terms, by graph matrices associated with concatenated shapes.

\begin{proposition}[(Informal) Equivalence of Chebyshev Polynomials and Concatenated Shapes]
For a linear combination of backbone-dangling shapes, $Q = \sum_{\psi \in \calB(Q) : \text{backbone-dangling} } c_\psi  \cdot \cm_\psi$, provided $\Var(Q) =1$,
 for any $j\geq 1$,  we have
\[
    P_j(Q)
    =
    \sum_{\substack{\tau:\text{ \(j\)-fold concatenation of }\calB(Q) \\ \tau =\psi_1\circ \psi_2\circ \ldots \circ \psi_j} }
    c_\tau \cm_\tau + o_d(1),
\]  
up to lower-order error terms of spectral norm $o_d(1)$ and the coefficients are given by \(
c_\tau = \prod_{t\in [j]} c_{\psi_t} \).

In words,   \(P_j(Q)\) is approximated by the
weighted sum of all \(j\)-fold proper concatenations of shapes from
\(\calB(Q)\), up to lower-order error terms.

\end{proposition}

With the graph-matrix framework in place, together with the correspondence between orthogonal polynomials and graph matrices of concatenated shapes, we now describe our analysis of the iterative construction and the technical challenges specific to the ellipsoid fitting problem. 

\subsection{Correction Primitive: Finding the Inverse \texorpdfstring{\(M^{-1}\)}{of M}}
At each step, the iterative procedure corrects the current affine deviation by applying \(M^{-1}\) to the deviation
vector \(\eta\), closely mirroring the identity-perturbation construction. The
difficulty is that, near the conjectured threshold, the matrix \(M\) can no
longer be analyzed as a small perturbation of the identity. Consequently, the
Neumann-series-based approach used in previous analyses is no longer sufficient,
and a more refined understanding of \(M^{-1}\) becomes necessary.

Let us first recall the Neumann-series-based approach, a.k.a. matrix Taylor expansion, used in the prior works. We closely follow the notation of \cite{HKPX23}, and use the decomposition 
\[ \numberthis \label{eq:M-decomposition}
    M = A+B,
\]
where
\[
    A
    :=
    \left(1+\frac1d\right)I_m
    +
    M_\alpha+M_\beta+M_D  
\]
and
\[
    B
    :=
    \frac1d J_m
    +
    \frac1d
    \left(
        \mathbf 1_m\eta^\top+\eta\mathbf 1_m^\top
    \right).
\]
Here \(J_m\) is the all-ones matrix, \(M_\alpha\) and \(M_\beta\) are
off-diagonal centered matrices, and \(M_D\) is diagonal and in fact negligible spectrally.  More explicitly, for
\(i\ne j \in [m] \),
\[
    \cm_\alpha[i,j]
    :=
    \sum_{a\ne b\in[d]}
        v_i[a]v_i[b]v_j[a]v_j[b] = 2\cdot \sum_{a<b\in [d]} v_i[a]v_i[b]v_j[a]v_j[b] \numberthis \label{eq:M_al_def}
\]
\[
    \cm_\beta[i,j]
    :=
    \sum_{a\in[d]}
    \left(v_i[a]^2-\frac1d\right)
    \left(v_j[a]^2-\frac1d\right), \numberthis \label{eq:M_beta_def}
\]
and
\[
    \cm_D[i,i]
    :=
    \|v_i\|_2^4-\frac2d\|v_i\|_2^2-1. \numberthis \label{eq:M_D_def}
\]
See \cref{fig:M_alpha} and \cref{fig:M_beta} for their diagrams. The main insight from the prior work is that $B$ is a rank-$2$ matrix, and one can use the Woodbury matrix identity to reduce the inverse of $M$ to that 
of $A$. 

For expository purposes, we advise the reader to ignore the negligible norm components in $A$ and focus on the first three terms \(
 \cm_\al +\cm_\beta + I\,.
\)
The omitted components do not affect the analysis in any essential way.
In particular,
in the regime with a sufficiently large constant factor gap, one may write
\[
     {A}^{-1}
    =
    \frac{1}{I+( {A}-I)}
    =
    \sum_{t\geq 0} (-1)^t ( {A}-I)^t,
\]
and the series converges precisely because \(\|A-I\|_{\mathrm{sp}}<1\).
However, this perturbative argument breaks down close to the threshold, even
when we remain a constant factor away from it.
Our key idea is that, although \( A\) is no longer close to the identity,
it remains close in spirit to a Wishart matrix. The only caveat is that the
underlying vectors are not genuinely i.i.d., because of their tensor-product
structure. Momentarily ignoring this distinction, if \(  A\) were an ideal
Wishart matrix with aspect ratio \(\gamma\), then its spectrum would satisfy
\[
    \mathsf{spec}(A)
    \subseteq
    \bigl[(1-\sqrt{\gamma})^2,\,(1+\sqrt{\gamma})^2\bigr] \pm o_d(1).
\]

We now identify the relevant aspect ratio $\gamma$ at the threshold.
Let us view \(A\) as the Gram matrix of the lifted rank-one vectors \(v_i v_i^\top\),
after removing the low-rank directions captured by \(B\). By symmetry, these
lifted vectors do not live in the full \(d^2\)-dimensional space but instead in a subspace of dimension
\(
   N_0= d(d+1)/2
\).
Thus, heuristically, \( A\) should behave like a Wishart matrix formed from
\(m\) samples in dimension \(N_0\approx d^2/2\). The corresponding Marchenko--Pastur aspect
ratio is therefore
\[
    \gamma
    =
    \frac{m}{N_0}
    =
    \frac{2m}{d^2}(1+o_d(1)).
\]
Throughout this work, we write simply \(\gamma=2m/d^2\),
suppressing lower-order terms in \(d\).


\paragraph{Inversion from Marchenko-Pastur.}

The above heuristic next leads us to the following polynomial basis for the
Marchenko--Pastur distribution with parameter \(\gamma\).
\begin{definition}[MP Polynomial Basis] \label{def:MP-basis}
Fix \(0<\gamma<1\). Let \(\{\mathfrak \q_t^{(\gamma)}\}_{t\ge0}\) be the polynomial sequence defined by
\[
    \q_0^{(\gamma)}(x)=1,
    \qquad
   \q_1^{(\gamma)}(x)=x-1,
\]
and, for \(t\ge1\),
\[
    \q_{t+1}^{(\gamma)}(x)
    =
    \bigl(x-(1+\gamma)\bigr)\q_t^{(\gamma)}(x)
    -
    \gamma \cdot \q_{t-1}^{(\gamma)}(x).
\]
When the dependence on \(\gamma\) is clear, we write \(\q_t=\q_t^{(\gamma)}\). 
\end{definition}

The following generating-function identity gives the desired expansion of
\(x^{-1}\) in this basis. We defer its proof to the appendix.

\begin{claim}\label{clm:inverse-expansion}
For \(x \in [(1-\sqrt{\gamma})^2, (1+\sqrt{\gamma})^2] \) in the Marchenko--Pastur support and \(\gamma<1\),
\(
    \frac{1}{x}
    =
    \frac{1}{1-\gamma}
    \sum_{t\geq 0} (-1)^t \q_t(x).
\)
\end{claim}
%

Formally, this gives us the inverse expansion
\[
    A^{-1}
    =
    \frac{1}{1-\gamma}
    \sum_{t\geq 0} (-1)^t \cdot  \q_t(A) \,. \numberthis \label{eq:poly-expansion-A}
\]
The remaining issue is that we have yet to rigorously justify that the spectrum of $A$ is indeed a Marchenko-Pastur distribution with parameter $\gamma$, and more importantly, up to the edge of the spectrum as well  
since
 this polynomial expansion is meaningful only on the
spectral interval where the Marchenko--Pastur approximation is valid. 

It should be noted that a similar result can also be deduced from the recent work of \cite{kogan2025extremal}, which establishes a more general theorem for kernel random matrices. Thus, our result may be viewed as a rediscovery of their theorem in the special case considered here. At a technical level, both proofs are based on the trace moment method, 
though our proof is different from the usual direct expansion of the expected value of the trace of high powers of the matrix.

Instead, our proof for the spectrum of $A$ closely follows our intuitive connection between orthogonal polynomials and graph matrices of concatenated shapes. En route, we generalize the previous equivalence about Chebyshev polynomials of the second kind to the corresponding orthogonal polynomials for Marchenko--Pastur, and show that an analogous equivalence with graph matrix of concatenated shapes holds.
\begin{restatable}[MP Polynomials and Shape Concatenation]{lemma}{mpPolynomialShapeConcatenation}
\label{lem:mp-polynomial-shape-concatenation}
Consider
\(
    A=\cm_\alpha+\cm_\beta+\left(1+\frac1d\right)I_m+M_D 
\) from the decomposition of $M$ in~\cref{eq:M-decomposition}.
For any   \(t \), we have
\[
    \q_t(A)
    \coloneqq
    \q_t^{(\gamma)}(A)
    =
    \fp_t(A)+o_d(1),
\]
where
\[
    \fp_t(A)
    \coloneqq
    \sum_{\substack{
        \tau=\tau_1\circ\cdots\circ\tau_t\\
        \tau_i\in\mathcal B(A)\ \forall i\in[t]
    }}
    \cm_\tau .
\]
Here \(\q_t^{(\gamma)}\), formally defined in~\cref{def:MP-orthogonal-poly}, denotes the \(t\)-th orthogonal polynomial associated
with the Marchenko--Pastur distribution of parameter \(\gamma\) , and
\(\fp_t(A)\) is the corresponding linear combination of graph matrices of
\(t\)-fold properly concatenated shapes. 

Crucially, the concatenation sum is
taken over the nontrivial graph-matrix shapes in \(\mathcal B(A) = \{\al,\beta \} \), and does not
include the trivial identity shape.
\end{restatable}

Diagrammatically, \(\q_t(A)\) is given by the sum over all \(t\)-fold proper
concatenations whose constituent shapes are \(\alpha\) or \(\beta\). We recall
these two basic shapes below.
\begin{figure}[ht!]
  \centering
      \begin{subfigure}[b]{0.38\textwidth}
        \includegraphics[width=\textwidth]{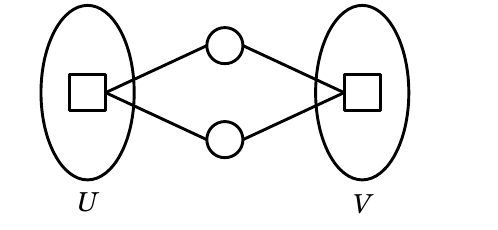}
        \caption{\(\cm_{\alpha}\).}
    \end{subfigure}
    \quad
    \begin{subfigure}[b]{0.32\textwidth}
        \includegraphics[width=\textwidth]{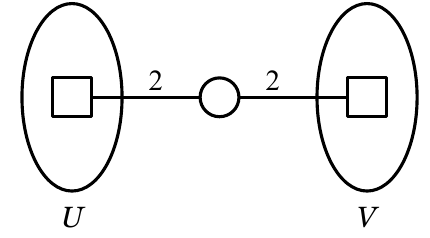}
        \caption{\(\cm_{\beta}\).}
    \end{subfigure}
    \captionsetup{width=.9\linewidth}
\end{figure}

In particular, the graph-matrix and
orthogonal-polynomial equivalence viewpoint allows us to cleanly decouple the main term from the deviation terms in the trace power calculation, and thereby deduce the spectrum of $A$.

\paragraph{Highlight of the Recurrence.}
The following surprising cancellation via graph matrices lies at the heart of our equivalence lemma between orthogonal polynomials and concatenated shapes. We highlight this equality at the basic level of $t=2$. Throughout this section, we will ignore the $o_d(1)$ term while our subsequent discussion shall also shed light on the source of such error terms. Recall that by the MP-recurrence from~\cref{def:MP-orthogonal-poly}, we have \[ 
\q_{2}(A) = (A - (1+\gamma) \cdot I_m )\cdot  \q_1(A) - \gam \cdot \q_0(A).
\]
Our goal is to show \[ 
\q_2(A) = \fp_2(A) +o_d(1)
\]
where we emphasize that the RHS is the sum of the graph matrices of all four $2$-fold concatenations of $\{\al,\beta\}$, namely $\{\al\circ \al, \al\circ \beta, \beta\circ \al, \beta\circ\beta \}$, and the $o_d(1)$ subsumes negligible components involving $1/d\cdot  I_m$ and $M_D$. We ignore such negligible components throughout the overview.
  
 Recall that we have $A = \cm_\al +\cm_\beta + I_m +o_d(1) $, $\q_0(A) = I_m $ and $\q_1(A) = \fp_1(A) = \cm_\al+\cm_\beta +o_d(1)$, so the LHS simplifies to \[ 
\q_2(A) =  (\cm_\al +\cm_\beta -\gam \cdot I_m) \cdot (\cm_\al +\cm_\beta) - \gam \cdot I_m  \numberthis  \label{eq:q2(A)} \,. 
\]
We now highlight the following key matrix equalities of graph matrices,\begin{enumerate}
	\item \[  \cm_\al \cdot \cm_\al = \cm _{\al \circ \al} +\gam \cdot I_m  +\gam \cdot \cm_\al +o_d(1)  \label{eq:al-al} \numberthis \]
	\item \[\cm_\beta \cdot \cm_\beta =  \cm _{\beta\circ\beta} + \gam \cdot \cm_\beta +o_d(1) \numberthis  \label{eq:beta-beta} \]
	\item \[ \cm_\al \cdot \cm_\beta  =  \cm_{\al\circ \beta} +o_d(1), \qquad \text{and} \qquad  \cm _{\beta}\cdot \cm_\al = \cm_{\beta\circ\al}+o_d(1)\,.  \]
\end{enumerate}
%
 
 The desired equality \[
\q_2(A) =   \cm_{\al\circ \al} + \cm_{\al\circ\beta} +\cm_{\beta\circ\beta} + \cm_{\beta\circ\al}
 \]
 is then immediate given these three equations.
 Therefore, verifying the desired equivalence lemma boils down to showing \cref{eq:al-al} and~\cref{eq:beta-beta}, especially regarding the $\gam$-factors therein. In words, in the multiplication of $\cm_\al\cdot \cm_\al$, and of $\cm_\beta \cdot \cm_\beta$, in addition to the proper concatenation terms that are part of $\fp_2(A)$, there is an additional factor of \(
 \gam (\cm_\al +\cm_\beta+I)
 \) from the intersection terms!
 
To illustrate this, instead of first expanding the algebra of the matrix product, we begin with diagrams of the intersection patterns that drive the two identities. In the diagrams below, each red dotted edge denotes an intersection edge, meaning that the two incident vertices receive the same label. In all five cases, each green edge appears exactly twice: it is of type \(h_1\) in (a)--(d), and of type \(h_2\) in \cref{fig:half_flat}. In \cref{fig:full_diamond_1} and \cref{fig:full_diamond_2}, the same is also true additionally for the purple edges.
  \begin{figure}[H]
    \centering

    \begin{subfigure}[b]{0.4\textwidth}
        \includegraphics[width=\textwidth]{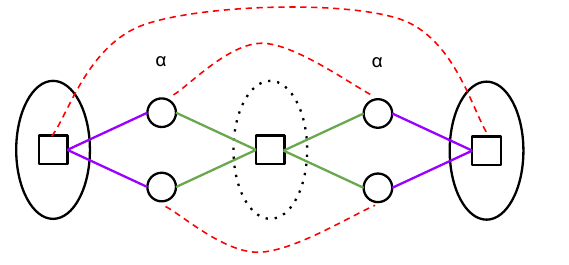}
        \caption{Diamond Backtracking-1}
        \label{fig:full_diamond_1}
    \end{subfigure}
    \quad
    \begin{subfigure}[b]{0.4\textwidth}
        \includegraphics[width=0.94\textwidth]{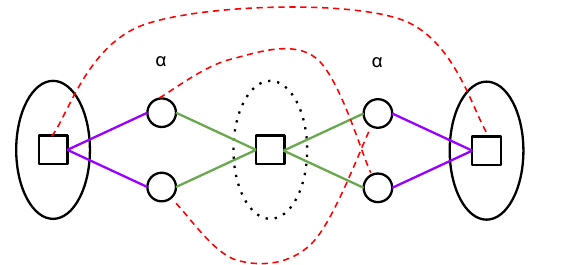}
        \caption{Diamond Backtracking-2}
        \label{fig:full_diamond_2}
    \end{subfigure}

    \vspace{1em}

    \begin{subfigure}[b]{0.4\textwidth}
        \includegraphics[width=0.94\textwidth]{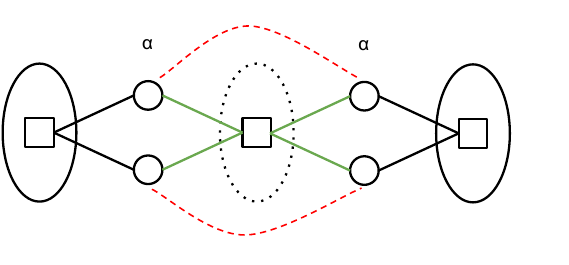}
        \caption{Half-Diamond Backtracking-1}
         \label{fig:half_diamond_1}

    \end{subfigure}
     \begin{subfigure}[b]{0.4\textwidth}
        \includegraphics[width=0.94\textwidth]{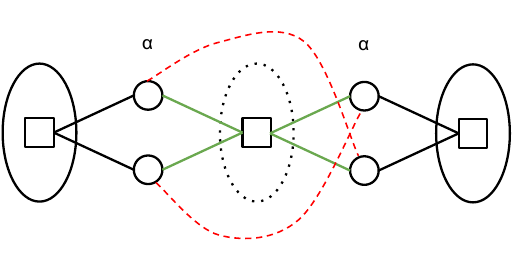}
        \caption{Half-Diamond Backtracking-2}
         \label{fig:half_diamond_2}

    \end{subfigure}
    
	    \vspace{1em}
%
        \begin{subfigure}[b]{0.48\textwidth}
        \includegraphics[width=0.94\textwidth]{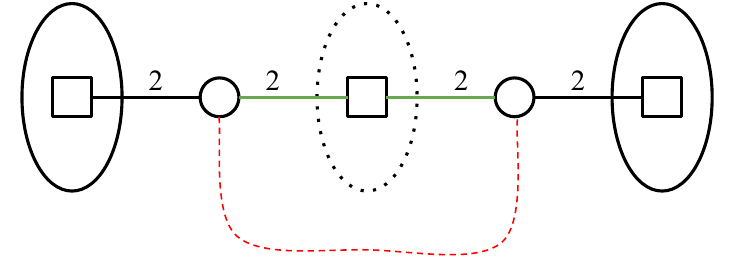}
        \caption{Half-Flat Backtracking}
         \label{fig:half_flat}

            \end{subfigure}
\label{fig:backtracking-int}
\caption{Well-behaved Backtracking Intersections}
\end{figure}
Without giving the formal definition of each intersection as it may be transparent from the diagram illustrations alone already, we note that we have three ``important'' intersection patterns that we call \emph{backtracking intersections}. Specifically, they are classified with the specific $m\times m$ matrix defined in the following.
 \begin{enumerate}
	\item \textbf{Diamond Backtracking}:
	 both of the two inner circle vertices appear twice, as well as the square vertex on the boundary. In this case, we have a diagonal matrix
%
obtained by considering
\[
    (\cm_\alpha^2)[i,i]
    =
    \sum_{k\neq i}\cm_\alpha[i,k]\cm_\alpha[k,i],
\]
and restricting to the summands in which the two circle labels match as an
unordered pair.

More explicitly, define
\[
\begin{aligned}
\cm_{\mathsf{DiamondInt}}[i,i]
&:=
\sum_{k\neq i}
\sum_{a\neq b\in[d]}
v_i[a]v_i[b]v_k[a]v_k[b]
\sum_{\substack{a'\neq b'\in[d]\\ \{a',b'\}=\{a,b\}}}
v_k[a']v_k[b']v_i[a']v_i[b']
\\
&=
\sum_{a\neq b\in[d]}
v_i[a]^2v_i[b]^2
\left(
    \sum_{k\neq i}2v_k[a]^2v_k[b]^2
\right).
\end{aligned}
\]
 where the factor of \(2\) comes from the two possible orderings for $a',b'$ to match with given $a,b$ as a set, corresponding to the two full-diamond intersection patterns illustrated in \cref{fig:full_diamond_1} and \cref{fig:full_diamond_2}.
	 \footnote{This can be equivalently viewed as coming from $ \E \left(2\sum_{a<b} \sum_{k\notin \{i,j\}} v_i[a]\cdot v_j[a]\cdot v_i[b]\cdot v_j[b] \right)^2 = 4 \binom{d}{2}m/d^4.$ }
	 
	 Notice this is a scalar with mean \[ 
	 d(d-1) \cdot 2(m-1) \cdot \frac{1}{d^4} = 2m/d^2 = \gam 
	 \]
	  where we use that each Gaussian random variable has variance $\frac{1}{d}$ in our setup. This allows us to show that up to $o_d(1)$ error terms, we have \(
	  \cm_{\mathsf{DiamondInt}} = \gam \cdot  I_m \,.
	  \)
\item \textbf{Half-Diamond Backtracking}: 
In the half-diamond backtracking pattern, the two inner circle vertices are
identified across the two copies, while the left and right boundary square
vertices remain distinct. Equivalently, we consider the product
\[
    (\cm_\alpha^2)[i,j]
    =
    \sum_{k\notin\{i,j\}} \cm_\alpha[i,k]\cm_\alpha[k,j],
\]
and restrict to entries of $i\neq j\in[m]$ with the summands in which
\(
    \{s,t\}=\{s',t'\}.
\)
Expanding gives
\[
\begin{aligned}
\cm_{\mathsf{HalfDiamondInt}}[i,j]
&\coloneqq
\sum_{k\notin\{i,j\}}
\sum_{s\neq t\in[d]}
v_i[s]v_i[t]v_k[s]v_k[t]
\sum_{\substack{s'\neq t'\in[d]\\ \{s',t'\}=\{s,t\}}}
v_k[s']v_k[t']v_j[s']v_j[t']
\\
&=
		 \underbrace{ \sum_{s\neq t\in[d]} v_i[s]v_i[t] v_j[s]v_j[t] }_{=\cm_\al[i,j]}  \cdot  \underbrace{(\sum_{k\notin \{i,j\}} 2 v_k^2[s]v_k^2[t])}_{=  \gam} \end{aligned}
\]
%

One key observation is that the  second term is tightly concentrated around its mean independent of the choice of $s\neq t$,\[
\E \sum_{k\notin \{i,j\}} 2 v_k^2[s]v_k^2[t]  = 2 \cdot (m-2) \cdot  (\frac{1}{d})^2 = \frac{2m}{d^2 } =\gam 
 \]
 when we consider any fixed $s\neq t\in [d]$. Therefore, with some algebra after centering with respect to the mean of the second term, we show that up to some error term with $o_d(1)$ spectral norm, we have\[
 \cm_{\mathsf{HalfDiamondInt}}  = \gam \cdot \cm_\al \,.
 \]

\item \textbf{Half-Flat Backtracking}: the inner circle vertex of the two
\(M_\beta\)-gadgets is identified across the two copies, while the left and
right boundary square vertices remain distinct. Equivalently, we consider
\[
    (\cm_\beta^2)[i,j]
    =
    \sum_{k\notin\{i,j\}}\cm_\beta[i,k]\cm_\beta[k,j],
\]
and restrict to the summands in which the two circle labels agree.

Recall that
\[
    \cm_\beta[i,j]
    =
    \sum_{a\in[d]} h_2(v_i[a])h_2(v_j[a]).
\]
Thus the half-flat contribution has entries
\[
\begin{aligned}
\cm_{\mathsf{HalfFlatInt}}[i,j]
&:=
\sum_{k\notin\{i,j\}}
\sum_{a\in[d]}
h_2(v_i[a])h_2(v_k[a])
h_2(v_k[a])h_2(v_j[a])
\\
&=
\underbrace{ \sum_{a\in[d]}
h_2(v_i[a])h_2(v_j[a])}_{=\cm_\beta[i,j]}
\underbrace{ \left(
    \sum_{k\notin\{i,j\}} h_2(v_k[a])^2
\right)}_{= \gam} .
\end{aligned}
 \]
 Analogously to the previous two terms, the second term is tightly concentrated around
 \[ 
 (m-2) \cdot \frac{2}{d^2}
 \]
 where the factor $2$ instead comes from $\E[h_2(v_k[a])^2] = \frac{2}{d^2}$.
  \end{enumerate}

 Subsequently in~\cref{sec:MP-sec}, we give a formal analysis of the error terms that we have ignored so far, and generalize this cancellation pattern to higher degrees beyond $\q_2$ as well to establish the equivalence between orthogonal polynomials and graph matrices of concatenated shapes. 
 
 Furthermore, the equivalence allows us to deduce the spectrum 
 of $A$ up to the edge.

\begin{restatable}[Spectral radius of \(A\)]{lemma}{spectralradiusA} \label{lem:spectral-radius-A} 
With high probability, \[ \mathsf{spec}(A) \subseteq \bigl[(1-\sqrt{\gamma})^2,\,(1+\sqrt{\gamma})^2\bigr] \pm o_d(1), \] where \(\gamma = \frac{2m}{d^2}. \) \end{restatable}

We emphasize that our purpose in revisiting the graph-matrix proof is
not merely to reprove a result that can also be obtained by more direct moment
calculations. Rather, the matrix \(A\) provides a simpler setting in which to
illustrate the main ideas that will later be needed for the analysis of the
inner matrix \(Q\) arising from the iterative construction, where no prior result
seems to apply in a black-box manner.
%
 

Finally, recall from~\cref{eq:M-decomposition} that $M$ can be viewed as a rank-$2$ update on $A$; existing results from the prior works allow us to switch from the inverse of $A$ to that of $M$ via the Woodbury identity.

%
\begin{restatable}[Explicit Inverse of \(M\)]{lemma}{Minverse} \label{lem:M-inverse} Let \[ r\coloneqq \frac{\1_m^T A^{-1}\1_m}{d}, \qquad s\coloneqq 1+\frac{\eta^T A^{-1}\1_m}{d}, \qquad u\coloneqq -1+\frac{\eta^T A^{-1}\eta}{d}. \] Then \begin{equation} \label{eq:M-inverse} M^{-1} = A^{-1} + \underbrace{ \frac{1}{s^2-ru}\, A^{-1} \left( u\,\frac{\1_m\1_m^T}{d} - s\,\frac{\eta\1_m^T+\1_m\eta^T}{d} + r\,\frac{ \eta\eta^T}{d} \right) A^{-1}}_{M_W \coloneqq}. \end{equation} \end{restatable}

Combining this with our polynomial expansion for $A$, we have now obtained an explicit formula for $M^{-1}$. More importantly, it admits a clean decomposition in the graph matrices, allowing us to ultimately write our inner matrix $Q$ via graph matrices.
\jnote{i moved basically all the technical discussion of Woodbury to the later section, but leave here the expression for M inverse. This looks fairly clean to me now.}

\subsection{The Rise of Backbone-Dangling Shapes: Iterations Made Concrete}
We now zoom in on the combinatorial structure of the shapes that arise in the iteratively constructed matrix \(Q\). We begin by recalling the graph-matrix decomposition used in prior work, which naturally leads to a distinguished family of graph matrices, called \emph{backbone-dangling shapes}. We follow the naming from prior work, while substantially generalizing the definition to accommodate the shapes generated by our iterative construction.

\paragraph{Recap of Identity Perturbation: Decomposition at the Base Level}
We start by recapping the decomposition into graph matrices as presented in the previous work that exploits graph-matrix decomposition \cite{PTVW22, HKPX23}. By and large, we follow the notation of \cite{HKPX23} most closely, while also emphasizing the distinctions from the prior work that appear in our decomposition.

Let's recall that \[ 
Q_0 = -C_F \cdot  \mathcal{L}^*(M^{-1}\eta)= - C_F \cdot \sum_{t\in [m]} w_0[t] \cdot v_i \cdot v_i^\top
\]
with \( w_0 = M^{-1} \cdot \eta  \) and \(\eta = \eta_0 =   \calL(I_d) -\1_m\) denoting the affine deviation of the identity matrix.

The focus of \cite{PTVW22, HKPX23} boils down to studying the following collection of shapes called \emph{dangling shapes}. These are the matrices in the decomposition of the base matrix $Q_0$ that appear in the off-diagonal. The expansion also contains nonzero diagonal terms; however, their contribution is negligible, so we suppress them throughout the discussion. These matrices can be viewed as arising from a $3$-way concatenation of a backbone-path, a dangling $A^{-1}$-path, and a final $h_2$ deviation attachment:
    \begin{enumerate}
        \item \textbf{Backbone Path:} the off-diagonal of component $\sum_{t\in [m]} v_tv_t^\top $ is represented by the shape shown in \cref{fig:R-base} which can be viewed as a backbone path; this shape serves as
the base gadget for the other matrices appearing in the inner matrix $Q$.
	\item \textbf{Dangling $A^{-1}$ Path} 
    For each term, we specify a length-$k\geq 0$ dangling path that starts from the middle square vertex such that
    \begin{itemize}
        \item for any $k>0$, each step comes from one of the following gadgets in $\{ M_\al, M_\beta, M_D, \frac{1}{d}I_m\}$;
        \item The dangling path is a (possibly improper) concatenation of shapes along the path.    \end{itemize}
    \item \textbf{Final $h_2$ Deviation Attachment:} Since $w_0 = M^{-1} \eta$, each $A^{-1}$ term is hit with an additional $\eta$ factor; we therefore attach an $h_2$ gadget (\cref{fig:gadgets}) to the end of the dangling path.
    We call this the ``final $h_2$ attachment gadget''.

    \end{enumerate}
See \cref{fig:R-A} for an example. It is important to point out that we make two simplifications here: 
\begin{enumerate}
	\item These are the restrictions of $Q_0$ to proper concatenations ignoring possible intersections from the $3$-way vertical concatenation. Controlling all such intersection patterns is a nontrivial issue and necessitates the lengthy charging arguments developed in prior work. For the moment, we set this complication aside. 
    We will return to it when choosing the correction terms.
	\item Moreover, we use the proxy of $M^{-1}\approx A^{-1}$ by ignoring the extra rank-$2$ component. This is less of a concern, as it can be straightforwardly shown to be lower-order terms, or to give a constant multiple of another $A^{-1}$ term. \anote{This sentence is a bit confusing} \jnote{ok to drop - i wanted to say the rank-2 component for $Q_0$ is not by itself negligible but it gives an $A^{-1}$ term as well.}
\end{enumerate}

  \jnote{add pointer}.

There is, however, one aspect of this analysis that simplifies substantially in our setting. In the prior analyses of \cite{PTVW22,HKPX23}, the expansion of \(A^{-1}\) is expressed through ordinary powers of \(A-I\). Consequently, the dangling path representing \(A^{-1}\) may contain numerous collisions among the gadgets placed along the path. By contrast, our expansion in orthogonal polynomials naturally produces proper, and hence injective, concatenations of shapes. This eliminates most of the collision patterns that arise from the inverse expansion itself.

\begin{figure}[ht]
    \centering
    \begin{subfigure}[b]{0.28\textwidth}
        \includegraphics[width=\textwidth]{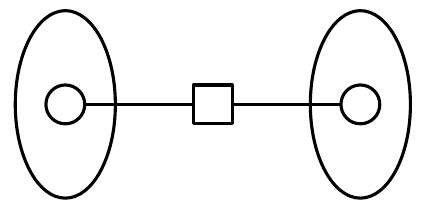}
        \caption{Off-diagonal part of $\sum_{i=1}^m v_i v_i^T$.}
        \label{fig:R-base}
    \end{subfigure}
    \quad
    \begin{subfigure}[b]{0.32\textwidth}
        \includegraphics[width=0.94\textwidth]{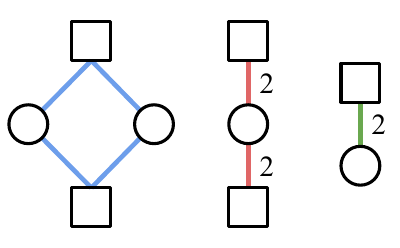}
        \caption{Left: $M_{\alpha}$ gadget. Middle: $M_{\beta}$ gadget. Right: final $h_2$ gadget.}
        \label{fig:gadgets}
    \end{subfigure}
    \quad
    \begin{subfigure}[b]{0.28\textwidth}
        \includegraphics[width=0.94\textwidth]{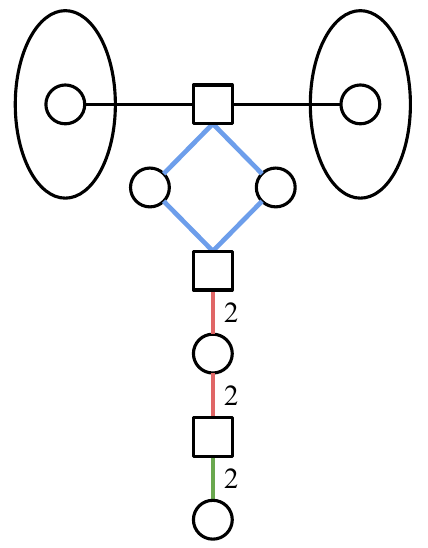}
        \caption{Off-diagonal part of $\sum_{i=1}^m (M_\al M_\beta \eta)[i] \cdot v_i v_i^T$.}
        \label{fig:R-A}
    \end{subfigure}
    \captionsetup{width=.9\linewidth}
    \caption{Examples of graph matrices in the decomposition of $Q$ adapted from \cite{HKPX23}.
    Recall that square vertices take labels in $[m]$ and circle vertices take labels in $[d]$.
    Unlabeled edges have Hermite index $1$.}
    \label{fig:examples}
\end{figure}

\paragraph{Dangling Path Once Again: Vertical Concatenation} 

 We now turn to the matrices that arise from the correction terms in the
iterative construction. For concreteness, we describe the first iteration; the
same decomposition applies to later iterations essentially verbatim. Again we focus on terms that arise from proper concatenation while deferring the discussion on intersection terms.

Consider the graph matrix decomposition of $H_{0}$ restricted to the properly concatenated terms that are on the off-diagonal, 
\[ 
\mathsf{Proper} H_{0} = \sum_{\substack{\tau \in \calB(H_{0})\\ \text{proper off-diagonal}}} c(\tau) \cdot \cm_\tau \,,
\]
and consider a fixed term $\tau \in \calB(H_{0})$. We ignore the diagonal terms as they are negligible. Modulo its coefficient $c(\tau)$, its assigned correction term in the iterative process is given by \[ 
\corr(\tau) \coloneqq \frac{1}{1-\gamma} \left(- \calL^*M^{-1}\calL(\cm_\tau)  + \gam \cdot \cm_\tau  \right)\,.
\]

Throughout this section, we focus solely on the first term of $  \calL^*M^{-1}\calL(\cm_\tau)$, and most importantly, the proper shapes therein: the remaining term of $\gam \cdot H_{0}$ is picked precisely to cancel with certain intersection terms that also arise from the sum of rank-$1$ terms. We defer the discussion to~\cref{sec:cooking-correction}.

\paragraph{Upgrade at the Terminal Attachment}
 We call such an operation \emph{vertical concatenation} from a graph matrix point of view: it is again a $3$-way concatenation of the \emph{backbone-path}, \emph{dangling $A^{-1}$-path} and \emph{deviation attachment} from the prior works except the final deviation attachment is no longer an $h_2$-gadget. Instead, the
terminal square vertex may be attached to a newly generated shape
\(\tau \in H_{0} \) and in general from the most recently generated higher-degree term $H_i$ in the Chebyshev expansion, with the two endpoints of \(\tau\)
connected to the same square vertex via two $h_1$ edges. We call the final attachment a \emph{deviation attachment}, and call the two edges connecting boundary vertices of $\tau$ to the square vertex \emph{attachment edges}. Correspondingly, we also call the terminal square vertex a \emph{deviation attachment vertex}.
 
 We showcase the process of vertical concatenation from a given shape $\tau \in H_{0}$ from $(a)$ to $(c)$ in~\cref{fig:vertical-concate}.

\begin{figure}[h]
    \centering
    \begin{subfigure}[b]{0.4\textwidth}
        \includegraphics[width=\textwidth]{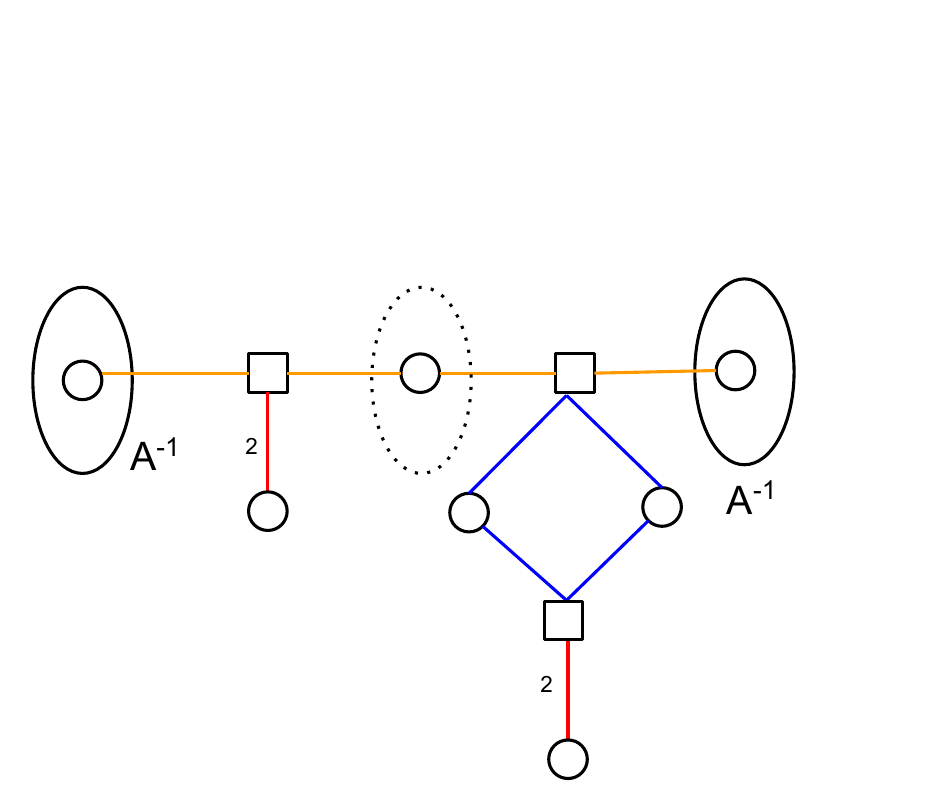}
        \caption{Shape $\tau \in P_2(Q_0)$}
        \label{fig:p2(Q)}
    \end{subfigure}
    \quad
    \begin{subfigure}[b]{0.4\textwidth}
        \includegraphics[width=0.94\textwidth]{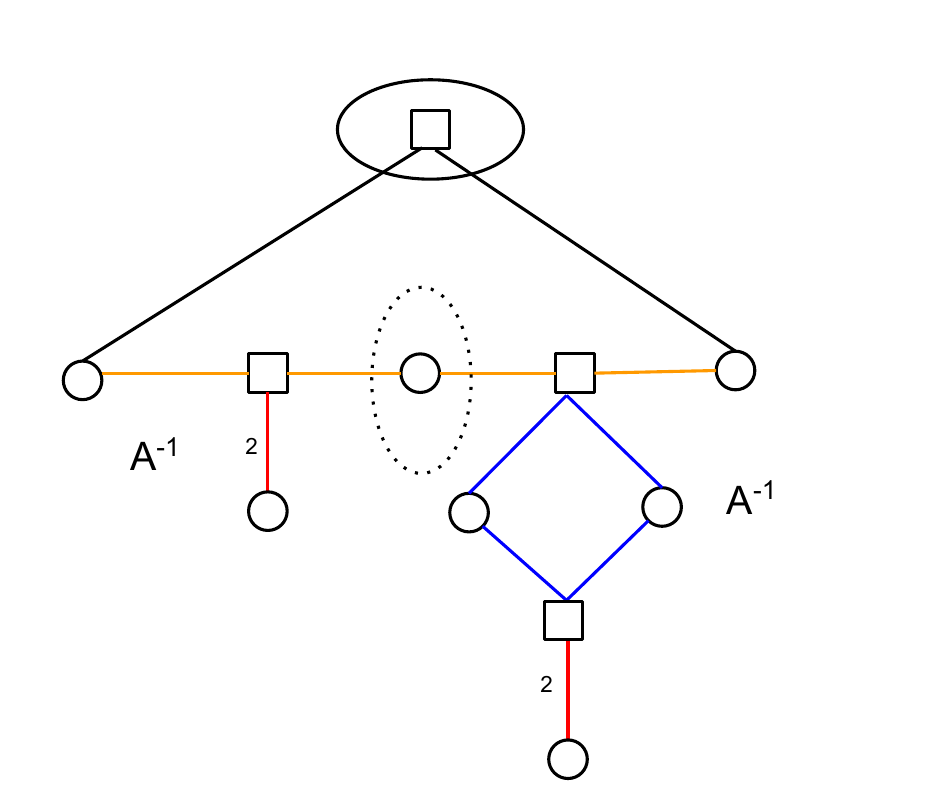}
        \caption{Deviation  w.r.t. $ \cm_\tau $}
        \label{fig:eta(p2(Q))}
    \end{subfigure}

    \vspace{1em}

    \begin{subfigure}[b]{0.48\textwidth}
        \includegraphics[width=0.94\textwidth]{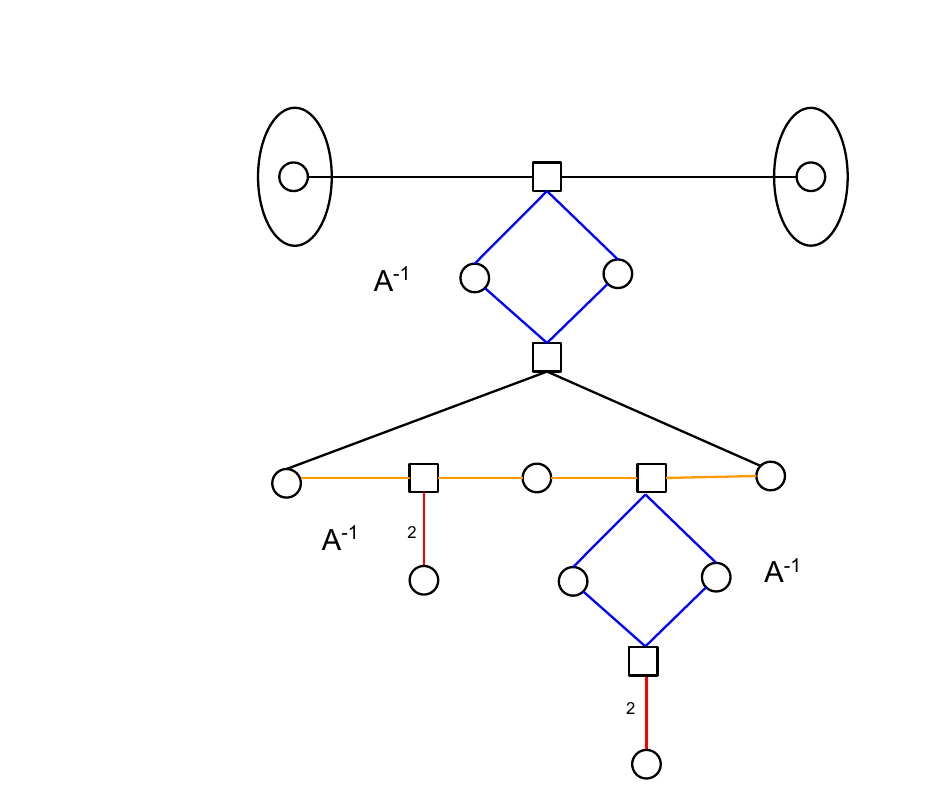}
        \caption{Backbone-dangling shape: $\corr(\tau) \in Q_1\setminus Q_0$}
        \label{fig:corr(tau)}
    \end{subfigure}
    \hfill
    \begin{subfigure}[b]{0.48\textwidth}
        \includegraphics[width=0.94\textwidth]{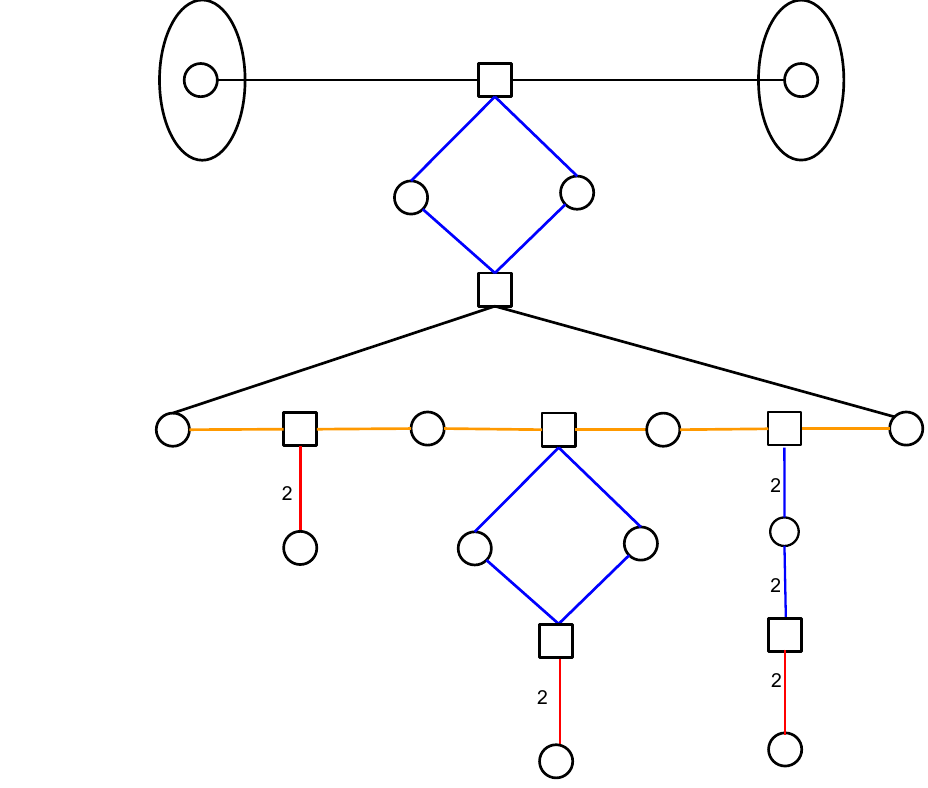}
        \caption{Backbone-dangling shape in $ Q_1\setminus Q_0$}
        \label{fig:corr(tau)-2}
      
    \end{subfigure}
	\caption{Vertical Concatenation and Backbone-Dangling Shapes}
	\label{fig:vertical-concate}
%
\end{figure}

Finally, as we apply such construction iteratively across the iteration levels, this gives rise to a recursive terminal attachment process defined as follows.
\begin{definition}[Recursive Terminal Attachment]
Let \(s\) be a square vertex. A recursive terminal attachment at \(s\) is
generated by the following rules.

\begin{enumerate}
    \item \textbf{Base case: \(h_2\)-attachment.}
    The attachment may be a single \(h_2\)-edge from \(s\) to a circle vertex
    \(a\):
    \[
        s \xleftrightarrow{h_2} a .
    \]

    \item \textbf{Concatenated-path attachment.}
    For any \(t > 1\), let \(\tau_1,\ldots,\tau_t\) be backbone-dangling
    shapes with compatible circle boundaries, and let    \[
        \tau =\tau_t \circ \tau_{t-1}\circ \cdots \circ \tau_1
    \]
    denote their proper concatenation. $\tau$ is then attached by adding two $h_1$ edges connecting $s$ to $U_\tau$ and $V_\tau$ (assuming $U_\tau\neq V_\tau$).
%
\end{enumerate}
\end{definition}

Notice we also single out the off-diagonal shapes appearing in the subsequent levels in \(Q\), since these are
the main objects of our analysis---the diagonal terms will be treated separately
as small-norm error terms and bounded at the end.

 Finally, we wrap this section up with the formal definition of \emph{backbone-dangling} shapes as follows: in summary, 
 these are the shapes that appear in the recursive application of the terminal-attachment process if the entire resulting shape is vertex injective.
\begin{definition}[Backbone-Dangling Shape]\label{def:backbone-dangling-shape}
A shape is called a backbone-dangling shape if it satisfies the following
properties.

\begin{enumerate}
    \item It contains an off-diagonal backbone path from a circle vertex
    \(U_\tau\) to another circle vertex \(V_\tau\), passing through exactly \emph{one} square
    vertex via two \(h_1\)-edges:
    \[
        U_\tau \xleftrightarrow{h_1} s \xleftrightarrow{h_1} V_\tau .
    \]

    \item Attached to this square vertex \(s\) is a dangling path arising from
    the \(A^{-1}\)-expansion.

    \item At the end of the dangling path, the terminal square vertex is
    equipped with a recursive terminal attachment in the sense of the preceding
    definition;
    \item It is a \emph{proper} shape, i.e., there is no vertex intersection.
\end{enumerate}
\end{definition}

Intuitively, the shapes we study can be viewed as generalizing the backbone shapes in~\cite{HKPX23} both vertically and horizontally. 

\subsection{Semicircular Spectrum of the Inner Matrix \texorpdfstring{$Q$}{Q}}
With a combinatorial characterization of the summands in $Q$, we have two final questions to settle in order to establish that the spectrum of $Q$ follows a semicircle distribution in $[-2,2]$. In turn, this will ultimately ensure that our choice of Chebyshev polynomials of the second kind for the polynomial expansion is valid. Concretely, we show the following:
  \begin{enumerate}
	\item \textbf{Semicircular Spectrum}: each constituent matrix in $Q$ is freely independent of each other;
	\item \textbf{Variance Normalization}: the variance of the inner matrix is $1$ (in the limit).
\end{enumerate}

Towards identifying the spectral property of the inner matrix, we set out to understand the combinatorial property of each shape arising in the expansion of $Q$.

\paragraph{Why are Backbone-Dangling Shapes Freely Independent?}

With the combinatorial description of backbone-dangling shapes in place, we now
explain why one should expect their associated graph matrices to exhibit
free-independence. Recall that each backbone-dangling shape consists of
two components: a backbone path, coming from the matrix
\(\sum_{t\in[m]} v_tv_t^\top\), and a dangling path attached to a square vertex,
coming from the correction term \(M^{-1}\eta\).

The free-independence heuristic is governed by two largely separate
considerations, corresponding to these two components. Consider the trace moment
calculation for \(Q\), and in particular the dominant contributions to
\(
    \mathbb E\operatorname{Tr}\bigl((QQ^\top)^q\bigr).
\)
A walk contributing to this trace can itself be viewed as having a backbone
component and a dangling component. At a high level, we would like to view these two
components separately.

\paragraph{The Backbone Component}
 Restricting attention to the backbone path, which comes from the matrix
    \(
        \sum_{t\in[m]} v_tv_t^\top,
    \)
    the dominant trace contributions require each labeled square vertex, whose
    label lies in \([m]\), to appear exactly twice, provided $m\gg d$. This should be contrasted
    with circle vertices, whose labels lie in \([d]\), and which do not obey the
    same combinatorial constraint.

    The reason is the sharp difference between the label costs of square and
    circle vertices. A square vertex contributes a factor of \(m\approx d^2\),
    whereas a circle vertex contributes only a factor of \(d\). Thus,
    repetitions of square vertices are much more wasteful in the combinatorial counting of the trace walks.
    In the dominant walks, one therefore wants to avoid any unnecessary
    repetition of square vertices. The twice-appearance condition is the minimal
    requirement forced by the mean-zero random variables appearing along the
    backbone path.

    For readers familiar with the graph matrix norm bound language, this can be equivalently viewed by contrasting the norm bound value with the square vertex being the separator as opposed to the circle vertex, which accounts for a gap of $\tilde{O}(\sqrt{d})$ factor in the norm bound.
    
    \paragraph{The Dangling Component} Once we observe that each square vertex makes exactly two appearances, we can further deduce in the dominant term that each square vertex must be paired with the same dangling component (in both summands from $Q$, and its dangling edge-set). Therefore, this allows us to further bundle steps of the trace walk that correspond to the same underlying shape (and labeled edge-set),  \[ 
    \Tr[(Q\cdot Q^\top )^{q}] =  \Tr \sum_{\tau_1,\tau_2,\dots,\tau_{2q}} \cm_{\tau_1} \cdot \cm_{\tau_2} \cdot \dots \cdot \cm_{\tau_{2q}}\,.
    \]
    At this point, analogous to usual trace moment calculations for Wigner matrices,  it is well-anticipated that the partition for the dominant term should proceed in a non-crossing manner, giving rise to the semicircle shape of the resulting spectrum of $Q$. 
    
    We give a formal verification of this strategy via the equivalence with orthogonal polynomials in the later section. It culminates in establishing the following lemma.
    \begin{restatable}[Orthogonal Polynomials and Shape Concatenation for Semicircular Distributions]{lemma}{semicircleShapeConcatenation}
\label{lem:semicircle-shape-concatenation}
Consider a linear combination of backbone-dangling shapes 
\[
    K=\sum_{\tau\in\mathcal B(K)} c_\tau \cdot \cm_\tau
\]
from a ground set $\calB(K)$.
Then, for any $t>0 $,
\[
    P^{\mathsf{sc}}_t(K)
    =
    \sum_{\substack{
        \tau=\tau_1\circ\cdots\circ\tau_t\\
        \tau_i\in\mathcal B(K)\ \forall i\in[t]
    }}
    \left(\prod_{i=1}^{t} c_{\tau_i}\right)\cm_\tau
    + o_d(1),
\]
where the sum is over proper \(t\)-wise concatenations of shapes from
\(\mathcal B(K)\), and the error is in spectral norm, provided \[ 
\Var(K) = 1 \,.
\]

\end{restatable}

 \paragraph{Controlling the Variance of \texorpdfstring{$Q$}{Q}}   
 With the shape of the spectrum settled, it still remains for us to identify the radius of the semicircle distribution, which boils down to the following matrix variance quantity that we consider, \[ 
 \Var(Q) = \frac{1}{d}\cdot  \E[\Tr(QQ^\top)] = \sum_{\tau \in \calB(Q)} \Var(c_\tau \cdot \cm_\tau) +o_d(1)\,.  \] 
 where the second equality follows from free independence of backbone-dangling shapes. For our ansatz construction to apply, it is crucial that $Q$ has variance $1$. We now illustrate how we keep track of the variance across the iterative process.

 
By design of our iterative construction, we observe that once a shape $\tau$ is added to $Q$ at some level of $Q_i$, its coefficient stays invariant for the subsequent iterations by~\cref{prop:no-reappearance-old-shapes}. Hence, it suffices for us to partition $\calB(Q)$ by the level in which each shape gets added to the inner matrix.

\paragraph{Variance for Base Level $Q_0$}

 Let's start from our base case.  
 \begin{restatable}[Decomposition of $Q_0$ via Graph Matrix and its Variance]{lemma}{qzerographvariance}
\label{lem:q0-graph-variance} Recall that $Q_0 \coloneqq -C_F \cdot \calR$,
  	we have 
  	\[ Q_0 = \sum_{\tau \in\calB(Q_0)} c(\tau) \cdot \cm_\tau + o_d(1) \]
  	where $\calB(Q_0)$ is a collection of (proper) \emph{backbone-dangling} shapes with variance
  	\[ 
  \Var(Q_0) = \sum_{\tau\in\calB(Q_0)} \Var(c_{\tau} \cdot \cm_{\tau}) +o_d(1)= (C_F)^2 \cdot \frac{\gam }{1-\gam } +o_d(1)\,.
  \]
\end{restatable} 
%
Our formal proof in the subsequent section proceeds via a graph matrix analysis, and it involves certain intersections that warrant extra care. The direct proof via graph matrix language additionally verifies the structural property that each non-negligible term is a backbone-dangling shape. That said, to illustrate the variance bound, we give a separate argument that is considerably more intuitive, without having to appeal to graph matrix language in depth.
%
%

Ignoring the extra scaling, it suffices for us to show $\Var(\calR)=\frac{\gamma}{1-\gamma} $. Observe that the trace quantity can be rearranged as \begin{align*}
 	\Tr(\calR \cdot \calR^T) = \sum_{i,j \in [m]} w[i] \cdot w[j] \cdot  \langle v_i, v_j\rangle^2 = w^T M w
 	&= (M^{-1} \eta)^T \cdot M \cdot  (M^{-1}) \eta   \\
 	&= \eta^T \cdot  M^{-1} \cdot  \eta 
 \end{align*}
Finally, it can then be verified by the following scalar calculation whose proof is deferred to~\cref{sec:def-scalar}. \begin{claim}[Base Variance] \label{clm:base-scalar-variance} 
	\[ 
	 \frac{1}{d} \E[\eta^\top M^{-1} \cdot \eta] = \frac{\gamma}{1-\gamma} \,.
	\]
\end{claim}

 \paragraph{Variance Evolution} 
Next, we derive a scalar recursion for the variance across iterations. The key point is that the graph-matrix and orthogonal-polynomial structure allows us to track variance shape-by-shape. In particular, it boils down to the following key identities. For any shape $\tau$ that appears in one of the ``new'' terms from the prior level, in~\cref{lem:vertical-concatenation}, we show that its corresponding correction term $\corr(\tau)$ to be added in the next level has variance \[ 
\Var(\corr(\tau)) =  \frac{\gamma}{1-\gamma} \Var(\tau) +o_d(1)\,. \]

We are now ready to deduce the scalar recurrence. Let $S_i \coloneqq \Var(Q_i)$ denote the variance at the $i$-th level, and let \(\textsf{New-}\mathcal P_j(Q_i)\) denote the collection of newly added terms at level-$i$ and degree-$j$, i.e., degree-$j$ terms in $H_i$ that are additionally multi-way proper concatenations from Chebyshev expansions.
%
%
%

\begin{align*}
S_{i+1} - S_i &= \sum_{\tau\in \calQ_{i+1}\setminus \calQ_{i}} c_\tau^2  \cdot \Var(\tau) \tag{definition of \(S_i\)} \\&= 
 \sum_{j = 2}^{\infty}{b_j^2\sum_{\beta = \tau_1 \circ \cdots \circ \tau_j \in \textsf{New-}P_j(Q_i)}}  c_\beta^2  \cdot \Var(\corr(\beta)) \tag{each term is a correction} \\ 
 &=  \sum_{j = 2}^{\infty}{b_j^2\sum_{\beta = \tau_1 \circ \cdots \circ \tau_j \in \textsf{New-}P_j(Q_i)}}  c_\beta^2   \cdot \frac{\gamma}{1-\gam } \cdot \Var(\beta)  \tag{\cref{lem:vertical-concatenation}}\\ 
&=\frac{\gam}{1-\gam } \cdot  \sum_{j = 2}^{\infty}{b_j^2\sum_{\beta = \tau_1 \circ \cdots \circ \tau_j \in \textsf{New-}P_j(Q_i)} \prod_{t=1}^j c_{\tau_t}^2} \cdot \Var(\tau_t) \tag{via free independence }\\
&= \frac{\gam }{1-\gam } \cdot  \sum_{j = 2}^{\infty}{b_j^2(S_i^j - S_{i-1}^j)} 
\end{align*}
where we take $S_0 =  C_F^2$  and $S_{-1} = 0$. In the last line, we used that all \(j\)-fold concatenations from \(\mathcal B(Q_i)\) contribute \((S_i)^j\), while those using only shapes from \(\mathcal B(Q_{i-1})\) contribute \((S_{i-1})^j\). Their difference is therefore exactly the contribution from \(\textsf{New-}\mathcal P_j(Q_i)\).

 Summing this expression from $i = 0$ to $k$ gives that 
\[
S_{k+1} = \frac{\gam}{1-\gamma}( S_0 +  \sum_{j=2}^{\infty}{{b_j^2} \cdot S_k^j}) \,.
\]
Letting $S_{\infty} = \lim_{k \to \infty}{S_k}$ (assuming this limit exists), and plugging in $\gam= \frac{1}{2}$ at the ideal threshold,  we have that 
\[S_{\infty} = C_F^2 + \sum_{j=2}^{\infty}{{b_j^2}S_{\infty}^j}\,.\] Since $\sum_{j=2}^{\infty}{b_j^2} = 1 - C_F^2$, $S_{\infty} = 1$ is a solution of this equation, yielding the following lemma.
\begin{lemma}[Ideal Variance of the Inner Matrix $Q$ at the Threshold]
	Let $Q \coloneqq \lim_{i\rightarrow \infty} Q_i$. We have 
	\[ 
	\Var(Q) = 1\,.
	\]
\end{lemma} 
This ultimately gives us the desired semicircle of radius $1$ for the spectrum of $Q$ in the limit ignoring various truncation procedures. 

\subsection{Finding the Correction Term} \label{sec:cooking-correction}
We now elaborate on how we arrive at the counter-intuitive correction term per~\cref{eq:primal-corr}. For simplicity, let's consider a fixed term $\tau$ in the graph expansion of $H_i$ for the most recent level-$i$, and its correction term is given by \[ 
\corr(\tau) \coloneqq \frac{1}{1-\gamma} (\calL^*M^{-1}\calL(-\cm_\tau) -\gam \cdot \cm_\tau) 
\]
for any high-degree term $\tau$ that arises in the Chebyshev expansion $H(Q_i)$. We assume throughout that $\tau$ is a properly concatenated shape as otherwise it would be a negligible term.  In particular, we justify why we do not opt for the more natural correction term via \[
\corr_{natural}(\tau) \coloneqq \calL^*M^{-1} \calL (-\cm_\tau)\,.
\]
It should be noted that we arrive at the particular choice of correction by starting with the more natural option of $\corr_{alt}(\tau)$. In fact, there is sufficient heuristic justification for such a term being the ``right'' correction primitive beyond being the most natural one. Let us elaborate upon this.

\paragraph{A ``Buggy'' Heuristic for the Natural Choice of $\corr_{natural}(\tau)$:}
 From the previous calculation for the variance of the inner matrix $Q$, we would like to show the added correction term (regardless of its concrete options) $\bar{\tau}$ has variance \[ 
\Var(\bar{\tau} ) = \frac{\gamma}{1-\gamma} \Var(\tau)
\]
\snote{use of the word ludicrous} \jnote{i just removed it}
with the key parameter of $\frac{\gamma}{1-\gamma}$ popping up that allows us to deduce a semicircular matrix of variance-$1$ in the limit. From this perspective, there is indeed a reason one may expect $\corr_{alt}$ to be the correct object to consider.
Let $\eta_\tau \coloneqq \calL(\eta_\tau)$ be the deviation incurred by $\tau$. By straightforward moment calculation for the vector, one can verify that \[ 
  \E\|\eta_\tau\|_2^2 = \gam \cdot \Var(\tau)\,.
\]
Next, recall from the calculation for $\Var(\calR)$ that we can conveniently rewrite the variance of the correction matrix as \[
\Var(\calL^*M^{-1}\eta_\tau) = \frac{1}{d} \cdot \E[\eta_\tau^\top \cdot M^{(-1)} \cdot \eta_\tau ] \,.
 \]
At this point, suppose we ignore the correlation of underlying randomness between $\eta_\tau$ and $M^{-1}$; we may then heuristically view $\eta_\tau$ as a genuinely random vector with respect to $M^{-1}$, and we would expect \[ 
\E[\eta_\tau^\top \cdot M^{-1} \cdot \eta_\tau ] \approx_{\text{heuristic}} \E[\|\eta_\tau\|_2^2] \cdot \frac{1}{m} \E[\Tr M^{-1}]
\]
assuming $\eta_\tau$ is in isotropic position relative to the eigenbasis of $M^{-1}$. 
 \jnote{i guess its the same but which one is more natural to put here? $M$ or $M^{-1}$?...}  Since we have shown $M$ follows the spectrum of Marchenko-Pastur with parameter $\gamma$ (even up to the edge), we have \( 
\frac{1}{m} \E \Tr M^{-1}  = \frac{1}{1-\gamma}\,.
\)
Combining the above, we would indeed arrive at \[ 
\Var(\corr_{natural}(\tau) ) = \frac{\gam}{1-\gamma} \Var(\tau)\,
\]
as desired! 

That said, as suggested by the paragraph title, this heuristic is flawed because of the correlated randomness.
\paragraph{Graph Matrices Strike Back: Removal of Non-Free Terms}
To address this, we appeal to graph matrices, which also in turn give us an intuitive illustration of how the correlation of underlying randomness kicks in. One intuitive way to see the issue is that there are non-trivial intersection terms that arise in the expansion of $\corr_{alt}(\tau)$ in the graph matrix basis - it is no longer merely proper concatenations that matter.

Informally, in~\cref{lem:proper-decomp-deviation-attach}, we show that \[ 
\corr_{natural}(\tau) = \calL^*M^{-1}\calL(\cm_\tau) = \gam\cdot \cm_\tau + \text{desired properly concatenated shapes} 
\]
up to negligible errors. The presence of $\cm_\tau$ in turn flags another concern for us: we would like the inner matrix $Q$ to consist of ``free'' matrices with semicircular spectrum. And it is clear that $\cm_\tau$ cannot be such a matrix - it arises from $P_j(G)$ for some $j>1$ and some putatively free matrix $G$. One intuitive example would be to consider proper terms in $\fp_2(\calR)$ as showcased in~\cref{fig:p2(Q)}.

At this point, we arrive at $\corr(\tau)$ as it is obtained by precisely removing the troublesome non-free term while being rescaled to have the desired variance at the same time! 
Most importantly, $\corr(\tau)$ achieves the same job in terms of satisfying the affine constraints.

\section{Refutation for Ellipsoid Fitting}
\label{sec:refutation}
In this section, we show how we can extend the iterative construction for ellipsoid fitting discussed above to give a refutation, i.e., the impossibility of ellipsoid fitting, when the number of constraints exceeds the threshold $m \ge d^2/4$.

\subsection{Overview of Refutation}
Recall that we let
\[
\mathcal L(X)_i= v_i^\top Xv_i ,
\qquad
\mathcal S:=\operatorname{Im}(\mathcal L^*)
=\operatorname{span}\{v_iv_i^\top:i\in[m]\},
\]
and let \(\calR\in\mathcal S\) be the identity perturbation satisfying
\(
\mathcal L(I-\calR)=\mathbf 1.
\)
Let \[ \Pi \coloneqq \Pi_S  = \calL^*M^{-1}\calL \] denote the orthogonal projection to $\calS$, and $\Pi^\perp$ the projection to its orthogonal complement.

We establish our refutation result by giving a dual witness for the ellipsoid fitting SDP. To get started, we record the dual SDP as follows,
\begin{proposition}\label{def:refutation-sdp}
	The dual of ellipsoid fitting is given by \(
	\Lambda \succeq 0
	\)
	such that \[ 
	\Lambda \in \calS \coloneqq \text{span}\{v_iv_i^\top: i\in [m] \}, \qquad \langle \Lambda, I-R\rangle < 0
	\]
where we remind the reader that we define the matrix inner-product as $\langle A,B\rangle = \Tr(A\cdot B)$.
\end{proposition}
As a sanity check, notice that such a matrix serves as a refutation for the primal program, as follows.
\begin{proposition}
Such a solution $\Lambda$ refutes the existence of a fitting ellipsoid.	
\end{proposition}
\begin{proof}
	Suppose there exists a solution $X$ for the primal, i.e., $\calL(X)=\1_m$ and $X\succeq 0$ and write $\Lambda = \calL^*(y)$. Then we have \[
  \langle y, \1 \rangle = \langle y, \calL(X)\rangle = \langle \calL^*(y), X \rangle =   \langle \Lambda, X\rangle =\Tr \left( X^{1/2} \Lambda X^{1/2}\right) \geq 0
\]
 by the PSDness of $\Lambda$ and $X$. On the other hand, since $\calL(I-\calR)=\1_m$,  we have \[
\langle y, \1_m \rangle = \langle y, \calL(I-\calR)\rangle = \langle \calL^*(y), I-\calR\rangle <0 
 \]
 yielding a contradiction.
\end{proof}

Next, we show the SDP dual can be reduced to seeking the following target matrix.
\begin{definition}[Target Dual Witness for Refutation] \label{def: target-refutation-witness}
	We call the following matrix our desired witness for refutation.
	 \[ 
 \Lambda = I_d+ c_R \cdot \calR - \Pi^\perp \cdot I_d  + Z = (I_d-   \Pi^\perp I)+ c_R \cdot \calR + Z  \succeq 0
 \] 
for some scalar $c_R > \frac{1-\gamma}{\gamma}$ and some matrix $Z \in \calS$ such that \( |\langle \calR, Z \rangle |= o(d)\) and $\Tr(Z) = o(d) $.
\end{definition}

To see why this is sufficient, and particularly how we arrive at the choice of $c_R >\frac{1-\gamma}{\gamma}$, we make the following observations:\begin{enumerate} 
	\item \jnote{check what condition on $\Pi^\perp I_d$ is most convenient} $\|\Pi^{\perp} I_d \|_{Frob} = \tilde{O}(1) = O(\frac{1}{\sqrt{d}}) $  since $I$ is almost entirely in the range of the rank-1 forms by considering the all-$1$ combination once appropriately scaled (concretely, $ d/m \cdot \1_m$). For intuition, this is a negligible term.
	\item $\Tr(\calR) = o(d)$;
 	\item  $\langle \calR, \Pi^\perp I \rangle = 0   $ since $\calR\in \cal S$.\end{enumerate}
%
This culminates in the following proposition.
\begin{proposition}
	The target dual matrix from \cref{def: target-refutation-witness} implies an SDP dual witness from~\cref{def:refutation-sdp}.
\end{proposition}

\begin{proof}
It remains only to verify that
\(\langle\Lambda,I_d-R\rangle<0\).
We have
\begin{align*}
    \langle\Lambda,I_d-\calR\rangle
    &=
    \left\langle
        I_d-\Pi^\perp I_d+c_R\cdot  \calR+W',\,
        I_d-R
    \right\rangle                                                    \\
    &=
    d-\|\Pi^\perp I_d\|_{\mathrm F}^2
      +(c_R-1)\Tr(\calR)
      +\Tr(W')
      -c_R \cdot \langle \calR,\calR\rangle
      -\langle W',\calR\rangle                                             \\
    &\le
    d-c_R \cdot \langle \calR,\calR\rangle+o(d)                                      \\
    &=
    d\left(
        1-c_R \cdot \frac{\gamma}{1-\gamma}+o_d(1)
    \right).
\end{align*}
Since
\[
    c_R>\frac{1-\gamma}{\gamma},
\]
the final expression is negative for all sufficiently large \(d\).
Thus
\(\langle\Lambda,I_d-\calR\rangle<0\), and \(\Lambda\) is a valid dual
witness.
\end{proof}

\jnote{below is old stuff that is no longer needed i believe...commenting out for now..}

\subsection{Dual Iterative Construction}
\snote{We will refer to our inner matrix for the dual refutation as $W$ to distinguish it from the matrix $Q$ in the primal program.}
To distinguish from our iterative process for the primal program, we will refer to our inner matrix for the dual refutation as $W$ throughout. Prompted by our desired target matrix, we start with \[ 
W_0 = C_F \cdot c_R \cdot \calR 
\]
for some constant $c_R$ to be chosen. We now apply the same spectral map $F$ to our inner matrix. 

\paragraph{Correction for Image Constraint} Throughout the iterative construction for the dual program, in contrast to satisfying the affine constraints as prescribed in the primal program, we would like to ensure instead $\frac{1}{C_F} F(W) \in \calS $---the image of the rank-$1$ forms. 

 Considering the first iteration $\frac{1}{C_F} F(W_0)$ and again its polynomial expansion, we have \[
\frac{1}{C_F} F(W_0) = I_d + c_R \cdot \calR + \frac{1}{C_F} \underbrace{ \sum_{j\ge 2} b_j\cdot  P_j(C_F \cdot c_R \cdot \cal R ) }_{H_0}\,.
 \]
 Since $I_d$ is in the desired image up to $o(d)$ error, we will view this as a negligible error term throughout, and correct such error alongside truncation error. Therefore, to ensure the resulting matrix is in the desired image, it suffices for us to correct for $H_0$. Again, it is by no means automatic that we should expect $H_0\in \calS$.
 
 Let's now consider how such correction might be achieved. By and large, we would like to find $\corr(H_0)$ such that \[
 H_0 + \corr(H_0) \in \calS\,.
 \] 
 Intuitively, one natural choice is to consider the orthogonal projection of $H_0$, and remove its orthogonal complement $(I_d - \Pi_\calS) H_0$. However, as we will discuss in ***, analogously to our correction for the primal program, the correlated randomness between $\Pi_S$ and $H_0$ produces a non-trivial fraction of $H_0$ in the above choice, posing a barrier to our ansatz solution as we would like to view the inner matrix as a sum of \textit{free} matrices. To address this issue, we alternatively consider \[ 
 \corr(H_0) \coloneqq \frac{1}{\gam }     \Pi_{\calS}   H_0- H_0 \,.
 \] 
It is straightforward to verify that this less natural choice also satisfies the image constraint as \[ 
H_0 + \corr(H_0) = \frac{1}{\gamma} \Pi_{\calS} H_0 \in \calS \,.
\]
Replacing the correction primitive in the primal iterative process gives us the following analogous process for the dual.

\begin{mdframed}[linewidth=0.2pt]
\textbf{Summary of the Initialization and Iterative Procedure for Refutation.}
\\\\
\textbf{Initialization:}
Set
\[
W_0\coloneqq C_Fc_R\calR,
\]
where \(c_R>0\) is later chosen so that the limiting inner matrix has
variance \(1\) and yields the desired negative correlation.
%
\\\\
\textbf{Iterative Update \(i\rightarrow i+1\):}
For \(i\ge 1\), define the newly created image-constraint deviation
\[
H_{i}
\coloneqq
\sum_{j\ge 2}b_j
\left(P_j(W_i)-P_j(W_{i-1})\right).
\]
Set
\[
\Delta_{i+1} = \corr(H_{i})
\coloneqq
\left(\frac{1}{\gamma} \Pi_{\calS}H_{i}  -   H_{i} \right) 
=
\frac{1}{\gamma} \left(\calL^*M^{-1}\calL H_{i}- \gamma  H_{i} \right),
\]
and update
\[
W_{i+1}\coloneqq W_i+\Delta_{i+1}.
\]
\\
\textbf{Output:}
For the limiting inner matrix \(Q\), take the dual witness
\[
\Lambda_R\coloneqq \frac{1}{C_F}F(W) - \Pi_{\calS}^\perp I_d.
\]
Note that $Z = \frac{1}{C_F}F(W) - I_d - c_R\calR$ for \cref{def: target-refutation-witness}.
\end{mdframed}

\begin{proposition}[Invariant for the Dual Iterative Process]
\label{prop:invariant-dual}
For every \(i\geq 0\),
\[
    \Pi_{\calS}^{\perp}
    \left(
        \frac{1}{C_F}F(W_i)-\Pi_{\calS}^{\perp}I_d
    \right)
    =
    \frac{1}{C_F}\Pi_{\calS}^{\perp}H_i.
\]
Equivalently,
\[
    \frac{1}{C_F}F(W_i)-I_d-\frac{1}{C_F}H_i
    \in \calS.
\]
Thus, at the end of iteration \(i\), the entire remaining violation of the
image constraint is incurred by the newly created higher-order term \(H_i\).
\end{proposition}

\begin{proof}
By definition of the correction,
\[
    H_i+\Delta_{i+1}
    =
    \frac{1}{\gamma}\Pi_{\calS}H_i
    \in\calS,
\]
and hence
\[
    \Pi_{\calS}^{\perp}\Delta_{i+1}
    =
    -\Pi_{\calS}^{\perp}H_i.
\]

For the base case, \(W_0=C_Fc_R\calR\in\calS\), and
\[
    F(W_0)=C_FI_d+W_0+H_0.
\]
Therefore,
\[
    \Pi_{\calS}^{\perp}
    \left(
        \frac{1}{C_F}F(W_0)-\Pi_{\calS}^{\perp}I_d
    \right)
    =
    \frac{1}{C_F}\Pi_{\calS}^{\perp}H_0.
\]

Now suppose the claim holds at level \(i\). Since
\[
    F(W_{i+1})-F(W_i)
    =
    \Delta_{i+1}+H_{i+1},
\]
we obtain
\begin{align*}
    \Pi_{\calS}^{\perp}
    \left(
        F(W_{i+1})-C_FI_d
    \right)
    &=
    \Pi_{\calS}^{\perp}H_i
    +\Pi_{\calS}^{\perp}\Delta_{i+1}
    +\Pi_{\calS}^{\perp}H_{i+1}\\
    &=
    \Pi_{\calS}^{\perp}H_{i+1}.
\end{align*}
Dividing by \(C_F\) completes the induction.
\end{proof}

\paragraph{Variance Calculation: Reverse Engineering for $c_R$} With the iterative process described above, we illustrate how to pick $c_R >0 $ as promised so that $c_R > \frac{1-\gamma}{\gamma}$ as in \cref{def: target-refutation-witness}. Recall that this is necessary to obtain a negative correlation with $\calR$.

To demonstrate how we pick $c_R$, let us first point to the similarity between our primal and dual corrections. In particular, for any $H_i$ from the $i$-th level of iteration, the primal correction is given by \[ 
\corr_{\mathsf{primal}}(H_i) =   \frac{1}{1-\gamma} \left(\calL^*M^{-1} \calL(-H_{i} ) + \gam \cdot H_{i}     \right) 
\]
recalling from~\cref{eq:primal-corr}. In the meanwhile, the dual update is given by \[ 
\corr_{\mathsf{dual}}(H_i) = \frac{1}{\gamma} \left(\calL^*M^{-1}\calL H_{i}- \gamma  H_{i} \right)  = -\frac{1-\gamma}{\gamma}  \corr_{\mathsf{primal}}(H_i)\,.  \numberthis \label{eq:primal-dual-correction} \]
Moreover, the identical calculation applies term-wise to each shape in the decomposition as well.

Therefore, following the variance calculation in the primal iteration, we obtain the following analogous variance recurrence for the dual program. To distinguish it from the primal variance, define $SW_i \coloneqq \Var(W_i)$ to denote the variance of the dual inner matrix $W_i$.

The same calculation from the primal applies to give us the update equation
\begin{align*}
SW_{i+1} - SW_i &= \sum_{\tau\in \calQ_{i+1}\setminus \calQ_{i}} c_\tau^2  \cdot \Var(\tau) \tag{definition of \(SW_i\)} \\&= 
 \sum_{j = 2}^{\infty}{b_j^2\sum_{\beta = \tau_1 \circ \cdots \circ \tau_j \in \textsf{New-}P_j(Q_i)}}  c_\beta^2  \cdot \Var(\corr_{\mathsf{dual}}(\beta)) \tag{each term is a correction} \\ 
 &= (\frac{1-\gamma}{\gamma})^2 \cdot  \sum_{j = 2}^{\infty}{b_j^2\sum_{\beta = \tau_1 \circ \cdots \circ \tau_j \in \textsf{New-}P_j(Q_i)}}  c_\beta^2  \cdot \Var(\corr_{\mathsf{primal}}(\beta)) \tag{\cref{eq:primal-dual-correction}}
 \\&=   (\frac{1-\gamma}{\gamma})^2 \cdot \sum_{j = 2}^{\infty}{b_j^2\sum_{\beta = \tau_1 \circ \cdots \circ \tau_j \in \textsf{New-}P_j(Q_i)}}  c_\beta^2   \cdot \frac{\gamma}{1-\gam } \cdot \Var(\beta)  \tag{\cref{lem:vertical-concatenation}}
 \\&= (\frac{1-\gamma}{\gamma})^2 \cdot \frac{\gamma}{1-\gam } \cdot \sum_{j = 2}^{\infty}b_j^2 \left( (SW_i)^j - (SW_{i-1})^j\right) 
\end{align*}
where the last line follows identically to that for the primal recurrence. Applying the telescoping sum again gives us \[ 
SW_{k+1} = \frac{\gam}{1-\gamma} SW_0 + \frac{1-\gamma}{\gamma}  \sum_{j=2}^{\infty}{{b_j^2} \cdot (SW_k)^j} \,.
\] 

For the variance to converge to $1$ in the limit, it is convenient to consider \(
B(s) \coloneqq \sum_{j\ge 2} b_j^2 s^j\,,
\) 
and the variance recurrence then can be rewritten in the form \(
SW_{k+1} =SW_0 + \frac{1-\gamma}{\gamma} B(SW_{k})\,.
\)
Since we additionally have $B(1) = 1-C_F^2$ by Parseval, and \(
\Var(SW_{0}) = C_F^2 \cdot c_R^2 \cdot \frac{\gamma}{1-\gam} \,,
\) requiring $1$ to be a fixed point gives \[ 
1= C_F^2 \cdot c_R^2 \cdot \frac{\gamma}{1-\gamma} + (1-C_F^2)\cdot \frac{1-\gam}{\gam}\,.
\] 
Therefore, we have \[ 
c_R^2 = \frac{1-\gamma}{\gamma} \cdot \frac{1}{C_F^2} \cdot \left(1- \frac{1-\gam}{\gamma} (1-C_F^2) \right)\,.
\]
Finally, we take $c_R$ to be the positive square root of the above. It remains for us to verify $c_R > \frac{1-\gamma}{\gamma}$.
\begin{claim} 
	For $c_R = \sqrt{\frac{1-\gamma}{\gamma} \cdot \frac{1}{C_F^2} \cdot \left(1- \frac{1-\gam}{\gamma} (1-C_F^2) \right)}$, we have $c_R >\frac{1-\gamma}{\gamma}$ provided $\gam >\frac{1}{2}$ which is exactly our desired threshold for refutation.
\end{claim}

\begin{proof} For convenience, write $\rho = \frac{1-\gam}{\gam}$.
	A direct calculation gives
\[
\begin{aligned}
c_R^2-\rho^2
&=
\frac{\rho\bigl(1-\rho(1-C_F^2)\bigr)}{C_F^2}
-\rho^2
=
\frac{\rho(1-\rho)}{C_F^2} >0,
\end{aligned}
\]
since $0<\rho<1$ for $\gam>\frac{1}{2}$.
\end{proof}

\begin{remark}
At the threshold \(\gamma=\frac12\), we have \(\rho=1\) and \(c_R=1\),
so the base term reduces to
\(
W_0=C_F \calR. 
\)	
\jnote{check the sign?}
\end{remark}

\section{Correction Primitive for Affine Constraint Deviations}
\label{sec:MP-sec}

In this section,  we prove that $A$ can indeed be inverted, and moreover, its inverse can be nicely decomposed in the graph matrix basis, via concatenated shapes of $\{\al, \beta \}$. Crucially, even though $A$ contains an identity term (without $1/d$ pre-factor), such a trivial shape does not appear as a ground-element for the concatenation.
\mpPolynomialShapeConcatenation*
In this section, we suggest that the reader momentarily ignore the negligible components in $A$, namely $M_D$ and $\frac{1}{d} I$, with our focus restricted to $\cm_\al + \cm_\beta + I_d $.
Once we establish tight control on the spectral radius of $A$, we then apply polynomial expansion via the orthogonal polynomials for $MP(\gamma)$ to understand $A^{-1}$. In particular, we show that $A^{-1}$ admits a ``nice'' decomposition in the graph matrix basis via concatenated shapes of $\al$ and $\beta$. Finally, as in existing analyses that are tight up to constant factors, we use the Woodbury identity to obtain the inverse for $M$ from $A^{-1}$.  This ultimately allows us to obtain again a nice decomposition of the inner matrix $Q$ in the graph matrix basis.

With the equivalence between concatenated shapes and the orthogonal polynomials, we show that the spectrum of $A$ continues to be MP-like, even at the spectral edge.
\spectralradiusA*
%
%
\subsection{MP Orthogonal Polynomial and Concatenated Shapes} For the reader's convenience, we recall the orthogonal polynomial basis
for the Marchenko--Pastur distribution, together with its recurrence.

\begin{definition}[Orthogonal Polynomial for Marchenko-Pastur] \label{def:MP-orthogonal-poly}
Fix \(0<\gamma<1\). Let
\(\{\mathfrak q_t^{(\gamma)}\}_{t\ge 0}\) be the polynomial sequence
defined by
\[
    \mathfrak q_0^{(\gamma)}(x)=1,
    \qquad
    \mathfrak q_1^{(\gamma)}(x)=x-1,
\]
and, for every \(t\ge 1\),
\[
    \mathfrak q_{t+1}^{(\gamma)}(x)
    =
    \bigl(x-(1+\gamma)\bigr)\mathfrak q_t^{(\gamma)}(x)
    -
    \gamma\,\mathfrak q_{t-1}^{(\gamma)}(x).
\]
\end{definition}

\begin{definition}[Concatenated Shapes for $A$]
	Define  $\calB(A) = \{ \al, \beta\}$; we let $\calP_t(A)$ denote the collection of shapes obtained by $t$-fold proper concatenations of shapes in $\calB(A)$, in other words, proper concatenations of $\al$ or $\beta$. For $t=0$, we define $\fp_0(A)=I_d$.
\end{definition}

With these definitions, we introduce our formal lemma for the equivalence between orthogonal polynomials and graph matrices for $A$.
\begin{lemma}[Quantitative version of~\Cref{lem:mp-polynomial-shape-concatenation}]
	\label{lem:mp-error-quant}
	\[ 
	\q_t(A) = \fp_t(A)  + \mathsf{Error}_t(A)\,.
	\]
	with \[ 
	\|\mathsf{Error}_t(A)\|_{sp} = \frac{\poly(t)}{\sqrt{d}} =  o_d(1)\,.
	\]
	provided $t \leq  d^\delta$ for some $\delta>0$ where we recall that $\fp_t(A)$ is the sum of the graph matrices in $\calP_t$, i.e., any $t$-fold concatenations of $\calB(A)$.
	
\end{lemma}

We begin by identifying the error incurred when passing to the graph-matrix perspective. 
Recall that $\calP_j(M)$ denotes the collection of proper $j$-way concatenations using base shapes from $\calB(A) = \{\cm_\al ,\cm_\beta \}$. Our goal now is to identify what shapes arise in the error term:  we will formally define their properties, and show they admit a small norm correction $o_d(1)$  and hence can be safely discarded.

Next, we identify the combinatorial properties of the error terms by induction. For the base case $t=1$ (since $t=0$ is trivial as we have identity on both sides), we have \begin{align*}
	LHS = A - I = \cm_\al + \cm_\beta + (1+\frac{1}{d}) I +M_D  - I = \cm_\al + \cm_\beta =  \sum_{\tau\in \calP_1(A)} \cm_\tau +o_d(1)\,.
\end{align*}
In other words, $\mathsf{Error}_1(A)= \frac{1}{d}I_d + M_D = o_d(1) $.

\paragraph{Highlight of the Cancellation, and Motivating the Error Term.}
As highlighted in the technical overview, the key point is that the backtracking
intersections produced by composing two \(\alpha\)-shapes (as well as two $\beta$-shapes) collapse after
summing over labels to exactly match the terms being subtracted in the three-term
recurrence for the MP polynomial basis.

Towards that end, we start by singling out the following three intersection shapes obtained from the multiplications involving $\cm_\al$ and $\cm_\beta$.
\begin{definition}[Half-Diamond Backtracking Intersection]
Let \(A\) and \(B\) be two \(M_\alpha\)-shapes. We say that the composition
\(A\cdot  B\) forms a \emph{half-diamond backtracking intersection} if
\begin{enumerate}
    \item \(U_A \not\equiv V_B\), where \(U_A\) is the left boundary square
    vertex of \(A\) and \(V_B\) is the right boundary square vertex of \(B\);

    \item the two circle vertices of \(A\) receive the same set of labels in $[d]$ as those in $B$.
\end{enumerate}
\end{definition}

Analogously, we also define a full-diamond backtracking intersection except now we have an extra vertex intersection between $U_A$ and $V_B$.
\begin{definition}[Full-Diamond Backtracking Intersection]
Let \(A\) and \(B\) be two \(M_\alpha\)-shapes. We say that the composition
\(A\cdot  B\) forms a \emph{full-diamond backtracking intersection} if
\begin{enumerate}
    \item \(U_A  \equiv V_B\), where \(U_A\) is the left boundary square
    vertex of \(A\) and \(V_B\) is the right boundary square vertex of \(B\);

    \item the two circle vertices of \(A\) receive the same set of labels in $[d]$ as those in $B$.
\end{enumerate}
\end{definition}

\begin{definition}[Half-Flat Backtracking Intersection]
	Let \(A\) and \(B\) be two \(M_\beta \)-shapes. We say that the composition
\(A\cdot B\) forms a \emph{half-flat backtracking intersection} if
\begin{enumerate}
    \item \(U_A \not\equiv V_B\), where \(U_A\) is the left boundary square
    vertex of \(A\) and \(V_B\) is the right boundary square vertex of \(B\);
    \item the circle vertex of \(A\) receives the same label in $[d]$ as that in $B$.
\end{enumerate}
\end{definition}

With these intersection shapes singled out, we first define the error term arising from the intended cancellation from the well-behaved intersections highlighted above. 

\begin{definition}[Backtracking residuals]
\label{def:backtracking-residuals}
Define
\begin{align*}
    \Delta_{\mathsf{HalfDiamondInt}}
    &\coloneqq
    \cm_{\mathsf{HalfDiamondInt}}
    -
    \gamma\cm_\alpha, \\
    \Delta_{\mathsf{HalfFlatInt}}
    &\coloneqq
    \cm_{\mathsf{HalfFlatInt}}
    -
    \gamma\cm_\beta, \\
    \Delta_{\mathsf{DiamondInt}}
    &\coloneqq
    \cm_{\mathsf{DiamondInt}}
    -
    \gamma I_m.
\end{align*}
%
\end{definition}
Next, we show that these deviations are indeed negligible. This is the core of the cancellation that enables the equivalence with orthogonal polynomials for the MP distribution.

\begin{proposition}[Half-Diamond Cancellation] \label{prop:half-diamond-cancellation}
We have
\[
    \cm_{\mathsf{HalfDiamondInt}}
        =
    \gamma \cdot \cm_\al  + o_d(1).
\]
and analogously,
\[
    \cm_{\mathsf{HalfFlatInt}}
        =
    \gamma \cdot \cm_\beta   + o_d(1).
\]
where $\gam = \frac{2m}{d^2}$.
\end{proposition}

\begin{proposition}[Diamond Cancellation] \label{prop:diamond-cancellation}
We have
\[
    \cm_{\mathsf{DiamondInt}}= \gam \cdot I_m+o_d(1).
\]
\end{proposition}

%

 We verify in~\cref{sec:MP-def-proof} that the deviation terms are of negligible norm. Together, these two identities show that the backtracking pieces generated by
shape concatenation reproduce exactly the correction terms in the
Marchenko--Pastur three-term recurrence. This is the cancellation that allows
the \(t\)-fold concatenation expansion to match \(P_t(A)\), rather than the
ordinary power \(A^t\). We defer the complete proof of \cref{lem:mp-error-quant} regarding the spectral norm bounds to the end of this section.

 Before that, we show that almost as a corollary to the equivalence lemma, we can also obtain control of the spectral edge of $A$ with ease. This is the focus of the subsequent section.

\subsection{Inverting the System of Affine Constraints}

In this section, we bound the spectral radius of $A$, which allows us to justify the polynomial expansion on the corresponding interval up to $o_d(1)$ fluctuation. Once the inverse of $A$ is justified, we appeal to existing results to obtain the inverse of $M$, following essentially the same route as in prior works, except that several scalar concentration bounds used therein require significant sharpening.
%
%
%
%

\paragraph{Inverse of $A$ and its Spectrum}
We now bound the spectrum of $A$, thereby validating our inversion via polynomial expansion in the corresponding orthogonal polynomials. Since the spectrum of $A$ is not symmetric around $0$, we pass through the centered matrix $Y$ as follows.
\begin{lemma}[Spectral Radius via Norm Bounds for the Centered Matrix]
Let
\[
    Y
    \coloneqq
    A-(1+\gamma)I_m
    =
    \cm_\alpha+\cm_\beta+M_D
    +
    \left(\frac1d-\gamma\right)I_m.
\]
Then, with high probability,
\[
    \|A-(1+\gamma)I\|_{\mathrm{sp}}
    =
    \|Y\|_{\mathrm{sp}} 
    \le
    \bigl(1+o_d(1)\bigr)2\sqrt{\gamma}.
\]
\end{lemma}

\begin{proof}
	
We prove the claim by the trace moment method. Fix
\(q=\polylog(d)\), to be chosen sufficiently large. It suffices to show that
\[
    \E\Tr(Y^{2q})
    \le
    \bigl(1+o_d(1)\bigr) \cdot  m \cdot (2\sqrt{\gamma})^{2q}.
\]

The main ingredient is that we should not expand \(Y^{2q}\) in the 
monomial basis. The cancellation coming from the negative diagonal term
\(-\gamma I\) is crucial. Instead, we expand \(Y^{2q}\) in the Marchenko–Pastur orthogonal-polynomial basis \(\{\q_t\}_{t\ge0}\) from~\cref{def:MP-basis}. We have \[
    \q_0(A)=I,
    \qquad
    \q_1(A)=A- I=\cm_\al + \cm_\beta + o_d(1),
\]
and, for \(t\ge1\),
\[
    (A - (1+\gamma) I)  \cdot  \q_t(A)
    =
    \q_{t+1}(A)+\gamma  \cdot \q_{t-1}(A)\,.\]

Write
\[
    Y^k
    =
    \sum_{t=0}^{k} c_t(k) \cdot \q_t(A)  = \sum_{t=0}^k c_t(k) \cdot \left(\fp_t(A)+ \mathsf{Error}_t\right).
\]
Ignoring for the moment the error from approximating $\q_t(A)$ by the $t$-wise proper concatenations in $\fp_t(A)$, the coefficients
satisfy the recurrence
\[
    c_t(k+1)
    =
    c_{t-1}(k)+\gamma c_{t+1}(k),
    \qquad t\ge1,
\]
with
initial condition
\[
    c_0(0)=1,
    c_t(0) =0 \,.
\]
for any $t>0$ as we have \[ 
Y^0 = I = \q_0(A) \,.
\]


We now derive the coefficient recurrence. Since \[ \q_0(A)=I, \qquad \q_1(A)=A-I, \] we have \[ Y\q_0(A) = \bigl(A-(1+\gamma)I\bigr)\q_0(A) = P_1(A)-\gamma \q_0(A). \]
 On the other hand, for \(t\ge1\), the MP recurrence gives us \[ Y\cdot \q_t(A)=\q_{t+1}(A)+\gamma \q_{t-1}(A).  \] 
 Therefore the coefficients satisfy \[ c_0(k+1)=\gamma c_1(k)-\gamma c_0(k), \] \[ c_1(k+1)=c_0(k)+\gamma c_2(k), \] and, for \(t\ge2\), \[ c_t(k+1)=c_{t-1}(k)+\gamma c_{t+1}(k). \]
  We will use the following coefficient bound and spectral norm for the error terms to wrap up the proof. 
 \begin{claim}[Coefficient bound in the MP basis]
\label{clm:MP-coef-bnd} 
  Assume \(0<\gamma\le1\). For every \(k\ge0\) and every \(t\ge0\), \[ |c_t(k)| \le \gamma^{-t/2}(2\sqrt{\gamma})^k. \] 
 In particular, for \(k=2q\), \[ |c_t(2q)| \le \gamma^{-t/2}(2\sqrt{\gamma})^{2q}. \] \end{claim} 
We prove this claim in~\cref{sec:MP-def-proof}. 
On the one hand, for the dominant term, we have  \[ 
\E\Tr c_0(2q) \cdot \q_0 = m \cdot c_0 \leq  m \cdot 2\sqrt{\gamma}^{2q}\,.
\]
On the other hand, for the deviation terms, we have 
\begin{align*}
	\E[ \Tr \sum_{t>1}^{2q} c_t(2q) \cdot (\fp_t(A) + \mathsf{Error}_t)] &\leq \sum_{t=1}^{2q} c_t(2q) \cdot m \cdot \|  \mathsf{Error}_t\|_{sp}\\
	&\leq  m\cdot  2q \cdot \max(c_t(t)\cdot  \frac{\poly(q)}{\sqrt{d}} \cdot \sqrt{\gamma}^{t-1} )\\
	&=m \cdot (2\sqrt{\gamma})^{2q} \cdot \frac{\poly(q,t)}{\sqrt{d}\cdot \sqrt{\gamma}}
\end{align*} 
where we plug in~\cref{lem:LCP-bound} and~\cref{clm:MP-coef-bnd} for the final bound. Since $q = \poly\log (d)$, we have \[ 
\E[\Tr Y^{2q}] \leq  O(1) \cdot   m \cdot 2\sqrt{\gamma}^{2q}\,.
\]
Taking the $1/2q$-th root then gives us the desired norm bound immediately. 
\end{proof}
%

\paragraph{Woodbury Identity: Moving Inverse from $A$ to $M$} We next fill in the details concerning the inverse of $M$. This is largely adapted from the corresponding section in \cite{HKPX23} except that we strengthen the concentration bounds for the scalars to remove extra $O(1)$ slack. Concretely, we establish the following lemma.
\Minverse*
\begin{proof}
  Recall that $M$ can be decomposed into $A$ and a rank-2 matrix $B$ (\cref{eq:M-decomposition}); the Woodbury identity then allows us to go from the inverse of $A$ to that of $M$. For completeness, we recall the following facts adapted from \cite{HKPX23}.
  
  \begin{fact}[Matrix Invertibility] \label{fact:invertibility}
    Suppose $A \in \R^{n_1 \times n_1}$ and $C \in \R^{n_2 \times n_2}$ are both invertible matrices, and $U\in \R^{n_1 \times n_2}$ and $V \in \R^{n_2 \times n_1}$ are arbitrary.
    Then, $A + U C V$ is invertible if and only if $C^{-1} + V A^{-1} U$ is invertible.
\end{fact}

\begin{fact}[Woodbury matrix identity~\cite{Woodbury1950}]  \label{fact:woodbury}
    Suppose $A \in \R^{n_1 \times n_1}$ and $C \in \R^{n_2 \times n_2}$ are both invertible matrices, and $U\in \R^{n_1 \times n_2}$ and $V \in \R^{n_2 \times n_1}$ are arbitrary. Then
    \[ 
    (A+ UCV)^{-1} = A^{-1} - A^{-1}U\left(C^{-1}+ VA^{-1}U\right)^{-1}VA^{-1}  \,.
    \]
\end{fact}

In light of \cref{fact:woodbury}, we can write $B$ as $B = UCU^T$ where $U = V^T = \frac{1}{\sqrt{d}} \begin{bmatrix} 1_m & \eta \end{bmatrix} \in \R^{m \times 2}$, 
$C = \begin{bmatrix} 1 & 1 \\ 1 & 0 \end{bmatrix}$,
and $M = A + UCU^T$.
Note that $C^{-1} = \begin{bmatrix} 0 & 1 \\ 1 & -1 \end{bmatrix}$, and we have
\begin{align*}
    C^{-1} + U^T A^{-1}U  = \begin{bmatrix}
	\frac{1_m^T A^{-1} 1_m}{d} & 1 + \frac{\eta^T A^{-1} 1_m}{d}\\\\
	1+ \frac{\eta^T A^{-1} 1_m}{d} & -1 + \frac{\eta^T A^{-1} \eta }{d} 
    \end{bmatrix}
    \eqqcolon
    \begin{bmatrix}
    	r & s \\
    	s & u
    \end{bmatrix}
    \,. 
    \numberthis \label{eq:russ}
\end{align*}

Assuming $A$ is invertible, from \cref{fact:invertibility}, we can prove that $M$ is invertible  by showing that the $2\times 2$ matrix $C^{-1} + U^T A^{-1}U$ is invertible, which is in fact equivalent to $ru - s^2 \neq 0$ from the above definitions.



\begin{restatable}[Scalar bounds]{lemma}{HyperParameterBounds}
\label{lem:hyper-parameter-bounds}
Let \(r,s,u\) be the scalars defined above. Then, with high probability,
\[
    r\coloneqq \frac{1_m^T A^{-1} 1_m}{d}
    =
    (1+o_d(1))\cdot \frac{m}{d}\cdot \frac{1}{1-\gamma},
\]
\[
    s \coloneqq 1+ \frac{\eta^T A^{-1} 1_m}{d}  = 1+o_d(1),
\]
and
\[
    u \coloneqq -1 + \frac{\eta^T A^{-1} \eta }{d}  = -1+\gamma+o_d(1).
\]
\end{restatable}

\begin{corollary}[$M$ is invertible] W.h.p. \[ 
ru-s^2 \neq 0\,.
\]	
\end{corollary}
Finally, expanding out Woodbury, we have \[ 
M^{-1} = A^{-1} + \frac{1}{s^2-ru}\cdot A^{-1} \cdot (u\cdot \frac{\1_m\1_m^T}{d} -s \frac{\eta\cdot \1_m^T+ \1_m^T \cdot \eta }{d}
 + r \cdot \frac{\eta \cdot \eta^T }{d}) \cdot A^{-1}\,.   \]
and this completes the proof of~\cref{lem:M-inverse}.
\end{proof}

\subsection{Formal Decomposition of MP Error Terms}

As highlighted in the technical overview, the key point is that certain
backtracking intersections produced by composing two \(\alpha\)-shapes, or
two \(\beta\)-shapes, collapse after summing over their internal labels.
Their leading contributions agree exactly with the lower-order terms
subtracted in the three-term recurrence for the Marchenko--Pastur
polynomials. The remaining terms will be organized into local-collision
pieces.

\paragraph{Extension to higher degree.}
To pass to higher degree, we build each gadget record from right to left.
A new error term is created at the first point where the newly attached
terminal gadget has an unintended vertex intersection with the previously
constructed record. We enforce local injectivity before this terminal
gadget, while allowing the terminal gadget to intersect either the current
locally proper prefix or an earlier block in the global record. The latter
possibility is needed for non-isolated backtracking configurations, in
which a backtracking pair also interacts with an earlier gadget.

We first distinguish the isolated backtracking pairs canceled by the MP
recurrence from the non-isolated intersections that remain as error terms.

\begin{definition}[Well-Behaved Backtracking Intersections]
\label{def:well-behaved-backtracking}
Let
\[
    \tau_t\cdot\tau_{t-1}\cdot\dots\cdot\tau_1,
    \qquad
    \tau_i\in\{\alpha,\beta, \frac{1}{d} I_m, M_D\},
\]
be a gadget record. We say that an adjacent pair
\(
    \tau_s\cdot\tau_{s-1}
\)
forms a \emph{well-behaved backtracking intersection} if
\begin{enumerate}
    \item \(\tau_s\) and \(\tau_{s-1}\) form one of the designated
    backtracking intersections: half-diamond, half-flat, or full-diamond;
    \item the vertices participating in this backtracking block have no
    additional vertex intersections with the earlier prefix
    \(
        \tau_{s-2}\cdot\dots\cdot\tau_1.
    \)
\end{enumerate}
If the first condition holds but the second fails, we call
\(\tau_s\cdot\tau_{s-1}\) a \emph{non-isolated backtracking pair}.
These are the three-way backtracking/diagonal intersections.
\end{definition}

\begin{definition}[Local-Collision Pieces for MP Polynomials]
\label{def:local-collision-pieces}
Fix a gadget record
\[
    \tau_t\cdot\tau_{t-1}\cdot\dots\cdot\tau_1,
    \qquad
    \tau_i\in\{\alpha,\beta, \frac{1}{d} I_m, M_D\}.
\]
A consecutive sub-record
\(
    R
    =
    \tau_b\cdot\tau_{b-1}\cdot\dots\cdot\tau_a
\)
is called a \emph{local-collision piece}, abbreviated \emph{LCP}, if it
satisfies the following conditions:
\begin{enumerate}
    \item Each constituent gadget is a gadget in $A$, i.e., from $\{\alpha,\beta, \frac{1}{d} I_m, M_D\}$.

    \item The concatenation preceding the terminal gadget is locally
    proper; namely,
    \(
        \tau_{b-1}\circ\tau_{b-2}\circ\cdots\circ\tau_a
    \)
    has no vertex intersections beyond the prescribed concatenation
    boundaries.

    \item The terminal gadget is either $\frac{1}{d}I_m $ or $M_D$, or a gadget \(\tau_b\) that has a nontrivial vertex
    intersection with a gadget that appeared earlier in the global record.
    This earlier gadget may lie inside the same block or inside a preceding
    block of the global decomposition.

    \item If the terminal interaction is an isolated well-behaved
    backtracking pair
    \(\tau_b\cdot\tau_{b-1}\), then its leading reduced contribution is
    canceled by the MP recurrence and is not regarded as an error term.
    Only a backtracking-residual term in the sense of
    \Cref{def:backtracking-residuals} is retained, supported on the same
    two-gadget record
    \(\tau_b\cdot\tau_{b-1}\).
\end{enumerate}
The length of an LCP is the number of underlying gadgets in its record.
We call \(\tau_a\) its start and \(\tau_b\) its terminal gadget. 
\end{definition}

\begin{definition}[LCP-Decomposition]
\label{def:lcp-decomposition}
Let \(\tau\) be an error term with underlying gadget record
\[
    \tau_t\cdot\tau_{t-1}\cdot\dots\cdot\tau_1.
\]
We say that \(\tau\) admits an \emph{LCP-decomposition} if its record can
be written as
\[
    \tau_t\cdot\tau_{t-1}\cdot\dots\cdot\tau_1
    =
    P\cdot R_r\cdot R_{r-1}\cdot\dots\cdot R_1,
\]
where \(P\) is a possibly empty proper concatenation and each \(R_j\) is
an LCP.

Equivalently, after removing a possibly empty leading proper run, the
remaining gadget record is a concatenation of LCPs.
\end{definition}

It should be noted that we allow vertex intersections between different blocks. The only forbidden
configuration is an isolated backtracking pair across the
boundary of two consecutive blocks---such a pair is canceled at leading
order by the MP recurrence, and any surviving backtracking-residual term is
included in an LCP supported on the original two-gadget record.

\begin{observation}[Terminal Interaction of an LCP]
\label{obs:lcp-terminal-not-isolated-backtracking}
Let
\(
    R
    =
    \tau_b\cdot\tau_{b-1}\cdot\dots\cdot\tau_a
\)
be an LCP. Then one of the following holds:
\begin{enumerate}
    \item \textbf{Non-backtracking intersection:}
    \(\tau_b\) has a nontrivial non-backtracking intersection with an
    earlier gadget in the record;

    \item \textbf{Non-isolated backtracking intersection:}
    \(\tau_b\cdot\tau_{b-1}\) is a non-isolated backtracking pair, i.e.,
    a three-way backtracking/diagonal interaction;

    \item \textbf{Backtracking residual:}
    the terminal two-gadget block
    \(\tau_b\cdot\tau_{b-1}\) supports a backtracking-residual term in the
    sense of \Cref{def:backtracking-residuals}.
\end{enumerate}
In particular, the leading reduced contribution of an isolated
well-behaved backtracking pair is never treated as an error term: it is
canceled by the MP recurrence.
\end{observation}

With these definitions, we can state the structural property of the MP
error terms.

\begin{proposition}[Structural Property for Error Terms]
\label{prop:error-term-A-equivalence}
For any $t>0$, any term in \(\mathsf{Error}_t(A)\) from the equivalence of MP polynomials and graph matrices of concatenated shapes admits an LCP-decomposition,
\[
    \tau_t\cdot\tau_{t-1}\cdot\dots\cdot\tau_2\cdot\tau_1,
    \qquad
    \tau_i\in \{\alpha,\beta, \frac{1}{d} I_m, M_D\}.
\]
\end{proposition}

\begin{proof}
We prove the statement by induction on \(t\). Throughout the proof, terms
containing \(M_D\) or \(\frac1d I_m\) are placed directly into
\(\mathsf{Error}_t(A)\). Moreover, whenever an isolated backtracking pair
is replaced by a backtracking-residual term, that residual retains the
original two-gadget record.

We first note that deleting the leftmost gadget from a record admitting an
LCP-decomposition preserves an LCP-decomposition. Indeed, if the leftmost
block is proper, its remaining sub-record is still proper. If the leftmost
block is an LCP, then its leftmost gadget is its terminal gadget, and
removing it leaves the locally proper prefix from
\Cref{def:local-collision-pieces}.

The claim is immediate for \(t=0,1\). We have
\[
    \q_0(A)=I_m,
\]
and
\[
    \q_1(A)
    =
    A-I_m
    =
    \cm_\alpha+\cm_\beta
    +
    \left(\frac1d I_m+M_D\right),
\]
where the parenthesized term is already included in
\(\mathsf{Error}_1(A)\); both are LCPs of length $1$.

Now fix \(t\geq1\), and assume the proposition holds at levels \(t\) and
\(t-1\). Modulo the already-suppressed terms involving \(M_D\) or
\(\frac1d I_m\), the MP recurrence gives
\begin{align*}
    \q_{t+1}(A)
    &=
    \bigl(A-(1+\gamma)I_m\bigr)\q_t(A)
    -
    \gamma\q_{t-1}(A) \\
    &=
    \bigl(\cm_\alpha+\cm_\beta-\gamma I_m\bigr)\q_t(A)
    -
    \gamma\q_{t-1}(A).
\end{align*}

Thus every new non-scalar multiplication term is obtained by attaching a
new gadget
\[
    \sigma\in \{\alpha,\beta, \frac{1}{d} I_m, M_D\} 
\]
to the left of a record
\(
    \tau_t\cdot\tau_{t-1}\cdot\dots\cdot\tau_1
\)
appearing at level \(t\). By induction, the old record admits an LCP-decomposition, and its tail
\(
    \tau_{t-1}\cdot\dots\cdot\tau_1
\)
also admits an LCP-decomposition. 

We now consider the cases.

\medskip
\noindent
\textbf{Case 0: \(\sigma \notin \{\al,\beta\} \).} 

In this case, we view this as a gadget with negligible contribution from either $M_D$ or $1/d I_m$. This is a terminal gadget by itself, or it forms an LCP as a terminal gadget with a sequence of properly concatenated shapes.

\medskip
\noindent
\textbf{Case 1: \(\sigma\) concatenates properly.}

Suppose that \(\sigma \in \{\al,\beta\} \) has no unintended intersection with the old
record. If the old record is proper, then
\[
    \sigma\circ\tau_t\circ\dots\circ\tau_1
\]
belongs to \(\calP_{t+1}(A)\) (and does not contribute to the error term). Otherwise, we append \(\sigma\) as a proper
block to the existing LCP-decomposition.  

\medskip
\noindent
\textbf{Case 2: \(\sigma\) creates a non-backtracking collision.}

Suppose that \(\sigma\) has a nontrivial intersection with the old record
that is not a designated backtracking intersection with \(\tau_t\). Then
\[
    \sigma\cdot\tau_t
\]
is an LCP: its prefix consists of the single, hence proper, gadget
\(\tau_t\), while its terminal gadget \(\sigma\) intersects a gadget that
appeared earlier in the global record. This earlier gadget may be
\(\tau_t\), a gadget in the tail, or a gadget in an earlier block, as
allowed by \Cref{def:local-collision-pieces}. Combining this terminal LCP
with the LCP-decomposition of the tail gives an LCP-decomposition of the
full record.

\medskip
\noindent
\textbf{Case 3: \(\sigma\cdot\tau_t\) forms a designated backtracking pair.}

First suppose that the pair is well-behaved, and hence isolated from the
earlier tail. If it is a half-diamond or half-flat pair, its leading
two-gadget contribution reduces to
\[
    \gamma\cm_\alpha
    \qquad\text{or}\qquad
    \gamma\cm_\beta,
\]
respectively. After concatenating with the tail, this is exactly the
matching contribution canceled by \(-\gamma\q_t(A)\).

If it is a full-diamond pair, its leading contribution is
\(\gamma I_m\). The two leading \(\alpha\)-gadgets therefore reduce to the
empty record, leaving
\[
    \gamma\,
    \cm_{\tau_{t-1}\cdot\dots\cdot\tau_1},
\]
which is precisely the
matching contribution canceled by \(-\gamma\q_{t-1}(A)\).

Thus, in every isolated backtracking case, the leading reduced
contribution cancels exactly. The only surviving summands are the
backtracking-residual terms from
\Cref{def:backtracking-residuals}. In the full-diamond case, these are the
diagonal backtracking-residual terms. By definition, each residual retains
the original record \(\sigma\cdot\tau_t\), and hence forms a terminal LCP.

Finally, suppose that \(\sigma\cdot\tau_t\) is non-isolated. Then some
vertex participating in the backtracking block also intersects the earlier
tail. This is not a residual from an isolated cancellation; it is a genuine
three-way backtracking/diagonal collision. Consequently,
\(\sigma\cdot\tau_t\) is again a terminal LCP. Combining it with the
LCP-decomposition of the tail gives the desired decomposition of the full
record.

\medskip

The scalar terms in the MP recurrence introduce no new vertex
intersections. Their role in the grouping above is to cancel the leading
contributions of the isolated backtracking pairs. Any pre-existing error
summands appearing in those scalar terms already admit
LCP-decompositions by the induction hypothesis.

The three cases exhaust all possible interactions of the newly attached
gadget with the previous record. Hence every surviving proper record in
\(\q_{t+1}(A)\) belongs to \(\calP_{t+1}(A)\), while every surviving
non-proper \(\alpha/\beta\)-record admits an LCP-decomposition. Under our
bookkeeping convention, each such record contains exactly \(t+1\)
underlying gadgets from \(\{\alpha,\beta\}\). This completes the induction.
\end{proof}

\subsection{Block-Value Bound for Local-Collision-Pieces}

\paragraph{Factor Assignment Scheme}
The crux of our argument is the following factor-assignment scheme. We introduce the following categorization for each step of the walk: we call it an $F/L$ step if it is along an edge appearing for the first/last time, and an $H$-step if neither (i.e., middle appearance).

\begin{mdframed}[frametitle = Combinatorial Factor Assignment for Vertex Factors] \label{prop:vtx-assignment}
 Each vertex $i$ requires a factor $\wt{(i)}$ when it first appears in the walk (namely a factor of $d$ for a circle vertex, and a factor of $m$ for a square vertex) and a subsequent factor of $O(q)$ when it appears as an incoming active-vertex of some $F$ edge in the walk. We assign its corresponding factor  as follows,
\begin{enumerate}
	\item Assign a factor of $\sqrt{\wt(i) } $ for the step if the vertex is appearing for the first time;
	\item Assign a factor of $\sqrt{\wt(i)}$ for the step if the vertex is appearing for the last (final) time;
	\item Assign a factor of $2q \cdot |V(\tau)|$ if the vertex is not incident to the vertex-component of final-appearance vertices connected to the current walk boundary $U_\tau$. (For example, the vertex makes a middle appearance as the destination of an $F$- or $H$-edge.) 
\end{enumerate}
\end{mdframed}

Besides the vertex factor, we also consider the following scheme for edge-factor assignment.
\begin{mdframed}[frametitle = Analytical Factor Assignment for Edges] \label{prop: edge-value-assignment}
Each edge (random variable $g_e$ that appears for an even number of times) contributes an analytical value of $1$, and we assign this value to each individual step that traverses along edge $e$ via factor $B_{ev}(i,e)$---the factor assigned to step-$i$ from edge-$e$ as follows.
\begin{enumerate}
	\item If the appearance of random variable $e$ contains $\geq 2$ copies of an $h_1$ edge, assign the first and final copy a factor of $\frac{1}{\sqrt{d}}$, and assign a factor of $\frac{2q}{\sqrt{d}}$ for any middle appearance;
	\item  If the appearance of random variable $e$ contains $\geq 2$ copies of an $h_2$ edge, assign the first and final copy a factor of $\frac{\sqrt{2}}{d} $;
\end{enumerate}	
\end{mdframed}

\begin{observation}
Let \(x\sim N(0,1/d)\), and let
\[
    h_1(x)=x,\qquad h_2(x)=x^2-\frac1d .
\]
For any \(t_1,t_2\ge 0\) with \(t_1+t_2>0\),
\[
    \E\!\left[h_1(x)^{t_1} h_2(x)^{t_2}\right]\neq 0
\]
only if \(t_1\ge 2\) or \(t_2\ge 2\).
Equivalently, if \(t_1,t_2\in\{0,1\}\) and \(t_1+t_2>0\), then
\[
    \E\!\left[h_1(x)^{t_1} h_2(x)^{t_2}\right]=0.
\]
\end{observation}
 
\begin{lemma}[Block-Value Bound for LCP]  \label{lem:LCP-bound}
	For any $t>0$, and $\tau$ an LCP of length-$t$, we have \[ 
	B(\tau) \leq \frac{\poly(t, q) }{\sqrt{d}} \cdot \sqrt{\gamma}^{t-1} \,.
	\]
\end{lemma}

\begin{proof}[Proof of \cref{lem:LCP-bound}]
	For any LCP $\tau$, we decompose it as $\tau = \tau_I \cdot \tau_R$ where $\tau_I$ is the terminal gadget, and $\tau_R$ is the properly concatenated part. The block-value bound then follows from the next two claims, as we take \[
	B_q(\tau)\leq B_q(\tau_R) \cdot B_q(\tau_I) \leq \frac{\poly(t,q)}{\sqrt{d}} \cdot \sqrt{\gamma}^{t-1}
	\]
\end{proof}

\begin{claim}[Block value of a locally proper concatenation]
\label{clm:block-value-proper-prefix}
Let
\[
    \tau=\tau_t\cdot \tau_{t-1}\cdot \dots \cdot \tau_1
\]
be an LCP of length \(t\), and let
\[
    \tau_R
    :=
    \tau_{t-1}\circ \tau_{t-2}\circ \dots \circ \tau_1
\]
be its locally proper concatenation. Assume \(\gamma<1\). Then
\[
    B_q(\tau_R)
    \le
    (1+o_d(1))^t\cdot 4(t-1)\sqrt{\gamma}^{\,t-1}.
\]
for any $q\ll d$.
\end{claim}

\begin{claim}[Block value bound for collision gadget] \label{clm:collision-gadget-block-bound}
	Let $\tau_I$ be a collision gadget (i.e., a terminal gadget in an LCP); we have \[
	B_q(\tau_I) \leq  \frac{\poly(q)}{\sqrt{d}} \cdot \sqrt{\gamma}
	\]
	for   $q\ll d^{\delta}$ for some $\delta>0$
	.
\end{claim}

\begin{proof}[Proof of \cref{clm:block-value-proper-prefix}]
We apply the block-value factor assignment to the proper concatenation
\[
    \tau_R=\tau_{t-1}\circ \tau_{t-2}\circ \dots \circ \tau_1.
\]
For each interval, we call a vertex a \emph{separator vertex} for the gadget if it makes both
a prior and a subsequent appearance outside this interval; in
particular, it is a middle-appearance vertex for the block-value applied to the whole interval of $\tau_R$.
We first observe that since each edge appears at least twice in the walk, there exists at least one gadget among the \(t-1\)
gadgets of \(\tau_R\) that has its left boundary disconnected from its right boundary by the separator of $\tau_R$. Call such a gadget an \emph{anchor} gadget, and the rest \emph{non-anchor}.

%

For a non-anchor \(\alpha\)-gadget, the block value is at most
\[
B_q(\al \text{-non-anchor}) =     2\cdot \frac{\sqrt{m\binom d2}}{d^2}
    =
    (1+o_d(1))\sqrt{\frac{2m}{d^2}}
    =
    (1+o_d(1))\sqrt{\gamma}\,.
\]
On the other hand, a non-anchor \(\beta\)-gadget is negligible: its two
\(h_2\)-edges contribute \(2/d^2\), while the dominant vertex contribution is
at most \(\sqrt{md}\). Hence
\[
    B_q(\beta\text{-Non-Anchor})
    \le
    \frac{2\sqrt{md}}{d^2}
    =
    o_d(1),
\]
in the regime \(m=O(d^2)\). In particular, it is bounded by
\((1+o_d(1))\sqrt{\gamma}\). Therefore, since there are at most \(t-2\)
non-anchor gadgets, we have
\[
    B_q(\text{non\mbox{-}anchor})
    \le
    (1+o_d(1))^{t-2}\sqrt{\gamma}^{\,t-2}.
\]

It remains to bound the anchor contribution. There are at most \(t-1\) choices
for the anchor gadget.

First suppose the anchor is an \(\alpha\)-gadget. For each such gadget, there
are two possible square-separator choices, and one possible two-circle
separator choice. The square-separator choices contribute at most
\[
    2\cdot
    \left(
        2\cdot \frac{\sqrt{m\binom d2}}{d^2}
    \right)
    =
    (1+o_d(1))\,2\sqrt{\gamma}.
\]
The two-circle separator contributes at most \(\gamma\), as we have
\[
    (1+o_d(1)) \cdot 2 \cdot \frac{m}{d^2} = (1+o_d(1))  \cdot \gamma\,.
\]
 Hence,
\[
    B_q(\alpha\text{-anchor})
    \le
    (1+o_d(1))(t-1)(2\sqrt{\gamma}+\gamma).
\]

Similarly, if the anchor is a \(\beta\)-gadget, then its middle circle vertex
serves as the separator. The two square-side contributions give
\[
    B_q(\beta\text{-anchor})
    \le
    (1+o_d(1))(t-1)\frac{2m}{d^2}
    =
    (1+o_d(1))(t-1)\gamma.
\]
Combining the two anchor possibilities, we obtain 
\[
    B_q(\text{anchor})
    \le
    (1+o_d(1))(t-1)(2\sqrt{\gamma}+2\gamma).
\]

Putting together the anchor and non-anchor block-value bounds,
\[
\begin{aligned}
    B_q(\tau_R)
    &\le
    B_q(\text{non\mbox{-}anchor})\cdot B_q(\text{anchor}) \\
    &\le
    (1+o_d(1))^{t-1}
    \sqrt{\gamma}^{\,t-2}
    \cdot
    (t-1)(2\sqrt{\gamma}+2\gamma).
\end{aligned}
\]
Since \(\gamma<1\), we have \(\gamma\le \sqrt{\gamma}\), and hence
\[
    2\sqrt{\gamma}+2\gamma
    \le
    4\sqrt{\gamma}.
\]
Therefore
\[
    B_q(\tau_R)
    \le
    (1+o_d(1))^{t-1}\cdot 4(t-1)\sqrt{\gamma}^{\,t-1}.
\]
Weakening \((1+o_d(1))^{t-1}\le (1+o_d(1))^t\), we obtain
\[
    B_q(\tau_R)
    \le
    (1+o_d(1))^t\cdot 4(t-1)\sqrt{\gamma}^{\,t-1}.
\]
This proves the claim.
\end{proof}

\begin{proof}[Proof of~\cref{clm:collision-gadget-block-bound}]
	Recall from the structural property of LCP decomposition and \cref{obs:lcp-terminal-not-isolated-backtracking}, the final collision gadget satisfies one of the following properties,
	\begin{enumerate}
    \item it has a nontrivial non-backtracking intersection with some
    earlier gadget in the record; or

    \item \(\tau_b\cdot\tau_{b-1}\) forms a non-isolated backtracking pair,
    i.e., a three-way backtracking/diagonal interaction; or

    \item  $\tau_b$ is the residual left after an isolated
    backtracking cancellation.
\end{enumerate}

\paragraph{Block value bound for non-backtracking intersection}
In this case, the edge-value is $O(\frac{q^4}{d^2})$ while we have a vertex factor of either \( 
\sqrt{md^2} 
\) or \( \sqrt{m^2} \) provided there is no additional vertex making a middle appearance. However, by the property of non-backtracking intersection, we have at least one extra vertex making a middle appearance, and any such appearance gives a gap of $O(\frac{q}{\sqrt{d}} )$, hence we have \[ 
B_q( \text{Non-backtracking Intersection Gadget} ) \leq O(\frac{t^5}{\sqrt{d}} \cdot \sqrt{\gamma} )\,.
\]

\paragraph{Block value bound for three-way diagonal intersection}
The same bound holds as we observe that we would have a bound of $O(1)$ if every vertex is making a first and final appearance within the backtracking intersection, while any vertex that has additionally made a prior appearance gives a slack of \[ 
\frac{1}{\sqrt{d}} \cdot O(q)\,,
\]
again giving a loose bound of
 \[ 
B_q( \text{Three-way Diagonal Gadget} ) \leq O(\frac{t^5}{\sqrt{d}} \cdot \sqrt{\gamma} )\,.
\]
\paragraph{Block value bound for backtracking residue}
For the half-diamond and half-flat residuals, the expansions in
~\cref{sec:MP-def-proof} show that every  summand either
contains an additional centered \(h_2\)-edge, together with an explicit
\(1/d\)-prefactor, or is an \(O(d^{-2})\)-multiple of
\(\cm_\alpha\) or \(\cm_\beta\). Thus every residual summand gains at
least \(d^{-1/2}\) relative to its uncentered backtracking block. The
full-diamond residual is diagonal and contains a centered \(h_2\)
fluctuation, giving the same gain. Therefore
\[
    B_q( \text{Backtracking Residue})
    \le
    \frac{\poly(t,q)}{\sqrt d}
    (\sqrt\gamma)^2.
\]

\paragraph{Negligible terminal gadgets.}
Finally,
\[
    B_q\!\left(\frac1d I_m\right)=\frac1d,
    \qquad
    B_q(M_D)=\widetilde O(d^{-1/2}),
\]
so these terminal cases satisfy the same claimed bound.

\end{proof}

Finally, we wrap up this section by completing the proof of the spectral norm bound of the error terms when we switch from orthogonal polynomials to graph matrices of concatenated shapes.
\begin{proof}[Proof of \cref{lem:mp-error-quant}]
	We use the factor assignment scheme developed in the series of works for obtaining tight norm bounds for graph matrices \cite{JPRTX, HKPX23, KPX24, Xu25, KX26, PotechinXu2026Theta}, which shows that, for the aforementioned factor assignment scheme, we have \[
	\E\Tr[ \left(\text{Error}_t \cdot \text{Error}_t^\top\right )^q  ] \leq \text{matrix-dimension} \cdot B_q(\text{Error}_t)^{2q}\,.
	 \] 
	 Plugging in $B_q(\text{Error}_t) = O(\frac{\poly(t)}{\sqrt{d}}) $ from the above bound  gives us \[ 
	 \E\Tr[ \left(\text{Error}_t \cdot \text{Error}_t^\top\right )^q  ] \leq  m  \cdot  \left(O(\frac{\poly(t)}{\sqrt{d}}) \right)^{2q}\,.
	  \] 
	  Finally, setting $q= \Theta(\log^4 d)$ gives us the desired bound by~\cref{claim:trace-to-norm-rough}.
	  \end{proof}

\section{Analysis of the Inner Matrix \texorpdfstring{$Q$}{Q}} \label{sec:inner-matrix-Q}
We analyze the inner matrix in this section.  We start by showing that each term in the iterative construction admits a ``nice'' decomposition into proper shapes.

 There are two more components that we obtain from the decomposition. In the first section, we verify that the variance of the inner matrix is $1$ in the limit at the sharp threshold of $d^2/4$. This allows us to ultimately obtain a positive definite 
 matrix when we truncate at finite levels in the next section when we additionally allow some small, vanishing slack to the threshold.

The second part of this section establishes the free independence of backbone-dangling shapes. In particular, it gives a tight norm bound for any linear combination of backbone-dangling shapes.

\subsection{Deviation Attachment and Vertical Concatenation}
In this section, we show that the vertical concatenation - arising from evaluating $M^{-1}$ on the deviation term $\eta_i$ at level-$i$- can indeed be viewed as a proper vertical concatenation where we enforce vertex injectivity across the whole piece. In other words, intersection terms in the vertical attachment are negligible, and the variance of the dominant terms behaves as we anticipate in the technical overview.

 More importantly, we show that it suffices for us to focus on the main term of $A^{-1}$ while the extra low-rank update from Woodbury \cref{fact:woodbury} may be ignored in this process as it corresponds to $o_d(1)$ norm components for $i\geq 1$. 
 
%
%
 
 \begin{lemma}[Linearizing Deviation Attachment into Proper Shapes]  \label{lem:vertical-concatenation}
 Given a backbone-dangling shape $\tau$, its deviation vector $\eta_\tau = L(\tau) \in \R^m $, and $w_\tau = M^{-1} \eta_\tau$, write \[ 
\corr(\tau) \coloneqq \frac{1}{1-\gamma} \left( \mathcal{L}^*(w_\tau ) - \gam \cdot \cm_\tau  \right)=   \mathsf{Proper}(\eta_\tau) + \mathsf{Error}(\eta_\tau)\,,
 \]
 where $\mathsf{Proper}(\eta_\tau)$ is a weighted linear combination of proper backbone-dangling shapes,	
 and $\mathsf{Error}$ are the remaining terms. We have
 \[ 
 \Var( \mathsf{Proper}(\eta_\tau))  = \frac{1}{1-\gam}\cdot \Var(\eta_\eta)+o_d(1)=  \frac{\gamma}{1-\gam}\cdot \Var(\tau) + o_d(1) 
 \]
 and
  \[
 \|\mathsf{ErrorShapes}\|_{sp} = o_d(1)\,.
  \]
We remind the reader that for a matrix $X$, we define its variance as $\Var(X)\coloneqq \frac{1}{\dim(X)} \cdot \E[\Tr(XX^\top)] $ and analogously for a vector.
  \end{lemma}
  
  Before we proceed with the analysis, we note that there are three facets in the above lemma:\begin{enumerate}
  	\item The target correction term admits a clean decomposition of graph matrices of proper shapes up to negligible error terms;
  	\item Each proper shape in the decomposition is a backbone-dangling shape except the $\gam \cdot \tau$ component;
  	\item  Altogether the proper shapes have variance precisely scaled by an additional factor of $\frac{1}{1-\gamma}$.
  \end{enumerate}  
 We give a sketch of the analysis here while we note that this is an informal version because it has not incorporated polynomial truncation of $M^{-1}$ (specifically that of $A^{-1}$). We additionally address the truncation issue with quantitative slack subsequently in~\cref{sec:formal-truncation}.


\paragraph{Analysis of $A^{-1}$ and Well-behaved Intersections} For the analysis of the $A^{-1}$ term, there is one component we would like to highlight for this analysis, and precisely why this lemma warrants a formal treatment: not all the intersection shapes in the deviation-attachment are negligible! We isolate out the following collection of intersection terms via the following definitions. 


Regardless of potential intersection within the shape, terms in $\mathcal{L}^*(w_\tau)$ are obtained from the following diagram operations:
\begin{enumerate}
	\item Starting with a backbone-path from $U$ to $V$ through a single square vertex in the middle;
	\item Attach the $A^{-1}$-path to the square-vertex on the backbone-path, and consider it as dangling down from the path.
	\item Attach $\eta_\tau$ to the other end (``bottom'') of the dangling $A^{-1}$-path.
	\item Each shape receives an additional coefficient from the coefficient in the polynomial expansion for $A^{-1}$, i.e., $(-1)^t \cdot \frac{1}{1-\gamma}$ for the $\q_t(A)$ term for any $t\geq 0$.  
\end{enumerate}

In particular, it should be emphasized here that as we have two concatenation operations in a $3$-way concatenation, and there are $2$ types of intersections that may appear:
\begin{enumerate}
	\item Intersections from the attachment of $A^{-1}$-Path;
	\item Intersections from the attachment of $\eta_\tau$ to the end of the dangling $A^{-1}$ path;
\end{enumerate}

This prompts us to define the following well-behaved intersection. For convenience, we call $s_{bp}$ the square vertex on the backbone-path where the $A^{-1}$-path is attached, and $s_{\eta}$ the square vertex at the end of the $A^{-1}$-path where $\tau$ is attached via two edges connecting the two circle vertices of the endpoints of $\tau$.

\begin{definition}[Well-behaved Top, Bottom, Top+Bottom Intersection $\mathcal{T}, \mathcal{B}, \mathcal{TB}$] \label{def:well-behaved-intersection-vertical}
We define the following three categories of well-behaved intersection patterns in the three-way concatenation of backbone-path, dangling $A^{-1}$-path, and deviation attachment $\eta_\tau$:

\begin{enumerate}
	\item \textbf{Top-Intersection $\mathcal{T}$}: an intersection such that $s_{bp}$ is incident to $2$ distinct vertices. In other words, the pair of edges that $s_{bp}$ is incident to in the backbone path matches the pair of edges that $s_{bp}$ is incident to in the dangling $A^{-1}$ path. Moreover, there is no other vertex intersection in the three-way concatenation.
	\item \textbf{Bottom-Intersection $\mathcal{B}$}: an intersection such that $s_{\eta}$ is incident to $2$ distinct vertices. In other words, the pair of edges that $s_{\eta}$ is incident to in the $A^{-1}$ path matches the pair of edges that $s_{bp}$ is incident to in the deviation attachment of $\eta_\tau$. Moreover, there is no other vertex intersection in the three-way concatenation.
	\item \textbf{Top+Bottom-Intersection $\mathcal{TB}$ }: it involves the well-behaved intersection around $s_{bp}$ and $s_{\eta}$, and there is no other vertex intersection.
\end{enumerate}
\end{definition}

We give diagram illustrations as follows.
\begin{figure}[ht!]
    \centering
    \begin{subfigure}[b]{0.45\textwidth}
        \centering
        \includegraphics[width=\textwidth]{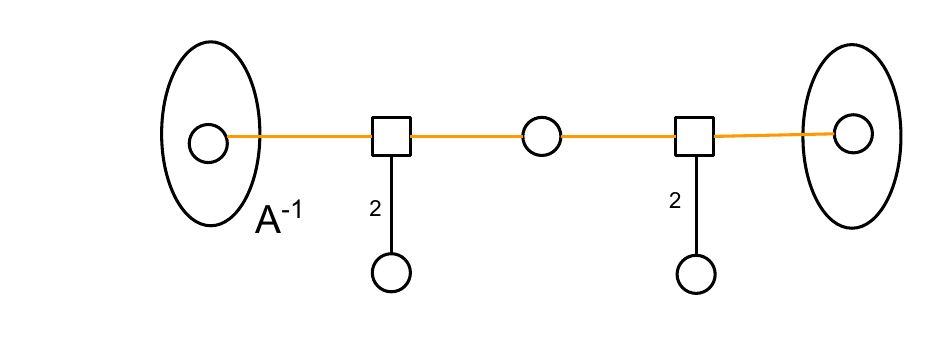}
        \caption{original shape $\tau$ to be corrected.}
    \end{subfigure}
    \quad
    \begin{subfigure}[h]{0.45\textwidth}
        \centering
        \includegraphics[width=\textwidth]{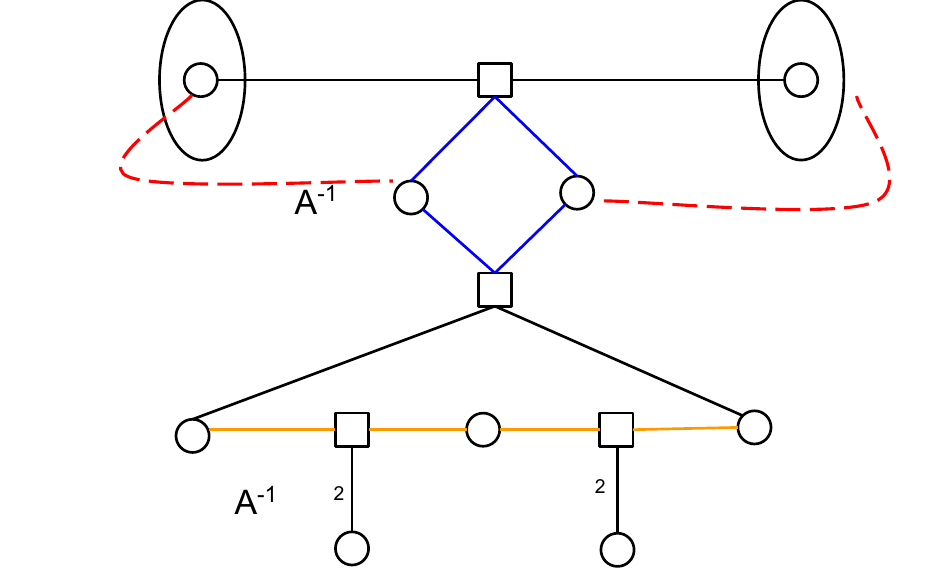}
        \caption{$\mathcal{T}$-intersection}
    \end{subfigure}
    \centering
	 \begin{subfigure}[h]{0.45\textwidth}
        \centering
        \includegraphics[width=\textwidth]{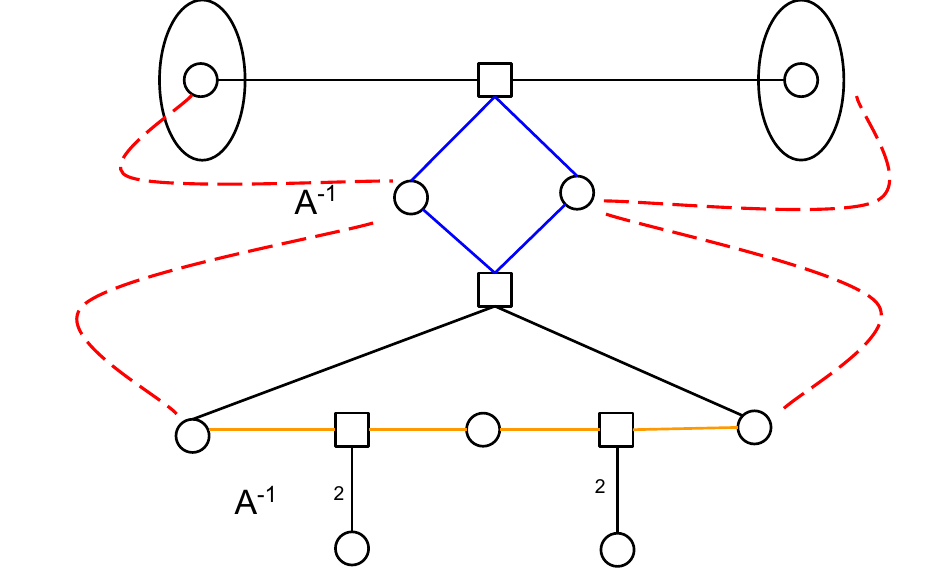}
        \caption{$\mathcal{TB}$-intersection }
    \end{subfigure}
    \begin{subfigure}[h]{0.45\textwidth}
        \centering
        \includegraphics[width=\textwidth]{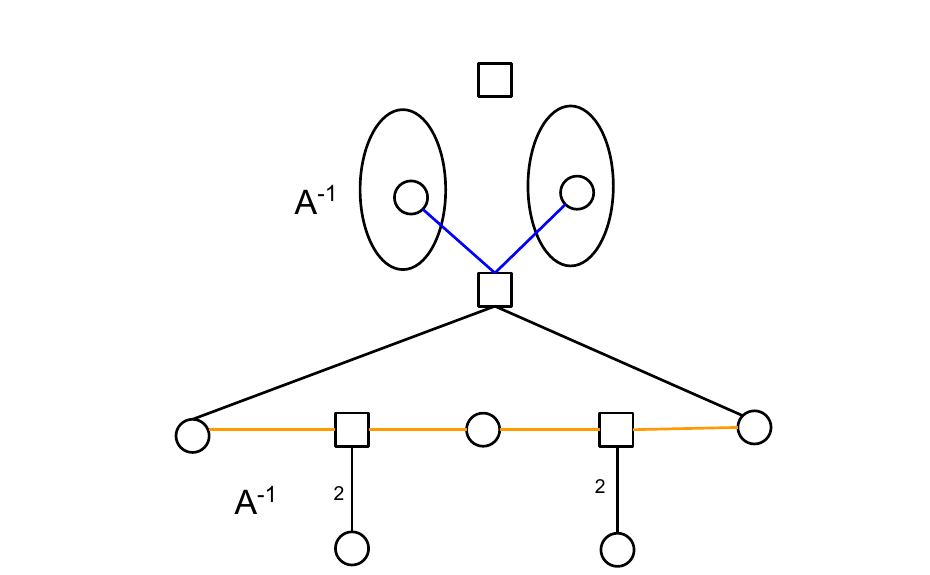}
        \caption{Linearization of $\mathcal{T}$-intersection}
    \end{subfigure}

    \caption{Examples of intersections.}
    \label{fig:concat}
\end{figure}

%
%
\begin{remark}
	Implicitly, the first and final gadget of the dangling $A^{-1}$-path must be an $\al$-gadget to form a well-behaved intersection.
\end{remark}
With these definitions, we observe that the well-behaved intersection term admits an immediate linearization by removing the square-vertex on the backbone-path as well as incident edges. Grouping the intersection terms that reduce to the same underlying shape gives us the following.

\snote{definition 5.4 and proposition 5.5 are somewhat redundant with each other}
\jnote{dropped previous 5.4}
\begin{proposition}[Linearization of Well-behaved Intersection] \label{prop:linearization-well-behaved-vert}
	Let $\tau_I$ be a well-behaved intersection in the sense of \cref{def:well-behaved-intersection-vertical}; we define its linearized shapes as follows:
	\begin{enumerate}
		\item $R(\tau_I)$ has the same vertex boundary as $\tau_I$;
		\item $R(\tau_I)$ is obtained from $\tau_I$ via the following operation: for any square vertex $s \in \{s_{bp}, s_{\eta}\}$ incident to well-behaved intersection (i.e., its incident edges match up), remove the square vertex and its incident edges.
	\end{enumerate}
	Then we have \[
	\cm_{\tau_I} = \frac{m}{d^2}\cdot  \cm_{R(\tau_I)} + \tilde{O}(1/d)\,.
	 \]
\end{proposition}
%
%
With the linearization operation defined, we first identify the properties of the shapes arising from linearization of well-behaved intersections.
\begin{claim} For terms in $\calL^*(w_\tau)$, the linearization of any well-behaved intersection is either $\tau$ (the original shape from horizontal concatenation in the prior level), $\tau^\top$, or a backbone-dangling shape. \label{clm:linearized-property-vertical}
\end{claim}
\begin{proof}
	We prove this by cases depending on the length $t$ of the $A^{-1}$-path. The crucial property to check is that we have only one single square vertex on the backbone path in the resulting shape.
\begin{enumerate}
	\item For the case of $t=0$, there is no $\mathcal{T}$ intersection (and hence neither $\mathcal{TB}$) while there is a $\mathcal{B}$ intersection. Such intersection recovers $\tau$ and $\tau^\top$ (up to flipping of the boundary) as the backbone edges linearize with the deviation-attachment edges. We remind the reader that $\tau$ is not a \emph{backbone-dangling} shape as it has $>1$ square vertex on the backbone path for $\tau\neq \emptyset$ beyond the base iteration.
	\item For the case of $t=1$, linearization of $\mathcal{TB}$ gives $\tau$ and $\tau^\top$. 
	\item For the other cases, consider the backbone-path as depth-$0$, and the horizontal-$\tau$-level as depth $t+1$. Any other linearization has a single vertex on the backbone path as this property is only violated when the horizontal $\tau$-level becomes depth-$0$ after linearization. It is straightforward to verify that each $\calT$ and $\calB$ linearization decreases the depth by $1$ (hence $\mathcal{TB}$ by $2$).
\end{enumerate}	
This completes the proof of our claim.
	\end{proof}
 From the linearization, we obtain an explicit decomposition into proper shapes by combining the proper shapes in the decomposition with the linearization of the well-behaved intersection shapes. Let $\mathsf{Proper}(\eta_\tau)$ be defined to be the weighted combination of properly concatenated shapes in this process as well as the well-behaved intersection terms formally as follows,

\begin{definition}	For any deviation vector $\eta_\tau$ such that $\tau\neq \emptyset$, 	we define $
	\mathsf{Proper} \mathcal{L}^*(w_\tau)$ to be the weighted linear combination of the following terms in the $3$-way concatenation in $\frac{1}{1-\gamma}\cdot  \calL^*(w_\tau)$:
	\begin{enumerate}
		\item Proper concatenations such that each non-trivial gadget in $A^{-1}$-path is an $\al$-gadget;
		\item Linearization of well-behaved intersections such that each non-trivial gadget in $A^{-1}$-path is an $\al$-gadget.
	\end{enumerate}
	It is important to distinguish between $\mathsf{Proper}(\eta_\tau)$ and $\mathsf{Proper}\mathcal{L}^*(w_\tau)$.
\end{definition}

\begin{figure}[ht!]
  \centering
      \begin{subfigure}[b]{0.60\textwidth}
        \includegraphics[width=\textwidth]{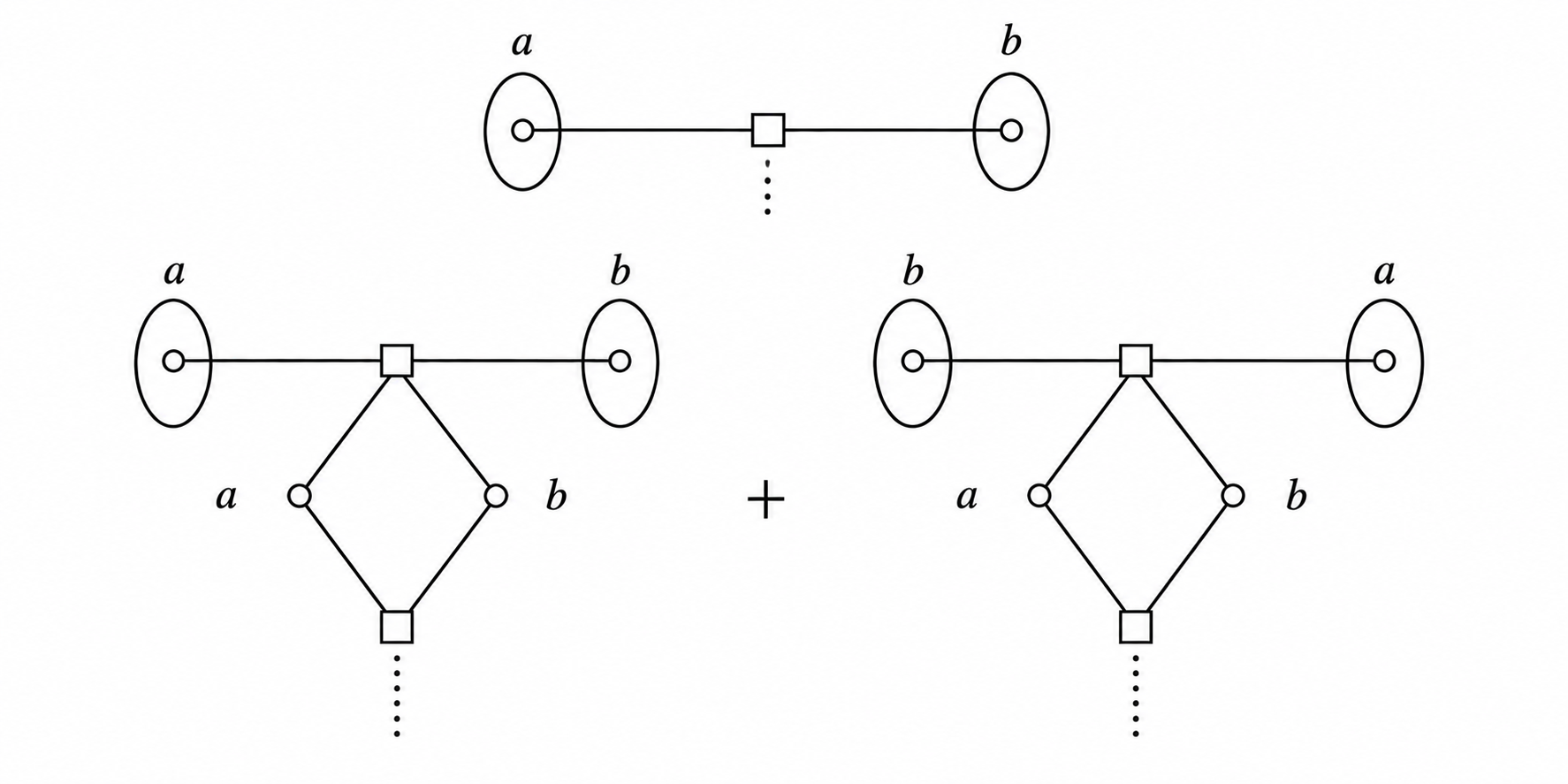}
    \end{subfigure}
            \caption{From $2$ Intersection Terms to $1$ Proper Shape}

\end{figure}

\begin{lemma}[Proper Decomposition for the Scaled Deviation Attachment]\label{lem:proper-decomp-deviation-attach}
	For any horizontally concatenated term $\tau$ to be corrected in the iterative process, 	 we have \[ 
	\mathsf{Proper}\mathcal{L}^*(w_\tau) = \frac{\gam}{1-\gam} \cdot \tau+ \sum_{\substack{\psi \in \mathsf{Proper Shapes}(\eta_\tau) \\ \text{length-$t$ path in } A^{-1} } }  (-1)^{t}  \cdot \cm_\psi \,.
	\]
\end{lemma}

\begin{proof}
	By~\cref{clm:linearized-property-vertical}, each term from the linearization is either $\tau$ or a properly concatenated shape. Hence it suffices for us to consider any such fixed shape $\tau$ and compute its coefficient.

Firstly, we observe that for any properly concatenated shape $\psi$, there are four terms that may give rise to this in the final linearized sum,\begin{enumerate}
	\item The properly concatenated shape itself;
	\item Intersection Shapes that reduce to $\psi$ via linearization of $\mathcal{T}$, call them $\mathcal{T}^{-1}(\psi)$;
	\item Intersection Shapes that reduce to $\psi$ via linearization of $\mathcal{B}$, call them $\mathcal{B}^{-1}(\psi)$;
	\item Intersection Shapes that reduce to $\psi$ via linearization of $\mathcal{TB}$, call them $\mathcal{TB}^{-1}(\psi)$.
\end{enumerate} 

For each term, the coefficient in addition to the coefficient of $c_\tau$ is given by the following,
\begin{enumerate}
	\item In the correction $\corr(\tau)$, $\mathcal{L}^*(w_\tau)$ receives a scaling of $  \frac{1}{1-\gam } $;
	\item There is an additional coefficient of $(-1)^t \cdot \frac{1}{1-\gamma}$ for each (intersection) term that uses a length-$t$ path from the $A^{-1}$ attachment.
\end{enumerate}
Next, we observe that for any fixed shape $\psi$, there are two shapes that linearize to $\psi$ via $\mathcal{T}$ and $\mathcal{B}$ respectively by adding a square vertex with the same pair of edges. The matching of the pair gives us the choice of $2$. Therefore, via~\cref{prop:linearization-well-behaved-vert}, we have \[ 
\cm_{\mathcal{T}^{-1}(\psi)} = 2\cdot \frac{m}{d^2} \cdot \cm_{\psi} +o_d(1) = \gam \cdot \cm_{\psi}+o_d(1);
\]
and similarly,
\[ 
\cm_{\mathcal{B}^{-1}(\psi)} = \gam \cdot \cm_{\psi}+o_d(1)\,.
\]
A consecutive application of $\mathcal{T}$ and $\mathcal{B}$ gives us \[ 
\cm_{\mathcal{TB}^{-1}(\psi)} = \gam^2 \cdot \cm_{\psi}+o_d(1)\,.
\]
For $\psi$ that uses a length-$t\geq 0$ path in the $A^{-1}$-path, shapes that linearize to $\psi$ via $\mathcal{T}$ and $\mathcal{B}$ have length $t+1$ while shapes via $\mathcal{TB}$ have length $t+2$. Write \[ 
\frac{1}{1-\gam} \cdot  \mathcal{L}^*(w) = \sum_{\substack{\psi \\ \text{shape from concatenation} \\ \text{possibly intersected} }} c(\psi) \cdot \cm_\psi 
\]
by expanding in the graph matrix decomposition without processing (i.e., linearizing) intersection terms. Let $c(\psi)$ be the coefficients in the raw decomposition, we have \[ 
c(\psi) = c(\mathcal{TB}^{-1}(\psi)) = -1\cdot c(\mathcal{B}^{-1}(\psi)) = -1\cdot c(\mathcal{T}^{-1}(\psi))  \,,
\]
by alternating the sign in $A^{-1}$, and additionally, for $\psi$ that uses a length-$t$ path in $A^{-1}$, we have\[
c(\psi) = (-1)^t\cdot \frac{1}{(1-\gam)^2}
\]
with one factor of $1/1-\gam$ from the normalization coefficient in front of $\calL^*(w)$, and the other factor from $A^{-1}$. 

Finally, grouping all the terms that produce $\psi$ (with length-$t$ in $A^{-1}$ path) after linearization, the total contribution via a proper shape $\psi$ in $\frac{1}{1-\gamma} \cdot \mathcal{L}^{*}(w_\tau)$ is  
\begin{align*}
	&c(\psi) \cdot \cm_\psi + c(\mathcal{T}^{-1}(\psi)) \cdot \cm_{\mathcal{B}^{-1}(\psi)} + c(\mathcal{B}^{-1}(\psi)) \cdot \cm_{\mathcal{B}^{-1}(\psi)} +  c(\mathcal{TB}^{-1}(\psi))\cdot \cm_{\mathcal{TB}^{-1}(\psi)}\\
	&= c(\psi) \cdot \cm_\psi (1- 2\gam +\gam^2) +o_d(1)\\
	&= (-1)^t \cdot \frac{1}{(1-\gam)^2} \cdot (1-\gam)^2 \cdot \cm_\psi+ o_d(1)\\
	&= (-1)^t \cdot \cm_\psi +o_d(1)\,. 
\end{align*}
This proves the second term in the desired equality. For the first term of $\cm_\tau$, apply the above argument to intersection shapes that linearize to $\tau$, i.e., $\calT^{-1}(\tau)$ and $\mathcal{TB}^{-1}(\tau)$, \[
c(\calT^{-1}(\tau)) \cdot \cm_{\calT^{-1}(\tau)} + c(\mathcal{TB}^{-1}(\tau)) \cdot \cm_{\mathcal{TB}^{-1}(\tau)} = \frac{1}{(1-\gam)^2}  (\gam-\gam^2) = \frac{\gam}{1-\gam}\,.
 \]
This completes the proof.
\end{proof}

Crucially, this shows that the $\frac{1}{1-\gamma}$ coefficients are precisely offset in the weighted linear combination when we combine the properly concatenated shape with the linearization of well-behaved intersection shapes.  Next, for each $\cm_\psi$, we obtain its variance bound by the following,
 \begin{proposition}[Variance of Vertical Attachment] \label{prop:variance-vertical-concate}  For $\psi\in \mathsf{Proper Shapes}(\eta_\tau)$ with a length-$t$ path from $A^{-1}$, we have 
 	\[
 	\Var( \cm_\psi ) = \gam^{1+t} \cdot \Var(\tau)\,. 
 	 \]
 \end{proposition}

 \jnote{this might require a proof}
 
 Since each shape in the collection is backbone-dangling, our main theorem implies that they are freely independent, and therefore the variance of these matrices also adds linearly. 
 Therefore, ignoring truncation for $A^{-1}$, we obtain \begin{align*}
	\Var(\mathsf{Proper}(\eta_\tau) )&= \sum_{\substack{\psi \in \mathsf{Proper Shapes}(\eta_\tau) \\ \text{length-$t$ path in } A^{-1} } }   \Var(\cm_\psi) 
	= \Var(\tau)\cdot   \sum_{t\geq 0} \gamma^{t+1}  
	=\frac{\gam }{1-\gamma} \cdot \Var(\tau) \,.
 \end{align*} 

%

%
\begin{proof}[Proof of~\cref{prop:variance-vertical-concate}]
We compare the two-copy diagram defining
\[
    \Var(\cm_\psi)
    =
    \frac{1}{d}\E\Tr(\cm_\psi\cm_\psi^\top)
\]
with the corresponding diagram for \(\Var(\tau)\).

The deviation attachment adds one square vertex and two
\(h_1\)-edges connecting it to the boundary vertices of \(\tau\).
There are two possible matchings of these edges between the two
copies, so this attachment contributes
\[
    \bigl(1+o_d(1)\bigr)\,
    2m\cdot\frac{1}{d^2}
    =
    \bigl(1+o_d(1)\bigr)\gam.
\]

By definition of \(\mathsf{ProperShapes}(\eta_\tau)\), every
nontrivial gadget on the \(A^{-1}\)-path is an \(\alpha\)-gadget.
Each such gadget adds one square vertex, two circle vertices, and four
\(h_1\)-edges. Its contribution to the two-copy diagram is therefore
\[
    \bigl(1+o_d(1)\bigr)\,
    2m d^2\cdot\frac{1}{d^4}
    =
    \bigl(1+o_d(1)\bigr)\gam.
\]
The \(t\) gadgets on the path thus contribute
\(\bigl(1+o_d(1)\bigr)\gam^t\).

Finally, the outer backbone adds two circle boundary labels \(a\neq b\)
and two \(h_1\)-edges. Its contribution is
\[
    \sum_{a\neq b}
    \E\!\left[v_s[a]^2v_s[b]^2\right]
    =
    d(d-1)\cdot\frac{1}{d^2}
    =
    1+o_d(1).
\]
Multiplying these contributions gives
\[
    \Var(\cm_\psi)
    =
    \bigl(1+o_d(1)\bigr)
    \gam\cdot\gam^t\Var(\tau)
    =
    \bigl(1+o_d(1)\bigr)\gam^{t+1}\Var(\tau).
\]
The error accounts for the global injectivity restrictions and is
uniform over the allowed truncated path lengths. The alternating sign
\((-1)^t\) of the \(A^{-1}\)-coefficient does not affect the variance.
\end{proof}

\begin{restatable}[Local Charging]{proposition}{propLocalTraversal}
\label{prop:local-traversal}
We can assign each edge-copy to at most one vertex so that the following properties hold:
\begin{enumerate}
    \item For each circle vertex, one edge-copy is assigned to each of its global first and global last appearances. In particular, it is assigned one edge-copy if it makes exactly one of these appearances, and two edge-copies if it makes both.

    \item For each square vertex on the dangling path, two edge-copies are assigned to each of its global first and global last appearances. In particular, it is assigned two edge-copies if it makes exactly one of these appearances, and four edge-copies if it makes both.

    \item Whenever a square vertex makes a global middle appearance but its local first appearance within the current block, it is assigned at least one edge-copy.

    \item No edge-copy corresponding to a middle-appearance step is assigned to the global first or global last appearance of any vertex.
\end{enumerate}
\end{restatable}

 This is a generalization of prop. 4.14 in~\cite{HKPX23} in that we additionally incorporate horizontal paths of length $>1$. Our charging and traversal scheme is also an extension of that from the prior work. We defer the full analysis to the appendix. 

\paragraph{Analysis of Base Matrix}
Finally, we extend the above to bound the variance of our initialization at $Q_0$.
\begin{lemma}
	$Q_0$ admits a decomposition into proper backbone-dangling shapes \[
	Q_0 = \sum_{\tau\in \calB(Q_0):\text{proper}} c(\tau) \cdot \cm_\tau  + o_d(1) \,,
	\]
	such that \[ 
	\Var(Q_0) = C_F^2 \cdot \frac{\gamma}{1-\gam}\,.
	\]
\end{lemma}

The proof is in two steps. Firstly, by prop~4.1 of \cite{HKPX23}, we can rewrite \[
M^{-1}\eta = \frac{r+s}{s^2-ru} \cdot A^{-1}\eta - \frac{u+s}{s^2-ru} A^{-1} \1_m\,.
 \]
 The key is to observe that the second term is an $o_d(1)$ error term, while we can decompose the first term via proper backbone-dangling shapes with the desired variance bound. Once rearranged as above, we apply similar ideas to the above calculation for the subsequent levels with one minor distinction: since the final attachment is now an $h_2$-gadget instead of two $h_1$ anchors to a horizontal level, the definition of bottom well-behaved intersection $\mathcal{B}$ needs to be adjusted for $Q_0$.
 \begin{definition}[Well-behaved  Intersection for $Q_0$] Let $s_\eta$ be the final square-attachment vertex at the end of the $A^{-1}$ dangling path. An intersection is a well-behaved bottom intersection if
 	  $s_{\eta}$ is incident to the same $h_2$ edge in $A^{-1}$ and in the $h_2$-attachment. In other words, the two vertices around $s_\eta$ intersect.

 	  Analogously, the modification applies to the $\mathcal{TB}$ intersection for $Q_0$.
 \end{definition}
 It should be noted that by the above definition, it is implicit that the final $A^{-1}$-gadget is a $\beta$-gadget, to be contrasted with the previous discussion for the subsequent correction terms. With this definition, we now consider terms in $\mathcal{L}(M^{-1}\eta)$.
 
 We declare the following terms as the main terms in the decomposition. \begin{definition}	 For the base iteration, we define $
	\mathsf{Proper} \mathcal{L}^*(M^{-1}\eta )$ to be the weighted linear combination of the following terms in the $3$-way concatenation in $\frac{r+s}{s^2-ru} \cdot A^{-1}\eta$ :
	\begin{enumerate}
		\item Proper concatenations such that each non-trivial gadget in $A^{-1}$-path is an $\al$-gadget;
		\item Linearization of well-behaved intersections such that each non-trivial gadget in $A^{-1}$-path is an $\al$-gadget without interior $\beta$-gadget.
	\end{enumerate}
	Let $\mathsf{ProperShape} \mathcal{L}^*(M^{-1}\eta )$ denote the collection of underlying proper shapes.
\end{definition}

Before proceeding, let's observe that among the properly concatenated terms, there is in fact a unique term $\psi$ that uses a length-$t$ path in $A^{-1}$ for any $t\geq 0$ since each gadget in the $A^{-1}$ path is an $\al$-gadget.  This allows us to call such a term $\psi(t)$.

\begin{claim}[Decomposition of Proper Shapes in $Q_0$]\label{clm:proper-q0-decomp}
	\[ 
	\mathsf{Proper} \mathcal{L}^*(M^{-1}\eta ) = \sum_{\substack{\psi\in \mathsf{ProperShape} \mathcal{L}^*(M^{-1}\eta )  \\ \text{length-$t$ in }A^{-1}} } (-1)^t \cdot \cm_\psi 
	\]
	Moreover, each shape $\psi\in \mathsf{ProperShape} \mathcal{L}^*(M^{-1}\eta )$ is a backbone-dangling shape
\end{claim}  
\begin{proof}
	This is analogous to \cref{lem:proper-decomp-deviation-attach} except the linearization of any well-behaved intersection is also a backbone-dangling shape as we no longer have a non-$h_2$ horizontal level in the attachment. Fix any proper shape $\psi$, again there are $3$ (intersection) shapes that can contribute via $\psi$ after linearization:
	\begin{enumerate}
		\item The properly concatenated shape $\psi$ itself;
		\item Intersection Shapes that reduce to $\psi$ via linearization of $\mathcal{B}$ or $\mathcal{T}$;
		\item Intersection Shapes that reduce to $\psi$ via linearization of $\mathcal{TB}$.
	\end{enumerate}
	Let's first consider the coefficient of $\psi$ in $\mathcal{L}^*(A^{-1}\eta )$.
	Combining with the $\frac{1}{1-\gamma}$ coefficient in $A^{-1}$ and the alternating sign in the length-$t$, we have the final coefficient of $\psi$ in $ \mathcal{L}^*(A^{-1}\eta )$ be \[ 
	\frac{1}{1-\gamma} \cdot (-1)^t (1 - 2\gam + \gam^2) = (1-\gam)\cdot  (-1)^t
	\]
	Next, recall that we have an additional coefficient of \[ 
	\frac{r+s}{s^2 -ru } = \frac{1}{1-\gamma}
	\]
	via \cref{lem:hyper-parameter-bounds} recalling that $s=O(1), r= \frac{m}{d}\frac{1}{1-\gamma}$ and $s = -(1-\gamma)$ up to $o_d(1)$ terms. Combining the above proves the desired claim.
		\end{proof} 
		
\begin{proof}[Proof of~\cref{lem:q0-graph-variance}]
    We now analyze the variance of the untruncated $Q_0$, deferring the analysis of the truncated version to the next section. We show that the remaining terms not in $\mathsf{Proper}\mathcal{L}^{*}(M^{-1}\eta)$ are of negligible norm in the appendix. By definition of $\mathsf{Proper}\mathcal{L}^{*}(M^{-1}\eta)$, each term is a backbone dangling shape without any intersection.
	
	Observe that for proper shape $\psi$ with length-$t\geq 0$ in $A^{-1}$ path, we have \[ 
	\Var(\cm_\psi) = \gam^{1+t}\,.
	\]
	To see this, we start by observing that for $t=0$, we have \[ 
	\Var(\cm_{\psi(t=0)} ) = \frac{1}{d^2} \cdot d^2 \cdot \frac{2}{d^2} \cdot m = \gam \,.
	\]
	Each additional gadget within the $A^{-1}$ path is an $\al$ gadget (as the $\beta$-gadget is negligible under an injective $A^{-1}$-path), and any such gadget contributes a factor of $\gam$. Assuming each backbone-dangling shape is freely independent (established in the later section), and recall that $Q_0$ is the identity perturbation scaled with an additional factor of $C_F$, by the decomposition in \cref{clm:proper-q0-decomp} we have \[ 
	\Var(Q_0) = C_F^2 \cdot \sum_{\psi(t)} \Var(\cm_{\psi(t)})  = C_F^2 \gam \sum_{t=0}^\infty \gam^t = \frac{\gamma}{1-\gam }\cdot C_F^2\,.
	\]
	
	\end{proof}		
\subsection{Concrete Primal Iterative Process via Graph Matrices}\label{sec:concrete-primal-graph-mat}
In this section, we instantiate the ideal iterative procedure described with graph matrices to streamline our analysis. For simplicity of the discussion, we focus on the primal program while analogous changes apply to the dual program in a straightforward manner. See~\cref{app:concrete-dual-iteration} for details regarding the dual program. We now focus on the primal.

We incorporate the following changes into the iterative process:
\begin{enumerate}
	\item Instead of reasoning about $H_i$ directly and adding update $\corr(H_i)$, we will restrict our attention to the properly concatenated shapes arising therein, i.e., $\fp_j(Q_i) - \fp_j(Q_{i-1})$ for any matrix $Q_i$ and any $j\ge 2$. 
	\item Our update will address solely proper shapes from  $\fp_j(Q_i) - \fp_j(Q_{i-1})$, and we will treat the remaining terms as negligible error terms. We call any error in this process $\textsf{HorizonErr}$.	\item We additionally truncate $A^{-1}$ expansion at degree-$D$ for each dangling $M^{-1}$-path. We defer these error terms for the truncation section.
	\item The three-way concatenation to correct for the proper shapes in $\fp_j(Q_i) - \fp_j(Q_{i-1})$ may incur another collection of (non-well-behaved) intersection terms as we address in the previous section. We call any error in this process $\mathsf{VerticalErr}(i)$. We discard these terms in the iteration update $\Delta_i$;
	\item We additionally discard diagonal terms in the update $\Delta_i$. These are small terms for which an iterative update is unnecessary, and we discard them to ensure $Q$ is definitionally a linear combination of backbone-dangling shapes recalling that they must be off-diagonal terms. Analogously, a similar trimming process is also applied to the initialization of $Q_0$.
\end{enumerate}
This is summarized in the following process. For any $i\ge 0$, define the proper residuals (to be corrected) by
\[
    H_0^{\mathrm{prop}}
    \coloneqq
    \sum_{j=2}^{D}b_j\fp_j(Q_0),
\]
and, for \(i\geq1\),
\[
    H_i^{\mathrm{prop}}
    \coloneqq
    \sum_{j=2}^{D}b_j
    \bigl(\fp_j(Q_i)-\fp_j(Q_{i-1})\bigr).
\]

Let \(\corr^{\leq D}(H_i^{\mathrm{prop}})\) denote the correction obtained after truncating
each MP expansion in the dangling \(M^{-1}\)-paths at degree \(D\) in the three-way concatenation of \(
\calL^* M^{-1} \calL \left( H_i^{\mathrm{prop}} \right)
\), and
let
\[
    \Delta_i
    \coloneqq
    \operatorname{Proper}\!
    \left(\corr^{\leq D}(H_i^{\mathrm{prop}})\right),
\]
where \(\operatorname{Proper}(\cdot)\) retains the proper
backbone-dangling shapes after linearizing all well-behaved
intersections. Importantly, intersection shapes from the vertical three-way concatenation as well as diagonal terms are discarded.
 
 Finally, we also apply similar procedure to the initialization at $Q_0$ to remove negligible components therein. Let the resulting matrix be $\operatorname{Proper}(\calR)$. 
 \begin{mdframed}[linewidth=0.2pt]
\textbf{Summary of the Concrete Truncated Iterative Procedure.}

\medskip
\textbf{Initialization:}
Set
\[
    Q_0\coloneqq-C_F \cdot \operatorname{Proper}(\calR),
\]
so that \(I+Q_0/C_F\) satisfies the affine constraints.

\medskip
\textbf{Iterative update \(i\to i+1\):}
For \(0\leq i<t^\ast\), set
\[
    \Delta_i
    \coloneqq
    \corr_{\mathrm{prop}}^{\leq D}
    \bigl(H_i^{\mathrm{prop}}\bigr),
    \qquad
    Q_{i+1}\coloneqq Q_i+\Delta_i.
\]

\medskip
\textbf{Output:}
Let
\[
    Q\coloneqq Q_{t^\ast}.
\]
The final matrix is obtained from
\(\frac1{C_F}F^{\leq D}(Q)\) by removing the aggregate discrepancy
defined below and correcting the final proper residual.
\end{mdframed}
\paragraph{Affine invariant for the concrete iteration.}
To compare the concrete iteration with the ideal one, define
\[
    \mathsf{HorizonErr}(i)
    \coloneqq
    \sum_{j=2}^{D} b_j
    \bigl(P_j(Q_i)-\fp_j(Q_i)\bigr)
\]
and
\[
    \mathsf{UpdateErr}(i)
    \coloneqq
    \Delta_i-\corr\!\left(H_i^{\mathrm{prop}}\right) = \corr_{\mathrm{prop}}^{\leq D}
    \bigl(H_i^{\mathrm{prop}}\bigr) -\corr\!\left(H_i^{\mathrm{prop}}\right) .
\]
Here \(\mathsf{UpdateErr}(i)\) collects the MP-truncation error and the
discarded vertical intersection terms as well as diagonal terms. Define the cumulative discrepancy
\[
    \mathsf{IterErr}_i
    \coloneqq \mathsf{HorizonErr}(i)
+
        \sum_{s=0}^{i} \mathsf{UpdateErr}(s) \,.\]

\begin{proposition}[Invariant for the concrete iteration]
\label{prop:concrete-iteration-invariant}
For every \(i\geq0\),
\[
    \calL\!\left(
        \frac1{C_F}
        \bigl(F^{\leq D}(Q_i)-\mathsf{IterErr}_i\bigr)
    \right)
    -
    \1_m
    =
    \frac1{C_F}\calL\!\left(H_i^{\mathrm{prop}}\right).
\]
\end{proposition}

At this point, we observe that there are three sources of error terms: vertical, horizontal, and MP truncation. We defer the last type of truncation term to the subsequent section alongside other truncation details. The remainder of this section is then dedicated to formalizing and establishing error bounds for vertical (\cref{sec:vertical-details}) and horizontal (\cref{sec:chebyshev-Q}) concatenation that culminate in \cref{lem:error-vertical} and \cref{lem:semicircle-shape-concatenation}.

\subsection{Combinatorial Analysis of the Error Terms for Vertical Concatenation}
\label{sec:vertical-details}

We now shift gear to considering the error terms, and bound the error term by the following lemma.
\begin{lemma}[Error Terms for Vertical Concatenation] \label{lem:error-vertical}
	Recall that we define 
	\[ 
	\mathcal{L}^*(w_\tau) = \mathsf{Proper}(\eta_\tau) + \mathsf{Error}(\eta_\tau)\,,
	\]
	we have \[ 
	 \|\mathsf{Error}(\eta_\tau)\|_{sp}=o_d(1)\,.
	\]
\end{lemma}

 There are primarily three sources of error terms:\begin{enumerate}
	\item Non well-behaved intersections in the concatenation of $A^{-1}$ and $\eta_\tau$;
	\item Terms from rank-$2$ components in $M^{-1}$;
	\item Diagonal terms from the vertical concatenation.
\end{enumerate}
For the error term analysis, we will again apply block-value bound throughout.
For starters, we consider the error terms from the concatenation of $A^{-1}$. We give a charging scheme for $A^{-1}$ terms, and then show it also applies for the rank-$2$ components.

\paragraph{Pre-processing Well-behaved Intersections} Since the intersection shapes may consist of a well-behaved intersection with extra vertex intersections, we first process away the well-behaved intersection term by removing the initial square vertex and its incident edges. With this pre-processing, we may assume there is no well-behaved intersection around the initial and final square vertex of the $A^{-1}$-path.

 This further allows us to deduce the following claim that shows either the backbone-path is unaffected in the concatenation operation, or we can identify an intersection property that can later be used to deduce a block-value gap.


%
%

\paragraph{Preprocessing the \(A^{-1}\)-expansion.}
Throughout this subsection, every occurrence of \(A^{-1}\) is replaced by its
truncated expansion in properly concatenated \(\alpha\)- and \(\beta\)-gadgets.
Terms containing an MP approximation error have already been shown to have
negligible spectral norm in~\cref{lem:mp-error-quant} and are absorbed into
\(\operatorname{Error}(\eta_\tau)\).

We also preprocess all well-behaved top, bottom, and top--bottom intersections
from~\cref{def:well-behaved-intersection-vertical}. Namely, whenever the
incident edge-pairs around \(s_{\mathrm{bp}}\) or \(s_\eta\) match, we apply
the linearization operation of~\cref{def:well-behaved-intersection-vertical}.
Consequently, every error term considered below has at least one vertex
intersection that is not removed by the well-behaved linearization.


\propLocalTraversal*

We introduce the following definitions for our analysis. Recall from our norm bounds analysis that we bound the norm of a graph matrix by analyzing length-\(2q\) ``block-walks'' of the underlying shape and controlling the vertex and edge factors contributed by each ``block-step.'' To carry out this local analysis, we must distinguish between the \emph{local} and \emph{global} appearances of labeled vertices and edges. We remind the reader that a labeled square vertex is an element of \([m]\), a labeled circle vertex is an element of \([d]\), and a labeled edge is an element of \([m]\times[d]\).

\begin{definition}[Local and global appearances]
    Consider a fixed block-walk. An appearance of a labeled vertex within a given block is called a \emph{local appearance}, whereas an appearance of that vertex anywhere in the full block-walk is called a \emph{global appearance}.

    A labeled vertex or edge makes its \emph{global first} or \emph{global last appearance} if the corresponding occurrence is its first or last appearance in the entire walk. Similarly, it makes its \emph{local first} or \emph{local last appearance} if the corresponding occurrence is its first or last appearance within the current block.
\end{definition}

We also introduce a special class of steps that will be treated differently by our charging scheme. The terminology ``reverse-charging'' will become clear once the edge-copy charging rule is described below.

\begin{definition}[Reverse-charging step/edge (Def.4.10 in~\cite{HKPX23})]
\label{def:reverse-charging}
    Fix a block in a given block-walk. We call a step \(u\to v\) \emph{reverse-charging} if the following conditions hold:
    \begin{enumerate}
        \item the underlying edge is making its global last appearance in the walk;
        \item the global first appearance of the same underlying edge also occurs within the current block;
        \item that first appearance traverses the edge in the opposite direction, namely from \(v\) to \(u\).
    \end{enumerate}
\end{definition}

We first describe a traversal process for a backbone-dangling shape and then define the charging scheme induced by this traversal.

\begin{mdframed} \textbf{Top-Down Traversal for Edge-Copy Charging}:
 Assign depth-label for the vertices in the our backbone-dangling shape as follows. This assignment process is oblivious of the vertex-labels.  \begin{enumerate}  
	\item \textbf{Base Case}: the boundary vertices in $U\cup V$ are of depth-$0$, as well as the square-vertex on the backbone path;
	\item \textbf{Traversal for Dangling $A^{-1}$-path:} view each $A^{-1}$-path attached to the square vertex in a top-down manner.  
	\item   \textbf{Traversal for Chebyshev Concatenation}: \begin{itemize}
	 \item Each horizontal-level is attached by two \emph{attachment edges} to the prior square vertex. For $\tau\neq \emptyset$, we have two attachment edges connecting two distinct circle vertices. For the special case of $\eta_0$, we consider the $h_2$ edge as consisting of two attachment-edges as well.
	 \item Without loss of generality, we assign one attachment edge to the left attachment vertex, and call the other edge \emph{attachment-reserve}.
	 \item We traverse each horizontal path from left to right.
	 	\end{itemize} 
\end{enumerate}
\textbf{Charging via Traversal:} 

\begin{enumerate}
	\item \textbf{Backbone Edges:} assign both edges to the square vertex $s_{bp}$;
	\item \textbf{Dangling Edges and Subsequent Horizontal Levels:} assign each edge to the vertex it leads to, unless it is a \emph{reverse-charging edge}, in which case we assign the edge to the source. \end{enumerate}
\end{mdframed}

\begin{proof}  
We make the following observations that are immediate from the charging scheme:
\begin{displayquote}
	 Square vertices that make their first appearance within any $A^{-1}$-path, $s_{bp}$ the square vertex on the initial backbone path, and  any circle vertex outside the $U\cup V$ boundary have been assigned the required factors.
\end{displayquote}
Any such vertex is assigned 1 or $2$ edges for its first local appearance, and these edges then protect the potential subsequent final appearance of the vertices within the same block-step by ``reverse-charging''.

Putting these vertices aside, it suffices for us to focus on square-attachment vertices (the vertices at the end of some $A^{-1}$-path that additionally spawn an additional horizontal level or $h_2$-gadget), and circle vertices within the boundary.

\paragraph{Square-Attachment Vertex (other than $s_{bp}$):}
 For any vertex that makes its first appearance as a square-attachment vertex, we first observe that it cannot make both its first and final appearances within the same block-step. This follows by our $3$-way concatenation property: the deviation-attachment $\tau$ that induces the correction term is a proper shape. Hence, any square-attachment vertex would come from the $\tau$ piece of the horizontal attachment, and it must be intersected by some vertex from $s_{bp}$ or the interior of the $A^{-1}$ path. Hence, the vertex is not making its \emph{first} appearance as a square-attachment vertex.

\paragraph{Finding Attachment Reserve}
 That said, we still need to assign two edge-copies to such vertices. It receives one factor from the edge leading to it in the traversal on the horizontal path, and we will identify an additional edge from the reserve-edge.  Without loss of generality, we may assume the left attachment edge is not reverse-charging to the square vertex (since by our preprocessing of well-behaved intersection, both edges cannot be both reverse-charging). We next claim that we can find an unassigned edge as \textbf{attachment-reserve}: starting from the other attachment edge, it is not assigned unless it is a reverse charging for the circle vertex. In that case, we observe that the square vertex must have made a prior appearance, and therefore any of the two edges assigned in the most recent appearance can be used as \textbf{attachment reserve}.

\paragraph{Charging Boundary Vertices:}
We first notice that we can apply the above \emph{finding-reserve} idea to the initial square vertex $s_{bp}$ as well to identify a reserve edge. Next, apply the corollary 4.16 and its preceding proposition 4.15 in \cite{HKPX23}:
\begin{displayquote} either one boundary vertex requires a $\sqrt{d}$ factor (a single edge copy), or there is a vertex making a middle appearance outside $U\cup V$
\end{displayquote}
It is clear in the first case as we assign the edge-copy to the contributing vertex. In the latter case, their assigned factor from the local traversal process is unnecessary and can be combined with the \textbf{backbone-reserve} edge and assigned to both boundary vertices.
\end{proof}


\paragraph{Extension to Diagonal Terms} Recall that in the formal iterative process, we additionally discard the diagonal terms at the end of each vertical concatenation instead of applying Chebyshev polynomials with correction iteratively. The analysis follows from the corresponding component in the proof for proposition 4.15 from~\cite{HKPX23}.

\paragraph{Extension to Woodbury Terms} To get started, it is useful to recall the following scalar bounds. \HyperParameterBounds*
As a result, we also have\[
\frac{1}{s^2 - ru } = \Theta \left(\frac{d}{m}\right)\,.
 \]
Additionally, we split $M_W$ into three components depending on the inner term, 
\[
	M_W \coloneqq  \frac{1}{s^2-ru}\, A^{-1} \left( \underbrace{ u\,\frac{\1_m\1_m^T}{d}}_{M_J} -  \underbrace{s\,\frac{\eta\1_m^T+\1_m\eta^T}{d}}_{M_C} + \underbrace{r\,\frac{ \eta\eta^T}{d} }_{M_\eta} \right) A^{-1}
	 \,,\]
Following the language in \cite{HKPX23}, each $M_W$ term can be viewed as either concatenating a permutation jump ($\1_m \1_m^\top$ or $\eta\cdot \1_m^\top)$ at the end of the first $A^{-1}$ path, or an extra $\eta$-gadget ($\eta\eta^\top$).
At a high level, we continue applying the traversal-charging scheme described above, albeit with the following caveats:
\begin{enumerate}
	\item For the $\1_m\1_m^T$ term as well as the $\1_m\eta^\top$ factor, the charging for the  first $A^{-1}$ path is essentially identical up to the final square-attachment vertex as it is not incident to an $h_2$-attachment, and thereby missing a factor of $1/\sqrt{d}$ from the corresponding reserve factor;
	\item We no longer necessarily have vertices making middle appearances (via non-trivial vertex-intersections), hence we need to identify a slack from other factors beyond just an application from the prior scheme.
\end{enumerate}

\paragraph{Analysis for Terms in $M_J$}
We make the following observations:
\begin{enumerate}
	\item Each term is scaled by a coefficient of \[ 
	\frac{1}{s^2-ru} \cdot u  \cdot \frac{1}{d} = O(\frac{d}{m}) \cdot O(1) \cdot \frac{1}{d} =\frac{1}{m} \,.  
	\]
	\item Each term can be described as follows: we attach an additional permutation-jump between square vertices (corresponding to $\1_m\1_m^T$), and from the destination square vertex, we start a new $A^{-1}$ dangling path along with deviation attachment of $\eta_\tau$ at the end;
	\item The permutation-jump requires an additional factor of $\sqrt{m}$ for the first local appearance, call the destination square vertex $s$.  
	\item Unless $s$ is the destination of some edge in the local traversal process, $s$ cannot make its first and final global appearances in the same block-step;
	\item For the remaining $A^{-1}$ path, traverse from the start vertex, and assign each edge to the vertex it leads to. 
	\item For the $\eta_\tau$ attachment to the final square vertex $s$ along the dangling $A^{-1}$ path, there are two cases:
	\begin{itemize}
		\item Assign both anchor edges to the square vertex if the final square vertex  $s$ is incident to $2$ distinct vertices (i.e., the pair of its neighbors in $A^{-1}$ path and in $\eta_\tau$ intersect in pairs). Again, we can apply partial linearization to remove the $s$ vertex and its incidental edges.
		\item Apply the horizontal traversal scheme: assign one edge to the circle vertex, and traverse horizontally, assign the edge to the vertex it leads to. Finally, we declare the other anchor-edge as reserve.
	\end{itemize} 
	 \end{enumerate}
Therefore, we have a total block-value of \[ 
O(1) \cdot \frac{1}{m} \cdot \sqrt{m} \cdot O(1) = o_d(1)\,.
\]
where the $O(1)$ factor comes from our local charging for the $A^{-1}$ backbone path as well as the floating $A^{-1}$-path and deviation attachment rooted at $s$.
\paragraph{Analysis for Terms in $M_\eta $}	
This is the dominant term where we use $\eta_i\neq \eta$ for $i\geq 1$. Note that this is a term that contains an $h_2$-attachment for the backbone-dangling path, and the coefficient for this term is  $1/d$.  We apply the following charging argument in which we split $1/d$ into two copies evenly:
\begin{enumerate}
	\item The backbone-dangling path contains an $h_2$-attachment, hence it suffices for us to apply the charging argument to the backbone-dangling path and bound the component to have a block-value $O(1)$. Next, it suffices for us to bound the block-value corresponding to the floating component.
	\item Assign $\tilde{O}(\frac{1}{\sqrt{d}}$) value to the initial circle vertex $s$ that is part of the $h_2$ attachment in the $\eta^\top A^{-1} \eta_i $-path. This leaves us an excess of $\tilde{O}( \frac{1}{\sqrt{d}})$ from the scalar coefficient;
	\item Our key observation is that the initial-circle vertex $s$ cannot be making both of its first and final appearances in a $\eta^\top A^{-1} \cdot \eta_i$ component for $i\neq 0$, and the remaining vertex factors can be charged by the edge-factors in the component, hence leaving us with an excess of $\frac{1}{\sqrt{d}}$.
\end{enumerate}

\paragraph{Analysis for Terms in $M_C$}

Recall that
\[
    M_C
    =
    -\frac{s}{d(s^2-ru)}
    A^{-1}
    \bigl(\eta\1_m^\top+\1_m\eta^\top\bigr)
    A^{-1},
    \qquad
    \left|\frac{s}{d(s^2-ru)}\right|
    =
    O\left(\frac1m\right).
\]
The case of $\eta = \eta_\emptyset$ is covered in the previous work of \cite{HKPX23}.
Thus, it suffices for us to consider \(\eta_i\neq\eta\), and we have
\[
    \calL^*(M_C\eta_i)
    =
    -\frac{s}{d(s^2-ru)}
    \left[
        \bigl(\1_m^\top A^{-1}\eta_i\bigr)
        \calL^*(A^{-1}\eta)
        +
        \bigl(\eta^\top A^{-1}\eta_i\bigr)
        \calL^*(A^{-1}\1_m)
    \right].
\]
We bound the two terms as follows:
\begin{enumerate}
    \item
    In the first term, \(\calL^*(A^{-1}\eta)\) is a
    backbone-dangling component with an \(h_2\)-attachment and has
    block-value \(O(1)\). The \(\1_m\)-endpoint in the floating
    component contributes at most \(\sqrt m\), while the remaining
    factors are charged as in the \(M_J\)-case. Hence this term has
    block-value
    \[
        O\left(\frac1m\right)\cdot\sqrt m\cdot O(1)
        =
        O\left(\frac1{\sqrt m}\right).
    \]

    \item
    In the second term, the \(\1_m\)-endpoint in
    \(\calL^*(A^{-1}\1_m)\) again contributes at most \(\sqrt m\).
    After paying this factor, the remaining normalization is
    \[
        O\left(\frac{\sqrt m}{m}\right)
        =
        O\left(\frac1d\right),
    \]
    since \(m=\Theta(d^2)\). The preceding \(M_\eta\)-analysis therefore
    applies directly to the floating component
    \(\eta^\top A^{-1}\eta_i\), leaving a block-value gap
    \(\widetilde O(d^{-1/2})\).
\end{enumerate}
Consequently,
\[
    B_q\bigl(\calL^*(M_C\eta_i)\bigr)
    \leq
    \widetilde O\left(
        \frac1{\sqrt m}+\frac1{\sqrt d}
    \right)
    =
    \widetilde O\left(\frac1{\sqrt d}\right)
    =
    o_d(1).
\]

\subsection{Norm Bounds for Linear Combination of Backbone-Dangling Shapes}

In this section, we prove norm bounds for general weighted linear combination of backbone-dangling shapes. We use $K$ as a placeholder for such matrix. Our main result is the following theorem for the spectral radius of $K$, which shows it obeys semicircular distribution up to the edge.

We do not explicitly state our result for the bulk of the spectrum, while noting that by similar analysis in \cite{PotechinXu2026Theta}, our results also establish semicircular spectrum by considering $O(1)$ moments of the trace.


\begin{theorem}[Spectral Radius for Linear Combination of Backbone-Dangling Shapes]
\label{thm:Q-norm}
	For a symmetric linear combination of backbone-dangling shapes $K$ with ground-set $\calB(K)$ s.t. \[ 
	K = \sum_{\tau \in\calB(M)} c_\tau \cdot \cm_{\tau}\,,
	\] we have \[ 
	\|K \|_{sp} \leq (2+o_d(1))  ,
	\]
	provided $\Var(K)=1$ 
	
	Quantitatively, let $\|c(K)\| = \sum_{\tau\in B(K)} |c_\tau|$ and $D_V$ be the size bound for shapes in $K$,  with high probability, \[ 
	\|K \|_{sp} \leq 2 + B_q(\mathsf{NonIdeal}) 
	\]	for \[ 
	B_q(\mathsf{NonIdeal}) \coloneqq  O(1) \cdot 3^{|V(\tau)|} \cdot   c(\tau) \cdot    \frac{1}{\sqrt{n}} \cdot \|c(K)\|_1^{  2|V(\tau_i)|}  \cdot (q\cdot D_V)^{2}\,.
	\]
\end{theorem}

We follow the ideas for the analogous result in \cite{PX25} by examining the trace moment calculation for $K^{2q}$ directly. \begin{enumerate}
	\item  We first apply the insight from block-value assignment scheme: it suffices to focus on $F$- and $R$-block-steps (i.e. block-steps in which all by one circle vertex are making their first/final appearances in the wlak) , as all other contributions can be shown to be negligible.

    \item \textbf{Ideal pairs.}
If an \(F\)-step using shape \(\tau\) is matched to an \(R\)-step with the same labeled
edge-set, we pick up a total value of 
\[
   (c_\tau)^2 \cdot  \Var(\tau) 
\]
Summing over \(\tau\) gives the ideal-pair value
\[
    B_q(\text{ideal})^2 \le \Var (K)
\]

    \item \textbf{Non-ideal steps.}
Every non-ideal step is charged to a block-value gap, i.e. to an additional
middle-appearance vertex.  Such a gap gives a factor \(d^{-1/2}\) after normalization.
The only remaining issue is that a single gap may be used to pay for several coefficient
mismatches. We use the gap-witness accounting below bounds this overhead quantitatively. 

\end{enumerate}

The key component of our argument is to identify a slack from any non-ideal step. Towards that end, we introduce some preliminaries.

\paragraph{Preliminaries for Trace Moment Calculation}

\begin{definition}[F/R/S-Step]
 For a block-step $t$ in a trace walk of backbone-correction shapes from the current boundary vertex $U_{t}=u$ to the next boundary vertex $V_{t} =v$, we classify it as following:
\begin{enumerate}
\item $F$-step: each vertex in $\tau_t$ is making a first appearance except the current boundary vertex $u$ (which by definition cannot).
\item   $R$-step: each vertex in $\tau_t$ is making a last appearance except the current boundary vertex $v$ (which by definition cannot). 
 \item Any other step is an $S$-step.
\end{enumerate}

\end{definition}
Additionally, we make a further classification of $F/R$-steps by highlighting steps that can be bundled into steps of the same edg-set.
\begin{definition}[Ideal $F/R$-Step]
For a block-step $t$ in a trace walk of backbone-correction shapes from the current boundary vertex $U_{t}=u$ to the next boundary vertex $V_{t} =v$, we classify it as following:
\begin{enumerate}
\item Ideal-$F$-step:  There is a subsequent step (possibly an $S$-step to be defined) with the same underlying edges-set, i.e. there is some step $t'>t$ s.t. $E(\tau_t) = E(\tau_{t'})$. 
\item   Ideal-$R$-step: There is a a prior step (possibly an $S$-step) with the same underlying edges-set, i.e. there is some step $t'<t$ s.t. $E(\tau_t) = E(\tau_{t'})$ (pick $t'$ to be the smallest one if there are multiple).
 \item Any other $F/R$ step is a \emph{non-ideal} F/R-step.
\end{enumerate}
 \end{definition}

\begin{proposition}[Block-Value for Ideal Pairs]
	For any $\tau$, an ideal-$F/R$-pair of shape $\tau$ gives block-value under our block-value assignment scheme, \[ 
	B_q(\mathsf{Ideal}(\tau) ) \leq (c_\tau)^2 \cdot \Var(\tau)\,. 
	\]
\end{proposition}
\begin{proof}
	Firstly, suppose the $F$-step has the same start vertex as the destination vertex as the $R$ step, then the corresponding block-value bound is precisely given by $\Var(\tau)$ by definition. On the other hand, suppose not and we call this an $R'$-step, we observe that the block-value assignment is irrespective of the specific vertex label - it only depends on vertex appearance status. Therefore, the corresponding $R'$-step receives the same value as the $R$-step, and our proposition holds again. 
	
	The extra factor of $c_\tau$ is added as we picked up two factors of $c_\tau$ in these two corresponding blocks for the trace moment calculation of the weighted linear combination.
\end{proof}
Summing over all $\tau\in \calB(K)$ gives us the following desired bound.
\begin{corollary}
	\[ 
	B_q(\mathsf{Ideal}) = \sum_{\tau \in \calB(K)} B_q(\mathsf{Ideal}(\tau) ) = \sum_{\tau\in\calB(K)} (c_\tau)^2 \cdot \Var(\tau) = \Var(K)
	\]
\end{corollary}

Next, to facilitate our accounting of block-value gaps, we introduce the following quantity that helps us keep track of each block-step that incurs a gap in the corresponding block-value does not get its gap assigned to \emph{too many} other block-steps.
\begin{definition}[Gap Witness-Vertex]
	For any block-step $\tau_i$, let $S$ be the set of separator vertices in $V(\tau_i)$, i.e. the vertices making a middle appearance in $\tau_i$, when $|S|>1$, we call each vertex $v\in S$ a gap witness-vertex for block-step $i$.
\end{definition}

	It is crucial to notice a block-step only incurs gap witness vertices if it contains more than $1$ vertex in the separator, corresponding to when we expect to incur a gap in the block-value.

\paragraph{Finding Slack for Mismatched Block-Steps}

We are now ready to introduce the core of our argument in matching $F,R$ steps- finding some block-steps with a gap for each non-ideal $R$-steps. Notice that we assume all vertices except one - the vertex in $V_t$ - are making their final appearance throughout this section, as otherwise such a block-step is either not an $R$-step, or locally comes with a block-value gap in itself.

\begin{definition}[Vertex Weight]\label{def:vtx-wt} 
	We assign each circle vertex a weight of $1$ and a square vertex of weight $2$. For a set of vertices $S$, we define its weight as the sum of its vertex weights.
	
	Additionally, as an edge-case, we assign a circle vertex in $\beta$-gadget a weight of $2$ as opposed to $1$. 
\end{definition}

The crux of our analysis is the following combinatorial argument that helps us identify extra vertices making middle appearance (i.e. extra separator vertices) for two block-steps sharing any edge.

\begin{claim}[Structural Property of Backbone Dangling Shapes] \label{clm: structural-clm}
	For any backbone dangling shape, and any $2$ edge-coloring, there is a weight $\geq 2$ separator where a separator is defined as the set of vertices incident to edges of both colors. Moreover, in the case the separator contains $2$ circle vertices, it does not contain circle vertex of any final $h_2$ attachment.
\end{claim}

\begin{proof}
	Since the reduced shape $\tau$ is a single connected component and not all edges have the same color, there exists at least one vertex incident to both edges. Moreover, consider the case that this is a circle vertex, as otherwise we have found a weight-$2$ separator.
	
	Next, we observe that any such vertex cannot be a boundary vertex in $U_\tau $ or $V_\tau $, nor a circle vertex in any final attachment gadget as well, as any such vertex is also degree-$1$ (crucially we view $h_2$ edge as a single edge in $\tau$). Therefore, we may start with the circle vertex promised, let it be $c$, and observe that it must be part of some $\al$ and $\beta$ gadget. 
	
	Suppose it is part of a $\beta$ gadget, then we are again done as this is a single circle vertex that has weight $2$ by our rule in~\cref{def:vtx-wt}. On the other hand, it is part of an $\al$-gadget, and it is straightforward to observe that there must be another vertex (possibly square) that is incident to edges of both color. To see this, note that we have two edge-disjoint paths at $c$ (i.e., a cycle) and the edges incident to $c$ are additionally of different colors.
 \end{proof}
 
We now use the above structural property to identify slack in block-steps.

\begin{proposition}[Simultaneous Appearance] \label{prop:simultaneous-appearance} Given an $L$-block-step at block-step $i$, the edge(s) in the current block-step must have appeared altogether (potentially with some other edge) in a previous $F$-step, or there is a local-gap (potentially from a previous step).

 In other words, unless there is a local-gap (potentially from a previous step), there is a block-step $t<i$ s.t. $E(\tau_i) \subseteq E(\tau_t)$, and moreover, $\tau_t$ is an $F$ block-step.
 
 Formally, at least one vertex making a final appearance at the current block-step $i$ is a block-witness vertex for some previous block-step, particularly step-$t$.
\end{proposition}

\begin{proof} \label{proof: simul-appearance}
Let $\tau_i$ and $\tau_t$ be the shape of the block-steps at each time. Note that we may have $\tau_i\neq \tau_t$, and moreover, $\tau_i, \tau_t$ comes with underlying edges labeled in $[n]$ as well.
 In the graph of $\tau_t$, color any labeled edge in $E(\tau_t)\cap E(\tau_i)$ that are making $F$ appearance in the latest (most recent) block-step $t$,  red, and color any other edge blue.

%
	

	Apply \cref{clm: structural-clm} on any red edge-component $C$ in $\tau_t$,  unless all edges are red and step-$t$ satisfies the simultaneous-appearance property for block-step $i$,  $C$ has at least two
distinct vertices that are incident to some blue edge otherwise.

Finally, notice any red-component vertex incident to some blue edge is a vertex making a middle appearance at block-step $t$, therefore at least two vertices are making middle appearance at time-step $t$. That is the block-step with a gap in block-value as desired. Importantly, observe the following,
\begin{enumerate}
	\item  Both blue vertices are gap witness vertices for block-step $t<i$ by construction (each middle appearance vertex is a block-value witness of the step when there is more than $1$ middle-appearance vertex) ;
	\item At least one of the blue vertices is making a final appearance at step-$i$ (since this is an $L$-step, there is at most one-vertex not making a final appearance);
	\item Therefore, we can conclude there is at least one final-appearance vertex in the current block that is a block-witness vertex for some previous step (in particular, block-step $t$).
\end{enumerate}
	 \end{proof}

\begin{proposition}[Simultaneous Closure]\label{prop:simul-closure}
	Given an $F$-block-step with edge-set $E(\tau_i)$, unless all edges are simultaneously closed, there is a local gap at the first subsequent block-step $t>i$ that uses some edge in $E(\tau_i)$. 
	
	Formally, at least one vertex making first appearance at step-$i$ is a gap witness-vertex for some block-step at the walk, specifically at step-$t$.
	\end{proposition}
\begin{proof}
	Suppose not all of $E(\tau_i)$ are closed (i.e. appearing for the final time) in the subsequent block-step that closes some edge in $E(\tau_i)$. Consider the subsequent block-step $\tau_t$ that uses some edge in $E(\tau_i)$, our goal is to identify that there are at least two vertices in $\tau_t$ making a middle appearance, forming a local gap in the block-value - these two vertices will be our candidates for a gap witness-vertex for block-step $i$. Since at most one of them is not making a \emph{first} appearance at the current step $i$, this proves the desired.
	
	In $\tau_i$, color the edges that appear in $E(\tau_i)$ red, and the remaining edges blue. Apply \cref{clm: structural-clm} on the red-component, it has at least two distinct vertices incident to some blue edge in $\tau_i$. By assumption, any blue-edge is unclosed since each makes their first appearance at $\tau_i$, and remains unclosed at $\tau_t$ - hence, they are bound to make future appearance again. Finally, notice that the two incidental vertices identified also appear in $V(\tau_t)$, and since each is incident to some unclosed edge bound to appear, each is making a middle appearance. This completes our proof that $\tau_t$ is a slack step.%
%
%
%
	
\end{proof}

\paragraph{One Gap to Rule All the Mismatching Coefficients from Individual Block-Step }

Finally, we exploit the definition of gap-witness vertex to make our identified slack-step quantified.

\begin{observation}
	Each gap witness-vertex is only assigned to a block-step in which it appears for the first or final time by~\cref{prop:simultaneous-appearance} and~\cref{prop:simul-closure}. In other words, each gap witness vertex is in charge of at most $2$ block-steps for violations of simultaneous appearance or closure.
\end{observation}

\paragraph{Formal Bounds for Linear Combination}

We now give a formal description of how coefficients are assigned to each block-step, and complete the proof for the block-value bound for linear combination $M$. We apply the following rule to assign coefficient from the weighted linear combination in the block-value assignment:
\begin{enumerate}
	\item For any ideal-$F/R$-block and any $S$-blcok-step, assign coefficient $c_\tau$ to the same block;
	\item For any non-ideal $F$-step of shape $\tau$, assign the coefficient $c_\tau$ to the block-step-$t$ identified via~\cref{prop:simul-closure};
	\item For any non-ideal $R$-step of shape $\tau$,\begin{itemize}
		\item  If its underlying edge-set violates the simultaneous-appearance property, assign the coefficient $c_\tau$ to block-step $t$ identified in~\cref{prop:simultaneous-appearance};
	\item Otherwise, its underlying edge-set comes from some $F$-step that violates  the ``simultaneous closure'' property, assign the coefficient $c_\tau$ to the block-step-$t$ identified via~\cref{prop:simul-closure}.
	\end{itemize}
\end{enumerate}

\begin{claim}
	For any $F$ block-step $i$ that violates simultaneous-closure property, its associated coefficient assigned from the block-value assignment scheme is at most  \[ 
	\sum_{\tau_1,...\tau_r: \text{assigned to step-}i } \prod_{\tau_i} |c(\tau_i)| \leq \left( \sum_{\tau \in \calB(K)}|c(\tau)| \right)^{|V(\tau_i)|}  =  \|c(K)\|_1^{  |V(\tau_i)|/2}
	\]
\end{claim}

\begin{corollary}Each non-ideal block-step gets assigned a coefficient factor of at most \[ 
\|c(K)\|_1^{  |V(\tau_i)|/2}\cdot \|c(K)\|_1^{|V(\tau_i)|} \leq \|c(K)\|_1^{2|V(\tau_i)|}
\,.\]
\end{corollary}

\paragraph{Bounds for Non-Ideal Steps}
Next, we verify the non-ideal shapes.
\begin{proposition} \label{prop:block0non-ideal-M}
	For any non-ideal step it gets assigned a bound at most \[ 
	B_q(\mathsf{NonIdeal})  \leq O(\frac{1}{\sqrt{d}}) \cdot  (3\|c(K)\|_1)^{ 3\cdot D_V} \cdot (q\cdot D_V)^2 \,, 	\]
	where $D_V \leq O(\log n) $ is the size-limit s.t. $|V(\tau)|\leq D_V$ for any $\tau\in \calB(K)$.
\end{proposition}
\begin{proof}
	For each non-ideal step of shape $\tau$ with weight$>1$ separator, we have a block-value of
	\begin{align*}
		B_q(\tau) &\leq \sum_{2\leq v\leq |V(\tau)|} 3^{|V(\tau)|} \cdot   c_\tau  \cdot \sqrt{n}^{|V(\tau)|-v }\cdot \|c(K)\|_1^{  2|V(\tau_i)|}  \cdot  (q\cdot D_V)^{v}\\
	\end{align*}
	where we note 
	\begin{enumerate}
		\item $3^{|V(\tau)|}$ is a trivial bound specifying vertex appearance of any vertex in a given $\tau$;
		\item $c_\tau = c(\tau) \cdot c_n(\tau) = c(\tau) \cdot (\frac{1}{\sqrt{n}})^{|V(\tau)|-1}$ is the coefficient for shape $\tau$ in $K$;
		\item Any non-slack vertex outside the separator gives factor $\gam<1$ under our edge-charging scheme; 
		\item Each slack block comes with at least an $\frac{1}{\sqrt{d}}$ for extra separator vertex;
		\item $(q\cdot D_V)^{v}$ is an upper bound for the factor restricted to vertices making middle appearances;
		\item $\|c(M)\|_1^{  2|V(\tau_i)|} $ is an upper bound for coefficient factor assigned to this step.
	\end{enumerate}
	Provided $D_V\leq O(\log n)$ ,the above series is dominated by $v=2$, giving us a bound of \[ 
	B_q(\tau) \leq O(1) \cdot 3^{|V(\tau)|} \cdot   c(\tau) \cdot    \frac{1}{\sqrt{n}} \cdot \|c(K)\|_1^{  2|V(\tau_i)|}  \cdot (q\cdot D_V)^{2}
	\]
	
	Summing over all possible shapes $\tau$ gives us
	\begin{align*}
		B_q(\mathsf{NonIdeal})  &\leq O(1) \cdot \sum_{\tau} \frac{1}{\sqrt{d}} \cdot  3^{D_V} \cdot |c(\tau)| \cdot (3\|c(K) \|_1)^{  |V(\tau_i)|2} (q \cdot D_V)^2\\
		&\leq O(\frac{1}{\sqrt{d}}) \cdot  (3\|c(K)\|_1)^{ 3\cdot D_V} \cdot (q\cdot D_V)^2 
	\end{align*}
 \end{proof}
 
We are now ready to wrap up this section by proving our main theorem for the spectral radius of a weighted linear combination of backbone-dangling shapes.

\begin{proof}[Proof to~\cref{thm:Q-norm}]
	We wrap up the proof for block-value bound. For each step, there are two cases by our factor assignment scheme- ideal-$F$ and ideal-$R$, each is assigned a value of $ \sqrt{\Var(K)}.$  Any non-ideal step gives block-value captured by our auxiliary function $B_q(\text{Non-Ideal})$. Therefore, we have \[ 
	\E[\Tr(KK^\top)^{q}] \leq d \cdot (2\sqrt{\Var(K)} + B_q(\text{Non-Ideal}))^{2q}
	\]
 	 	The proof for norm bound then follows from \cref{claim:trace-to-norm-rough}.
\end{proof}

\subsection{Chebyshev Polynomial and Horizontal Concatenations}
Throughout this section, we use $K$ to denote an arbitrary linear combination of backbone-dangling shapes.

\begin{theorem}[Orthogonal Polynomials and Shape Concatenation for Semicircular Distributions (Formal)]  
\label{thm:SC-concate}
For any linear combination of backbone-dangling shapes \[
K = \sum_{\calB(K)} c_\tau \cdot \cm_\tau \,,
 \]
	from a ground set $\calB(K)$ of size $|\calB(K)|$ and each shape has at most $D_V$ vertices. Then, for any $t>0$, we have \[ 
	P^{\mathsf{sc}}_t(K) = \fp_t(K) + O( \frac{\poly(t, D_V) }{\sqrt{d}} )
	\]
	Quantitatively, we have \[ 
	P^{\mathsf{sc}}_t(K) = \fp_t(K) + \text{SC-Error}(t)
	\]
	for \[
	\|\text{SC-Error}(t)\|_{sp} \leq O(2^{t}) \cdot B_q(\text{non-ideal})	\]
	with $B_q(\text{non-ideal})$ the auxiliary function defined from~\cref{thm:Q-norm}.
\end{theorem}

\paragraph{Key Ingredient for Cancellation}

The overall structure is identical to analysis of orthogonal polynomial approximation for $A$, except now we use the basis with respect to the semicircular distribution with variance-$1$ and the recurrence now follows by Chebyshev polynomial of the second kind.

The main idea is the following cancellation lemma that shows $K^2$ is either proper concatenation or the identity matrix which shall be canceled out by the Chebyshev recurrence relation.
\begin{lemma}[Backtracking Residue for SC Polynomials]  
\[ 
 K^2 = \sum_{\tau \in \calP_2(K) } c_\tau \cdot \cm_\tau + I_d + o_d(1)\,, 
\]
provided \[ 
\Var(M) = 1\,.
\]
\end{lemma}

To shed light on the above lemma, it is crucial for us to highlight a phenomenon unique to the cancellation for $Q$ that is rather distinctive from that of the prior analysis for the Marchenko-Pastur distribution. Recall that in the prior analysis, the cancellation is made possible by backtracking intersections; in particular, this happens for any two consecutive steps with the exact same edge set, as shown in \cref{prop:diamond-cancellation} and analogously in \cref{prop:half-diamond-cancellation}. These two criteria may be viewed as combinatorial interpretations of the dominant terms in a length-$2$ walk (e.g., $\Tr(AA^\top )$ for diamond cancellations). This is formally proven via~\cref{prop:backtracking-sc}.

\subsection{Formal Analysis for Chebyshev Polynomials and Concatenations}
We now give a formal analysis for the error terms that arise when we switch from Chebyshev polynomials of $R$ to graph matrices of concatenated shapes in $R$.

\label{sec:chebyshev-Q}

Firstly, we formally define well-behaved intersections at the heart of the cancellation.
\begin{definition}[Well-behaved Backtracking Intersections for SC Polynomials]
	Given $\tau \in \calB(K)$ that are backbone-dangling shapes, we say the composition $\tau\cdot \tau'$ forms a backtracking intersection if \begin{enumerate}
		\item $U_{\tau} \equiv V_{\tau'} $: the left-boundary vertex of $U_\tau$ matches with the right-boundary vertex of $V_{\tau'}$ ;
		\item $E(\tau) \equiv E(\tau')$: the underlying edge-set of the two sets receive the same labels.  
	\end{enumerate}
	We denote the corresponding backtracking intersection pattern as $\bti(\tau)$.
\end{definition}
The crux of the above observation is the following proposition that shows backtracking intersections gives a copy of identity, enabling the key cancellation for Chebyshev polynomials. 

\begin{proposition}[Backtracking Residual for SC Polynomials] \label{prop:backtracking-sc}
	Define \[ 
	\Delta_{\bti} \coloneqq  \sum _{\tau \in \calB(K)} (c_\tau)^2 \cdot  \cm_{\bti(\tau)} - I_d \,,	\]
	we have \[ 
	\|\Delta_{\bti}\|_{sp} = o_d(1)\,.
	\]
	provided $\Var(K)=1$.
\end{proposition} 
\begin{proof}[Proof Sketch]
For any shape $\tau$, we have \[ 
c_\tau^2 \cdot \cm_{\bti(\tau)}  = c_\tau^2 \cdot  \Var(\tau) \cdot I_d -o_d(1)
\]	
with the discrepancy from falling factorial of the vertices while we assign each a full factor of $d$ and $m$ and the fluctuation of second moment Gaussian random variables from their mean. 

Summing over all $\tau\in \calB(K)$ gives the desired.
 We defer the formal proof to \cref{app:def-Q}.
\end{proof}

\begin{definition}[Local-Collision Pieces for Chebyshev Polynomials (SC-LCP)]
\label{def:SC-LCP}
Fix a concatenation record
\[
    \tau_t\cdot \tau_{t-1}\cdot \dots \cdot \tau_1,
\]
A consecutive sub-record
\[
    \Theta
    =
    \tau_b\cdot \tau_{b-1}\cdot \dots \cdot \tau_a
\]
is called a \emph{local-collision piece} for Chebyshev polynomials if it is an MP-LCP as in~\cref{def:local-collision-pieces}, except that each constituent shape can be any backbone-dangling shape as in~\cref{def:backbone-dangling-shape}.

\end{definition}

	It should be pointed out that this is the precise translation of \cref{def:local-collision-pieces} for MP-polynomials to the matrix setting via Chebyshev polynomials of the second kind.

\begin{proposition}[Structural Property for Error Terms (Chebyshev)]
\label{prop:chebyshev-error-term} 
[Analog of \cref{prop:error-term-A-equivalence}]
For $K$ a linear combination of backbone dangling shapes with variance $\Var(K)=1$, any error term in $\text{Error}(t) = P_t(K)- \fp_t(K) $
admits an LCP-decomposition.
%
\end{proposition}

\begin{proof}
	The proof is identical to that in \cref{prop:error-term-A-equivalence} except that we only need to consider the full backtracking intersection of the reduced shapes with the corresponding backtracking residual.
	
	We apply the induction from the proof of~\cref{prop:error-term-A-equivalence}. The base cases, the observation
that deleting the leftmost shape preserves an LCP-decomposition, and
the proper-attachment, non-backtracking-collision, and non-isolated
backtracking cases are unchanged after replacing the MP gadget alphabet
by \(\calB(M)\) and MP backtracking by generalized backtracking.

It remains only to modify the isolated-backtracking case, and this is the only difference from the proof of
\cref{prop:error-term-A-equivalence}. In particular, since the
Chebyshev recurrence has no analogue of the term
\(-\gamma P^{\mathrm{sc}}_t(M)\), there are no half-diamond or half-flat
cancellation cases. The induction therefore gives the claimed
SC-LCP-decomposition by the cancellation from~\cref{prop:backtracking-sc}. 
\end{proof}

\paragraph{Reduction to the block-value bound.}
It remains to apply the block-value bound below to the
SC-LCPs appearing in this decomposition.


\begin{lemma}
For a locally-collision-piece $\tau$ of length-$t$, we have the following block-value bound   
\[
B_q(\tau) \leq \prod_{i=1}^{t-1} 2^{t-1}  \cdot  B_q(\text{non-ideal}) \]
for $B_q(\text{non-ideal})$ defined from \cref{thm:Q-norm}.
\end{lemma}
\begin{proof}

Apply the block-value scheme to the non-terminal component $\tau_1\circ\dots \circ \tau_{t-1}$ by viewing $\tau$ as a specific walk of length-$t$ that arises in the trace moment calculation,  we he have, \[
B_q(\text{Non-Terminal} ) \leq  2^{t-1} \,.
 \]
  
 Finally, we notice the the terminal-gadget is a slack block-step in the analysis for \cref{thm:Q-norm}, and it suffices for us to bound it by the auxiliary function of $B_q(\text{non-ideal})$. 
\end{proof}
 
We are now ready to wrap up the proof for the equivalence theorem of Chebyshev polynomials and graph matrices of concatenated shapes.
\begin{proof}[Proof to~\cref{thm:SC-concate}]
	This follows by the the block-value function for any LCP that gives an upper bound for the expected trace, which then yields a norm bound via~\cref{claim:trace-to-norm-rough}.
\end{proof}

\section{Formal Analysis: Inner and Outer Truncation}\label{sec:formal-truncation}

We now give an overview of our analysis by putting in the details to handle truncation terms. On a high level, we proceed in three steps:
\begin{enumerate}
	\item We first understand the effects of truncation on the function $F$: concretely, we show that the variance of the inner matrix after polynomial truncation and early termination of the iterative process is still $1-o_d(1)$ by our choice of parameters;
	\item Next, we observe that our construction is not necessarily bound 
    to the particular choice of $F$. In particular, we show that the above bound can be readily adjusted to an analogous positive function $F_\delta$ to be defined;
	\item Finally, we recall that the inner matrix, and therefore its variance, is a function of the threshold parameter $\gamma$. By adjusting the threshold $\gamma$ with $o_d(1)$ slack, we can ensure the inner matrix has variance-$1$ (for our spectrum estimates to apply) and this concludes the proof of our main theorem.
\end{enumerate}

We first bound the effects of truncation by the following.
%
%

 \paragraph{Truncation of the Inner Matrix $Q$}
 We work with the following parameters for truncations to be chosen. 
\begin{definition}[Parameter $t^*$: Termination of the Iterative Process]
    For $t^*$ to be picked, we will let it denote the termination level for the iterative process. In other words, we will construct a sequence of inner matrices $Q_1, Q_2,\dots  $ up to $Q_{t^*}$. 
\end{definition}

\begin{definition}[Parameter $D$: Inner Truncation for Chebyshev Polynomials] \label{def:inner-trunc}
	We define $D$ to be the truncation parameter for $P(Q_i)$ at each level of $Q_i$. In other words, at each level of $Q_i$, we only consider $P_{j}(Q_i)$ for $j\leq D$. 
\end{definition}

We will follow the recursion until some $t^*$ to be chosen which gives 
us a matrix $Q_{t^*}$. Since we stop the process early, and given the truncation of $P_{j\leq D_j}$ at each level of $Q_i$,  the variance of $Q_{t^*}$ is not exactly $1$ but $1-o_{D,t^*}(1) $ by the following lemma. Crucially, the variance continues to go to $1$ by taking large enough truncation threshold of $t^*$ and $D$.

Let \(\varepsilon_{D,t^*}\) measure the variance loss caused by the two
truncations.
\begin{lemma}[Truncated Variance is still almost $1$ (Lemma~D.11 of \cite{PotechinXu2026Theta})
] 
\label{lem:truncated-var-damage}
Let $Q_{t^*,D}$ be the inner matrix with inner truncation parameter $D$ and termination at iteration level $t^*$; we have
\[
 \Var(Q_{t^*,D} ) = \frac{\gam}{1-\gamma}\left( 1 - \varepsilon_{D,t^*} \right).
\]
for \[ 
\varepsilon_{D,t^*} = O( \frac{1}{(t^*)^2} + \frac{1}{D^2})\,\,.
\]
Notice this quantity is $1-o_{D,t^*}(1)$ at the threshold $\gamma=1/2$.
\end{lemma}

\paragraph{Shifting from a Non-Negative Function to a Positive Function}

Next, we note that our choice of non-negative function $F$ is not necessarily fixed --- we may instead pick a positive function by softening the prior choice of $F$ to enable a tiny $\delta$ positive mass everywhere in the support.
Consider the following positive function $\widetilde{F}_\delta$ defined as \[
   \widetilde{F}_\delta(x) \coloneqq  \max(2x, \delta)\,.
\]

Moreover, we consider the rescaled version $F_\delta$ such that $P_1(x)$ has coefficient $1$ in the Chebyshev polynomial expansion for $F_\delta$. Concretely, we take \[
 \widetilde{F}_\delta(x) = \widetilde{C}_F + \widetilde{a}_1 x + \sum_{t>1} \widetilde{a}_i \cdot P_t(x)
 \]
 and obtain the rescaled version \[ 
 F_{\delta} (x) = C_\delta + x + \sum_{t>1} b^{(\delta)}_t \cdot P_t(x)\,,
 \]
 where we take \(
 C_\delta \coloneqq \frac{\widetilde{C}_F}{\widetilde{a}_1}, b^{(\delta)}_t\coloneqq  \frac{\widetilde{a}_t}{\widetilde{a}_1}\,.
  \)
  Let $\calQ$ be the new inner matrix when we apply the iterative process with $F_\delta$ instead of $F$, with the same initialization condition $\calQ_0 \coloneqq C_\delta \cdot \sum_{i\in[m]} w_0[i] \cdot v_iv_i^T$, the appropriately scaled starting point from the identity perturbation, and $\calQ_{\infty} = \lim_i \calQ_i$.  Moreover, let $\calQ_{t^*,D}$ denote the analogous truncation at level $t^*$ with degree $D$ at each level for the $\calQ$ matrix.

 At this point, we remind the reader that our inner matrix $Q$ (as well as $\calQ$) in the end is a function of the threshold parameter $\gamma$. The next claim shows that by allowing a tiny, vanishing gap from the threshold of $\frac{1}{2}$ (which corresponds to the $m\sim d^2/4$  threshold for ellipsoid fitting), we can ensure the inner matrix has variance-$1$ and the resulting matrix has $\delta$-P.D. mass.
 \begin{claim}[Threshold Shift for P.D. Mass: Truncated Version]
\label{claim:pd-shift-threshold}
There is a choice of
\[
    \gamma_{\delta,D,t^*}
    =
    \frac12-\Theta(\delta)
\]
for which the truncated inner matrix has variance $1$, i.e., \( 
\Var(\calQ_{D,t^*}(\gamma) )= 1\,,
\)
provided \(
    \delta \gg \varepsilon_{D,t^*}.
\)

\end{claim} 
 For simplicity, we will now call the final inner matrix $\calQ^*$ with truncation fixed.
 
 \paragraph{Outer Truncation for $F$}
 With the inner matrix   $\calQ^*$, we proceed to analyze \[ 
 \frac{1}{C_\delta} \cdot  F_\delta(\calQ^*)\,.
 \]
To that end, we consider $F_\delta^{\leq D}$, the degree-$D$ truncation of $F_\delta$, and we have  \[
\frac{1}{C_\delta} \cdot  F^{\leq D}_\delta(\calQ^*) = I + Q^*/C_\delta + \sum^D_{j>2} \frac{b^{(\delta)}_j}{C_\delta} \cdot P_j(\calQ^*)\,.
 \] 
 
By standard results from polynomial approximation, we show that as long as the inner matrix has the anticipated spectrum, the inevitable $o_n(1)$ deviation in the spectral edge as well as the outer truncation for $F_\delta$ does not pose substantial damage to our positive mass of $\delta$ that we have started with.

\begin{claim}\label{clm: outer-truncation}
   For an inner matrix with spectrum  \[ 
 \mathsf{spec}(\calQ^*) \in [-2+ \eps_{\mathrm{sp}} , 2+ \eps_{\mathrm{sp}}  ]
 \]
 for some parameter $\eps_{\mathrm{sp}}$, we have
 	\[ 
 	F_{\delta}^{\leq D}(\calQ^*) \succeq \left( \delta - \tilde{O}\!\left(\frac{1}{D} + \eps_{\mathrm{sp}} \cdot \poly(D)\cdot \exp(D\cdot  \eps_{\mathrm{sp}})\right)\cdot  \right) I.
 	\]
\end{claim}

\subsection{Shifting to a Positive Function with Truncation}

We now justify the threshold shift after replacing the nonnegative activation
\(F(x)=2x_+\) by a strictly positive perturbation in a 3-step argument.
\begin{enumerate}
	\item  First, we recall how the variance recurrence identifies the relevant
coefficient mass;
\item  Then, we show that this coefficient mass increases by
\(\Theta(\delta)\) after shifting to \(F_\delta\);
\item Finally, we incorporate the
finite iteration and inner-degree truncation errors and show that \(\gamma_{\delta,D,t^*}\) can be chosen so that the truncated inner matrix has variance \(1\);
\end{enumerate} 

\paragraph{Step-1: The variance mass.}

Consider a function
\(
    G(x)=C_G+x+\sum_{j\ge2} a_j P_j(x),
\)
whose \(P_1(x)=x\) coefficient is normalized to be \(1\).  In the ideal,
untruncated iteration, if \(S_k\) denotes the variance of the inner matrix at
iteration \(k\), then the variance recursion has the form
\[
    S_{k+1}
    =
    \frac{\gamma}{1-\gamma}
    \left(
        C_G^2+\sum_{j\ge2} a_j^2 S_k^j
    \right).
\]
Therefore, for \(S=1\) to be a fixed point of the ideal recursion, it is
sufficient to take
\[
    \frac{\gamma}{1-\gamma}
    \left(
        C_G^2+\sum_{j\ge2}a_j^2
    \right)
    =
    1.
\]
This motivates the definition
\(
    A_G
    :=
    C_G^2+\sum_{j\ge2}a_j^2.
\)
As a sanity check, for the original function $F$,  Parseval gives  \(A_F=1\), and hence
\(
    \gamma_F=\frac{1}{1+A_F}=\frac12\,,
\)
 recovering the threshold in the ideal calculation.

\paragraph{Step-2: Coefficient mass shift under \(F_\delta\).}

We now replace \(F\) by a strictly positive perturbation
\[
    \widetilde F_\delta(x):=\max(2x,\delta).
\]
Write its Chebyshev expansion as
\(
    \widetilde F_\delta(x)
    =
    \widetilde C_F+\widetilde a_1x+\sum_{j\ge2}\widetilde a_jP_j(x).
\)
We then rescale by the \(P_1\)-coefficient and define
\(
    F_\delta(x)
    :=
    \frac{\widetilde F_\delta(x)}{\widetilde a_1}
    =
    C_\delta+x+\sum_{j\ge2}b_j^{(\delta)}P_j(x),
\)
where
\(
    C_\delta:=\frac{\widetilde C_F}{\widetilde a_1},
       b_j^{(\delta)}:=\frac{\widetilde a_j}{\widetilde a_1}.
\)
The relevant variance mass for \(F_\delta\) is
\[
    A_\delta
    :=
    C_\delta^2+\sum_{j\ge2}(b_j^{(\delta)})^2.
\]
We bound the change of $A_\delta$ via the following claim.
\begin{claim}[Coefficient Mass Shift]
\label{claim:A-delta-shift}
For sufficiently small \(\delta>0\),
\[
    A_\delta=1+\Theta(\delta).
\]
\end{claim}

\paragraph{Step-3: Incorporating truncation.}

We now return to the truncated construction. Recall that \(
    \varepsilon_{D,t^*}
    :=O(
    \frac{1}{D^2}+\frac{1}{(t^*)^2})
\)
denotes the variance damage coming from the inner Chebyshev truncation degree
\(D\) and the finite iteration level \(t^*\). With these ingredients in hand, we are ready to wrap up the proof of the following claim. This shows that with the truncated construction for the inner matrix, we can pick a $\gamma$ (which ultimately governs the threshold of $m$ we obtain) to ensure the inner matrix still has variance-$1$, albeit at a $\delta$-cost away from the threshold.

\begin{claim}[Threshold Shift for P.D. Mass: Truncated Version (Restatement of \cref{claim:pd-shift-threshold})]
There is a choice
\[
  \gamma_{\delta,D,t^*}
    =
    \frac12-\Theta(\delta)
\]
for which the truncated inner matrix produced by the \(F_\delta\)-based
construction has variance
\[
    \Var\!\left(\calQ_{D,t^*}(\gamma_{\delta,D,t^*})\right)=1
\]
provided 
\[
    \delta\gg \varepsilon_{D,t^*}.
\]
For convenience, we write $\gam_\delta \coloneqq   \gamma_{\delta,D,t^*}$ when the dependence is clear.
\end{claim}
As a reminder, we will eventually take $\delta=o_d(1)$, so this cost is essentially harmless for our result.

\begin{proof}
By \cref{claim:A-delta-shift}, the untruncated variance mass associated with
\(F_\delta\) is
\[
    A_\delta=1+\Theta(\delta).
\]

Let \(A_{\delta,D,t^*}\) denote the effective variance mass retained after
truncating the \(F_\delta\)-based construction at inner degree \(D\) and
iteration depth \(t^*\).  \Cref{lem:truncated-var-damage} gives
\[
    A_{\delta,D,t^*}
    =
    A_\delta - O(\varepsilon_{D,t^*}).
\]
i.e., \[
    A_{\delta,D,t^*}
    =
    1+\Theta(\delta) - O(\varepsilon_{D,t^*}) = 1+\Theta(\delta)
\]
for \(\delta\gg \varepsilon_{D,t^*}\).
For any \(D,t^*\), the variance of the truncated inner matrix has the form
\[
    \Var\!\left(\calQ_{D,t^*}(\gamma)\right)
    =
    \frac{\gamma}{1-\gamma}\,A_{\delta,D,t^*}.
\]
Thus, to make the variance equal to \(1\), it suffices to choose
\(\gamma=\gamma_{\delta,D,t^*}\) such that
\[
    \gamma_{\delta,D,t^*}
    =
    \frac{1}{1+A_{\delta,D,t^*}}.
\]
Using \(A_{\delta,D,t^*}=1+\Theta(\delta)\), we conclude that
\[
    \gamma_{\delta,D,t^*}
    =
    \frac{1}{2+\Theta(\delta)}
    =
    \frac12-\Theta(\delta).
\]
\end{proof}

%

\subsection{Bounds on the Truncation Error for \texorpdfstring{$A^{-1}$}{1/A}}
To avoid infinite series in our construction of the ellipsoid matrix, we would also need to truncate $A$. That said, explicit control on the spectrum of $A$ in turn allows us to obtain a reasonable truncation. This is the focus of this section.

There are two sources of error as we transfer from an ideal-world proof to a real-world proof:\begin{enumerate}
	\item The truncation error incurred by truncating $A^{-1}$ at degree-$D$. 	\item The edge-deviation error: our chosen polynomial expansion is only valid for the interval of $[(1-\sqrt{\gam})^2, (1+\sqrt{\gam})^2] $ while our instantiated matrix $A$ may further have inevitable $\varepsilon_A= o_d(1)$ fluctuation at the edge. \end{enumerate}
It can be shown that the second type of error term dominates in our regime of interest; we show the following.

\begin{restatable}[Truncation error under spectral-edge fluctuation]{lemma}{mpinverseedgefluctuation} \label{lem:mp-inverse-edge-thickening} Suppose \[ \operatorname{spec}(A) \subseteq \bigl[(1-\sqrt{\gam})^2-\varepsilon_A,\, (1+\sqrt{\gam})^2+\varepsilon_A\bigr] \] for some \(\varepsilon_A>0\), and assume \[ (1-\sqrt{\gam})^2-\varepsilon_A>0. \]  We have\[ \left\| A^{-1}-(A^{-1})_{\le D} \right\|_{sp} \le O_\gam\!\left(\poly(D) (\sqrt{\gam})^{D+1}\right) + O_\gam\!\left( \varepsilon_A \cdot  \poly(D) \cdot  e^{O_\gam(D\sqrt{\varepsilon_A})} \right). \] \end{restatable}
We prove this by standard techniques from scalar polynomial approximation in~\cref{sec:defer-details-truncation}.

\subsection{Bounds on Truncation Error for Inner and Outer Truncation}
Let $\calQ^*= \calQ_{D,t^*}(\gamma_{\delta,D,t^*})$ be the final inner matrix at the chosen threshold such that the inner matrix has variance-$1$ after inner truncation.
\begin{restatable}[Early-termination correction]{lemma}{lemEarlyTerminationCorrection}
\label{lem:early-termination-correction}
There exists a matrix \(D_E\) such that
\(
    \frac{1}{C_\delta}F_\delta^{\le D}(\calQ^*)+D_E
\)
satisfies the affine constraints exactly; namely,
\[
    v_i^\top
    \left(
        \frac{1}{C_\delta}F_\delta^{\le D}(\calQ^*)+D_E
    \right)
    v_i
    =
    1
    \qquad \text{for all } i\in[m].
\]
Moreover,
\[
    \|D_E\|_{sp}
    \le
    O\!\left(\sqrt{\varepsilon_{D,t^*}}\right).
\]
\end{restatable}

\begin{proof}
	
Let's start by identifying where error terms can arise. On a high level, other than terms in the initial $Q_0$, each other shape in the final polynomial expansion is intended to pair with one other shape so that they group into a term in the kernel of the affine constraints, in other words, \[
 v_i^\top (c_\tau \cdot \cm_\tau + c_{\corr(\tau)} \cdot \cm_{\corr(\tau)}) v_i = 0
\,.\]
That said, notice that as we pick the polynomial expansion threshold to be $D$ for the inner truncation as we design the update for each iteration, as well as the final outer truncation for $F_{\leq D}(\calQ^*)$, there is no missing term due to the inner truncation.  Therefore, the only error term is incurred by the early termination as we terminate the iterative process at some finite level.

Next, we observe that, at the chosen $\gam^*$ s.t. the truncated inner matrix has variance-$1$, the untruncated inner matrix has corresponding variance \(
\frac{\gam^*}{1-\gam^*} \cdot A_\delta \,.
 \)
Recall that $\gam^*$ is chosen as \(
\gam^* = \frac{1}{1+ A_{\delta,D,t^*}} 
 \) with \( A_{\delta,D,t^*} = A_\delta - O(\varepsilon_{D,t^*}) \); we have

\[
    \frac{\gam^*}{1-\gam^*}\cdot A_\delta
    =
    \frac{A_\delta}{A_{\delta,D,t^*}}
    =
    \frac{A_\delta}{A_\delta-\Delta_{\delta,D,t^*}}
    =
    1+
    \frac{\Delta_{\delta,D,t^*}}
         {A_\delta-\Delta_{\delta,D,t^*}}
    =
    1+O(\varepsilon_{D,t^*})\,.
\]
Thus, this gives us a bound on the variance of the terms removed by early truncation, \[ 
\text{Early-Termination-Error} \leq O(\varepsilon_{D,t^*})
\]

Therefore, combining with our matrix norm bounds which apply as each missing correction term is a backbone-dangling term, and satisfies the size constraint $D_V$, \begin{align*}
	\|\sum_{\tau: \text{missing correction term}} c_{\corr(\tau)} \cdot \cm_{\tau} \|_{sp}  &\leq (2+o_d(1)) \cdot \Var( \sum_{\tau: \text{missing correction term}} c_{\corr(\tau)} \cdot \cm_{\tau}) \\
	&\leq (2+o_d(1))\cdot   \sqrt{\text{Early-Termination-Error}} \\
	&=  O(\sqrt{\varepsilon_{D,t^*}}) \,.
\end{align*} 
\end{proof}

\subsection{Choice of Parameters and Proof of Main Theorems}
\label{sec:parameter-instantiation}

For convenience, we use the same degree \(D\) for the inner and outer Chebyshev
truncations and for the MP-polynomial expansion of \(A^{-1}\). Fix a
sufficiently small absolute constant \(a>0\), and set
\[
    D=t^\ast=\left\lceil(\log \log d)^a\right\rceil,
    \qquad
    q=\left\lceil(\log d)^4\right\rceil,
    \qquad
    \delta=D^{-1/2}.
\]

\paragraph{Verification of the norm-bound regime.}
To get started, we verify the following size bound on the shapes that arise in the truncated process to ensure our main theorem for matrix norm bounds applies.
\begin{claim}[Size of shapes in the truncated construction]
\label{claim:shape-size-truncated-construction}
Suppose all Chebyshev and MP expansions are truncated at degree \(D\).
Then every shape appearing in \(Q_i\) has at most
\[
    D^{i+O(1)}
\]
vertices. Consequently, every shape in the final degree-\(D\) polynomial
or terminal correction has at most
\[
    D_V\leq D^{t^\ast+O(1)}.
\]
In particular, if \(D=t^\ast=(\log d)^a\) for any fixed \(a<1\), then
\(D_V=d^{o(1)}\).
\end{claim}

\begin{proof}
Let \(V_i\) be the maximum shape size in \(Q_i\). The initialization
contains only constantly many MP paths of length at most \(D\), so
\(V_0=O(D)\). Each update concatenates at most \(D\) previous shapes and
adds \(O(D)\) vertices through the correction, giving
\[
    V_{i+1}\leq DV_i+O(D).
\]
Thus \(V_i\leq D^{i+O(1)}\); the final polynomial and terminal correction
change the exponent only by \(O(1)\). Finally,
\[
    \log D_V=O(t^\ast\log D)=o(\log d),
\]
which implies \(D_V=d^{o(1)}\).
\end{proof}

\begin{claim}[Shape Count Bound] \label{clm:shape-count}
	There are at most $c^{D_V}$ backbone-dangling shapes of $D_V$ vertices for some constant $c>1$.
\end{claim}

\begin{proof}
Encode a shape by its depth-first construction. Each new vertex is
introduced in one of \(O(1)\) ways: through an
\(\alpha\)-, \(\beta\)-, or \(M_D\)-gadget, a deviation attachment,
or a horizontal-concatenation operation. The active attachment site
is implicit in the traversal, and orientations and opening/closing
symbols enlarge the alphabet only by a constant factor. Thus a
\(T\)-vertex shape has an injective \(O(T)\)-length encoding over a
fixed finite alphabet, giving at most \(C^T\) possibilities.
\end{proof}

Since each shape has coefficient at most $1$, we can therefore deduce a bound on $|c(Q^*)|$.
\begin{corollary}
	Let $|c(Q^*)| = \sum_{\tau \in \calB(Q^*)} c_\tau $, we have \[ 
	|c(Q^*)|  \leq O(C^{D^D}).
	\]
\end{corollary}

\begin{remark}
For every fixed \(\delta>0\), \(C^{D^D}\leq d^\delta\) provided
\[
    D\leq c_{C,\delta}
    \frac{\log\log d}{\log\log\log d}.
\]
\end{remark}

Since \(q=(\log d)^4\), we also have \(qD_V=d^{o(1)}\).
Thus the polynomial losses in the MP-LCP, SC-LCP, and
vertical-attachment bounds are absorbed by their \(d^{-\Omega(1)}\)
gains. Together with \cref{thm:Q-norm} and
\cref{cor:trace-to-norm-quant}, this gives
\[
    \eps_{\mathrm{sp}}
    =
    O\!\left(\frac1{\log^2 d}\right)+d^{-1/2+o(1)},
    \qquad
    \mathsf{spec}(Q^\ast)
    \subseteq
    [-2-\eps_{\mathrm{sp}},\,2+\eps_{\mathrm{sp}}].
\]
In particular, \(D\sqrt{\eps_{\mathrm{sp}}}=o(1)\).
 By \cref{lem:truncated-var-damage},
\(
    \varepsilon_{D,t^\ast}
    =
    O\!\left(\frac1{D^2}+\frac1{(t^\ast)^2}\right)
    =
    O(D^{-2})
    =
    O(\delta^4).
\)
Let \(C>0\) be a fixed constant dominating the polynomial dependence on
\(D\) in \cref{lem:mp-inverse-edge-thickening} and
\cref{claim:pd-shift-threshold}. Since \(D=(\log \log d)^a\) for some $a>0$  and
\(\eps_{\mathrm{sp}}=O(\log^{-2}d)+d^{-1/2+o(1)}\), we have
\[
    D\sqrt{\eps_{\mathrm{sp}}}=o(1),
    \qquad
    \eps_{\mathrm{sp}}D^C=o(D^{-2}).
\]
Consequently, \cref{lem:mp-inverse-edge-thickening} implies
that the degree-\(D\) truncation error for \(A^{-1}\) is \(o(D^{-2})\),
while \cref{clm: outer-truncation} gives an outer-truncation error of
\(o(\delta)\).

Finally, the early-termination correction from
\Cref{lem:early-termination-correction} satisfies
\[
    \|D_E\|_{\mathrm{sp}}
    =
    O\!\left(\sqrt{\varepsilon_{D,t^\ast}}\right)
    =
    O(D^{-1})
    =
    o(\delta).
\]
Thus the positive spectral floor \(\delta\) dominates every truncation
and termination error.

\paragraph{Wrapping up the proof of \cref{thm:ellipsoid-thm}.}
By~\cref{claim:A-delta-shift} and~\cref{claim:pd-shift-threshold},
we may choose
\(
    \gamma_{\delta,D,t^\ast}
    =
    \frac12-\Theta(\delta)
\)
so that the truncated inner matrix \(Q^\ast\) has variance \(1\).
The preceding spectral and outer-truncation bounds give
\[
    \frac1{C_\delta}F_\delta^{\leq D}(Q^\ast)
    \succeq
    \left(\frac{\delta}{C_\delta}-o(\delta)\right)I.
\]
Moreover, \Cref{lem:early-termination-correction} provides \(D_E\) such
that
\(
    \Lambda
    \coloneqq
    \frac1{C_\delta}F_\delta^{\leq D}(Q^\ast)+D_E
\)
satisfies the affine constraints exactly. Since
\(\|D_E\|_{\mathrm{sp}}=o(\delta)\), we have
\(
    \Lambda\succeq \Omega(\delta) \cdot I.
\)
Thus, this proves feasibility for
\[
    m
    = \left(1-\Theta(\delta)\right) \cdot   
    \frac{d^2}{4}.
\]
%
%
%
%
 
 \paragraph{Wrapping up the proof of
\cref{thm:ellipsoid-refutation-thm}.}
The dual correction differs from the primal correction only by the
bounded factor
\[
    \operatorname{correct}_{\mathrm{dual}}(H)
    =
    -\frac{1-\gamma}{\gamma}
    \operatorname{correct}_{\mathrm{primal}}(H).
\]
Hence the preceding shape-size, norm, and truncation estimates apply
unchanged to the dual construction.

Take
\[
    \gamma_+
    =
    \frac12+\Theta(\delta),
    \qquad
    \rho
    \coloneqq
    \frac{1-\gamma_+}{\gamma_+},
\]
and choose \(c_R\) according to the dual variance calculation so that
the truncated inner matrix \(W^\ast\) has variance \(1\). By
\cref{sec:refutation},
\[
    c_R-\rho=\Theta(\delta).
\]
Combining the above, we conclude that $\Lambda_R\succeq 0 $ and $\Lambda_R \in \calS=\operatorname{Im}(\calL^*)$.

Next, we verify the objective value. Finally, write
\[
    \Lambda_R
    =
    I-\Pi^\perp I+c_R\calR+Z.
\]
Since \(\calR\in\calS\),
\[
    \langle\Pi^\perp I,\calR\rangle=0,
    \qquad
    \langle\Pi^\perp I,I\rangle
    =
    \|\Pi^\perp I\|_F^2.
\]
Therefore,
\begin{align*}
    \langle\Lambda_R,I-\calR\rangle
    &=
    d-\|\Pi^\perp I\|_F^2
    +(c_R-1)\Tr(\calR)
    -c_R\langle\calR,\calR\rangle \\
    &\qquad
    +\Tr(Z)-\langle Z,\calR\rangle.
\end{align*}
By the quantitative variance estimates from
\cref{sec:refutation} and the sub-sequential error bounds from \cref{lem:hyper-parameter-bounds},
\[
    \langle\calR,\calR\rangle
    =
    \frac d\rho+\widetilde O(\sqrt d).
\]

\begin{claim}[Auxiliary bounds] \label{clm:refutation-aux-bounds}
	We have the following bounds, \begin{enumerate}
	\item $ \|\Pi^\perp I\|_F^2 = \tilde{O}(\sqrt{d})$;
	\item  $ |\Tr(\calR)|\leq d \cdot B_q(\text{non-ideal}) = o(\delta d)  $, and similarly $ |\Tr(Z)| \leq d \cdot  B_q(\text{non-ideal})$ for the auxiliary function from~\cref{thm:Q-norm};
	\item $ \langle Z,\calR \rangle = o(\delta d) $ .
\end{enumerate}
\end{claim}
Combining the above gives us \[
    \|\Pi^\perp I\|_F^2
    +
    |c_R-1|\,|\Tr(\calR)|
    +
    |\Tr(Z)|
    +
    |\langle Z,\calR\rangle|
    =
    \widetilde O(\sqrt d)+o(\delta d) = o(\delta d)\,,
\]
since $\delta = 1/\poly(\log \log d)$.
Therefore, we obtain
\begin{align*}
    \langle\Lambda_R,I-\calR\rangle
    &=
    d-c_R\frac d\rho+o(\delta d)\\
    &=
    -\frac{c_R-\rho}{\rho}\,d+o(\delta d)\\
    &=
    -\Theta(\delta d)<0.
\end{align*}

Finally, since
\[
    \delta=D^{-1/2}= 1/\poly(\log \log d)= o_d(1),
\]
both of our theorems hold with vanishing slack $o_d(1)$ as desired.

\clearpage
\newpage
\bibliographystyle{alpha}
\bibliography{bib}
\clearpage
\newpage
\appendix

\clearpage
\newpage

\section{Deferred Details for the Iterative Construction} \label{sec:defer-construction}

\subsection{Concrete Dual Iterative Process via Graph Matrices}
\label{app:concrete-dual-iteration}

We modify the ideal dual iteration exactly as in the primal program in~\cref{sec:concrete-primal-graph-mat}: all Chebyshev and MP expansions
are truncated at degree \(D\), only proper horizontal concatenations are
corrected, and only the proper shapes produced by the vertical
concatenation are retained. Since
\[
    \corr_{\mathrm{dual}}(H)
    =
    -\frac{1-\gamma}{\gamma}\,
    \corr_{\mathrm{primal}}(H),
\]
the shape decomposition and error bounds are unchanged up to a bounded
scalar factor.

Define
\[
    H_0^{\mathrm{prop}}
    \coloneqq
    \sum_{j=2}^{D}b_j\fp_j(W_0),
\]
and, for \(i\geq1\),
\[
    H_i^{\mathrm{prop}}
    \coloneqq
    \sum_{j=2}^{D}b_j
    \bigl(\fp_j(W_i)-\fp_j(W_{i-1})\bigr).
\]
Let
\[
    \corr_{\mathrm{dual}}(H)
    \coloneqq
    \frac1\gamma
    \left(
        \calL^\ast M^{-1}\calL(H)-\gamma H
    \right),
\]
and let
\(\corr_{\mathrm{dual,prop}}^{\leq D}(H)\) denote the proper part of
its degree-\(D\) MP expansion, after linearizing all well-behaved
vertical intersections.

\begin{mdframed}[linewidth=0.2pt]
\textbf{Summary of the Concrete Truncated Dual Procedure.}

\medskip
\textbf{Initialization:}
Set
\[
    W_0
    \coloneqq
    C_Fc_R\operatorname{Proper}(\calR).
\]

\medskip
\textbf{Iterative update \(i\to i+1\):}
For \(0\leq i<t^\ast\), set
\[
    \Delta_i^{\mathrm{dual}}
    \coloneqq
    \corr_{\mathrm{dual,prop}}^{\leq D}
    \bigl(H_i^{\mathrm{prop}}\bigr),
    \qquad
    W_{i+1}\coloneqq W_i+\Delta_i^{\mathrm{dual}}.
\]

\medskip
\textbf{Output:}
Let \(W\coloneqq W_{t^\ast}\). The final dual witness is obtained from
\[
    \frac1{C_F}F^{\leq D}(W)-\Pi^\perp I
\]
by removing the aggregate discrepancy below and correcting the final
proper residual.
\end{mdframed}
To account for the change, define
\[
    \mathsf{HorizonErr}^{\mathrm{dual}}(i)
    \coloneqq
    \sum_{j=2}^{D}b_j
    \bigl(P_j(W_i)-\fp_j(W_i)\bigr),
\]
\[
    \mathsf{UpdateErr}^{\mathrm{dual}}(i)
    \coloneqq
    \Delta_i^{\mathrm{dual}}
    -
    \corr_{\mathrm{dual}}(H_i^{\mathrm{prop}}),
\]
and
\[
    \mathsf{IterErr}^{\mathrm{dual}}_i
    \coloneqq
    \mathsf{HorizonErr}^{\mathrm{dual}}(i)
    +
    \sum_{s=0}^{i-1}
    \mathsf{UpdateErr}^{\mathrm{dual}}(s).
\]
Here the update error collects the MP-truncation and discarded vertical
intersection terms.

\begin{proposition}[Invariant for the concrete dual iteration]
\label{prop:concrete-dual-iteration-invariant}
For every \(i\geq0\),
\[
    \Pi^\perp
    \left(
        \frac1{C_F}
        \bigl(
            F^{\leq D}(W_i)
            -
            \mathsf{IterErr}^{\mathrm{dual}}_i
        \bigr)
        -
        I
    \right)
    =
    \frac1{C_F}\Pi^\perp H_i^{\mathrm{prop}}.
\]
\end{proposition}

\begin{proposition}[No reappearance of old shapes in later updates]
\label{prop:no-reappearance-old-shapes}
Suppose a shape \(\sigma\) is first added to the ground set \(\calB(Q_i)\) at
iteration level \(i\). Then \(\sigma\) does not appear in any subsequent update:
for every \(j\ge i\),
\[
    \sigma \notin \calB(Q_{j+1})\setminus \calB(Q_j).
\]
Consequently, the coefficient of \(\sigma\) remains unchanged throughout all
later iterations.
\end{proposition}

\begin{proof}
It suffices to keep track of the deviation term that creates each added shape.
Indeed, every shape added in the update \(Q_{j+1}-Q_j\) is the correction shape
associated with some new deviation term in
\[
    \textsf{New}\text{-}P_\ell(Q_j)
    =
    P_\ell(Q_j)\setminus P_\ell(Q_{j-1}),
    \qquad \ell\ge 2.
\]
By definition, such a new deviation term contains at least one shape from the
most recent level of \(
    \calB(Q_j)\setminus \calB(Q_{j-1}).
\)
Consequently, the correction shape added to \(Q_{j+1}\) contains a newly created
dangling path that was not present before level \(j\).

Thus, once a shape is added to the inner matrix at level $i$, any subsequent addition to the inner matrix contains a dangling path unseen at level-$i$; hence, once \(\sigma\) is added, it never appears in any later update. Since the
iterative construction only appends new shapes and never modifies existing
coefficients, the coefficient of \(\sigma\) remains unchanged in all subsequent
iterations.
\end{proof}

\subsection{Deferred Details for $A^{-1}$ Expansion}
\begin{claim}[Restatement of \cref{clm:inverse-expansion}]
For \(x\) in the Marchenko--Pastur support and \(\gamma<1\),
\[
    \frac{1}{x}
    =
    \frac{1}{1-\gamma}
    \sum_{t\geq 0} (-1)^t \q_t(x).
\]
\end{claim}

\begin{proof}
Let
\(
    G(z,x)
    :=
    \sum_{t\geq 0} \q_t(x) z^t
\)
be the generating function of the polynomials. Using the recurrence above,
we have\[
    G(z,x)
    =
    \frac{1+\gamma z}
    {1-\bigl(x-(1+\gamma)\bigr)z+\gamma z^2}.
\]
Evaluating at \(z=-1\) gives
\[
    \sum_{t\geq 0} (-1)^t \q_t(x)
    =
    G(-1,x)
    =
    \frac{1-\gamma}{x}.
\]
Rearranging proves the claim.
\end{proof}

\subsection{Diagram Illustrations of the Iterative Scheme}

\begin{figure}[h]
    \centering

    \begin{subfigure}[b]{0.6\textwidth}
        \includegraphics[width=\textwidth]{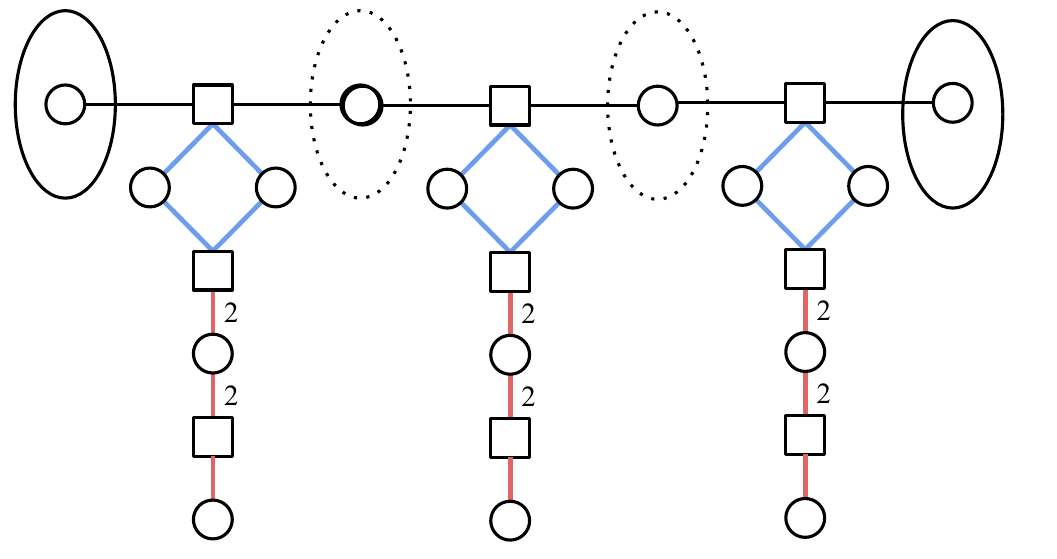}
    \end{subfigure}

\begin{subfigure}[b]{0.6\textwidth}
        \includegraphics[width=\textwidth]{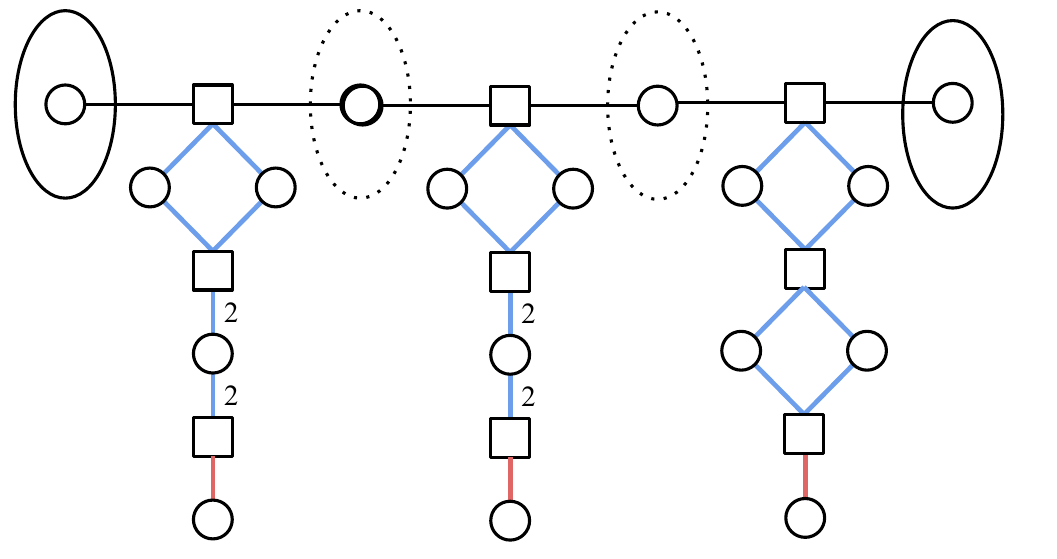}
    \end{subfigure}

%
%
\end{figure}

\section{Deferred Deviation Bounds}
\label{sec:def-scalar}
\subsection{Calculations for Scalar Concentration}

\begin{proof}[Proof of~\cref{clm:base-scalar-variance}]
	
Recall that $M^{-1}$, as well as $M$, has two components. We bound the quadratic form of $\eta$ with respect to each component separately as follows.

 \begin{restatable}{claim}{EtaAInvEtaConcentration} 
\label{clm:concentration-eta-A-eta}
We have
\[
    \frac{1}{d}\cdot \E\!\left[\eta^\top A^{-1}\eta\right]
    =
    \gamma + o_d(1).
\]
\end{restatable}
 
\begin{restatable}{claim}{WoodburyCorrectionConcentration} From the rank-$2$ update in Woodbury, 
\label{clm:concentration-woodbury}
we have
\[
\frac{1}{d} \cdot
\E\!\left[
\eta^\top
\left(
\frac{1}{s^2-ru}\,
A^{-1}
\left(
u\cdot \frac{J_m}{d}
-
s\cdot \frac{\eta \mathbf 1_m^\top+\mathbf 1_m \eta^\top}{d}
+
r\cdot \frac{\eta\eta^\top}{d}
\right)
A^{-1}
\right)
\eta
\right]
=
\frac{\gamma^2}{1-\gamma}.
\]

\end{restatable}
The proof of the claim follows by combining the above two claims.

\end{proof}

\HyperParameterBounds*
\begin{proof}
	The dominant term is already handled, and we focus on the deviations terms here.
	
	\paragraph{Analysis for $r$}
		Recall that $r = \frac{1}{d}  \cdot \1_m^T \cdot A^{-1} \cdot \1_m  $, and  \[ 
 	A^{-1} = \frac{1}{1-\gamma} \sum_{t\geq 0} (-1)^t \cdot \q_t(A)  = \frac{1}{1-\gamma} \sum_{t\geq 0} (-1)^t \cdot \left(\fp_t(A)+o_d(1) \right)
 	\]
 	where the second equality ignores the $o_d(1)$ error term in passing from $\q_t(A)$ to graph matrices of concatenated shapes.
	 
	 	Expanding the \(2q\)-th power, each summand is described in graph-matrix language by a closed walk of \(2q\) floating components, where each component has the form \[
    \Box
    \;-\;
    \bigl(A^{-1}\text{ path via concatenation}\bigr)
    \;-\;
    \Box .
\] Here the two boundary square vertices come from the summation of \(\mathbf 1_m\), and the global normalization contributes a factor \(1/d\) for each component.
	  
	  It suffices for us to bound \[
	 \E[(r - \frac{m}{d})^{2q} ] = \E[(\frac{1}{d} \sum_{t>0} \1_m^\top \q_t(A) \cdot \1_m)^{2q}]  
	 \,,\]  
	 Next, we claim that for any $t$, the block-value is at most \(O(\frac{\sqrt{2q}}{d}\cdot  \sqrt{m} \cdot \sqrt{\gamma}^t ) \):\begin{enumerate}
	 	\item The $\frac{1}{d}$-factor comes from the normalization factor in $r$;
	 	\item For any $t>0$, each edge needs to appear at least twice in the calculation, similar to the usual trace method calculation. Hence, each vertex also makes at least two appearances;
	 	\item The first square vertex needs to appear at least twice, hence we assign a factor of $\sqrt{m} \sqrt{2q}$ for each of its appearance;
	 	\item The remaining factors are identical as those in the calculation for the spectral norm for $\fp_t(A)$, in particular, each $\al$ gadget gives a factor of $\sqrt{\gamma}$ and $\beta$ gadget, as well as error terms from approximation from shape concatenation gives $o_d(1)$ factor.
	 \end{enumerate} 
	 Setting \(q=\polylog(d)\) and recalling  \(m=\Theta(d^2)\), this bound is \[ O_\gamma\!\left( \frac{\sqrt{2q}}{d}\sqrt m \right) = O_\gamma(\sqrt q) = o_d\!\left(\frac{m}{d}\right), \] because \(m/d=\Theta(d)\). By Markov's inequality, the same estimate holds with high probability: \[ \frac{1}{d}\sum_{t\ge 1}\mathbf 1_m^\top P_t(A)\mathbf 1_m = o_d\!\left(\frac{m}{d}\right). \] Combining this with the \(P_0(A)\) contribution, we conclude that \[ r = \frac{1}{1-\gamma}\cdot \frac{m}{d} + o_d\!\left(\frac{m}{d}\right) = (1+o_d(1))\frac{m}{d}\cdot \frac{1}{1-\gamma}. \]

	\paragraph{Analysis for $s$} 
	 Recall that $s = 1+ \frac{1}{d} \cdot \eta^T A^{-1} \1_m$; it suffices for us to bound the $2q$-th moments of $\frac{1}{d} \eta^T A^{-1} \1_m$. In graph matrix language, this is a length-$2q$ walk across floating components where each component has the form\[
    \circ
    \xleftrightarrow{\;h_2\;}
    \Box
    \;-\;
    \bigl(A^{-1}\text{ path via concatenation}\bigr)\;-\;\Box
        \,.
\]
We bound the factors as follows,\begin{enumerate}
	\item The starting circle vertex gets a factor of $\sqrt{d\cdot (2q) }$, and we have a normalization factor of $\frac{1}{d}$;
	\item The $h_2$ edge gives a factor of $\frac{1}{d}$, and we combine that with the factor of the square vertex, that is \[ 
	\frac{\sqrt{2} }{d} \cdot \sqrt{m} =O(1) \,. 
	\]
	\item The subsequent factors are identical as the previous calculation, that we have a factor of $(1+o_d(1))\sqrt{\gamma}$ for each $\cm_\al$ gadget from the concatenation as well as its error term.
\end{enumerate}
For any length-$t\geq 0$ concatenation, this combines to a factor of \[ 
\frac{1}{d } \cdot  \sqrt{d\cdot (2q) }  \cdot O(1) \cdot \sqrt{\gam}^t\,.
\]
Since $\gam<1$, setting $q=\Theta(\polylog(d))$ gives us the desired bound.

\paragraph{Analysis for $u$}
For $t=0$, we have \[
\frac{1}{d} \cdot \eta^\top \q_0(A) \eta  = (1+o_d(1)) \cdot \frac{2md }{d^2} = (1+o_d(1)) \gam \,.
\]
 An analogous calculation applies for the higher moments.
   	
  	For $t=1$, we have \[
  	\q_1(A) = \fp_1(A) = \cm_\al + \cm_\beta \,.
  	 \]
  	 The crucial observation is $\eta^T \cm_\al \eta = o_d(1)$  while the second term gives $\gamma^2$. We first show \[ 
  	 \frac{1}{d } \eta^\top \cdot \cm_\beta \cdot \eta =(1+o_d(1)) \cdot \gamma^2 \,. 
  	 \]  	 w.h.p.
  	 In graph matrix language, each term is described \[ 
  	 \circ \xleftrightarrow{\;h_2\;} (\Box \xleftrightarrow{\;h_2\;} \circ  \xleftrightarrow{\;h_2\;} \Box) \xleftrightarrow{\;h_2\;}\circ 
  	 \,,\] and the dominant term is given by colliding all three circle vertices. Formally, this gives a contribution of \[
  	 \frac{1}{d }\cdot d \cdot  \frac{2}{d^2}\cdot m \cdot \frac{2}{d^2}\cdot (m-1)  =  (1+o_d(1))\cdot \gam^2\,.
  	  \]
  	 as we have $m(m-1)$ choices in total for the two square vertices.
  	 
  	 Next, to see that the other terms give an $o_d(1)$ contribution, we note that by injectivity of $\fp_t(A)$, each other term has mean-$0$. In other words, each edge needs to appear twice, and it cannot happen in a single floating component; thus we have \begin{enumerate}
  	 	\item A factor of $\frac{1}{d}$ from the normalization, and a factor of $\sqrt{d (2q)}$ for the starting circle vertex;
  	 	\item Each subsequent factor can be bounded by similar calculation for $A^{-1}$-path (as well as its error terms) to be at most $O(1)$.
  	 \end{enumerate}
  	 Finally, setting $q=\Theta(\polylog d)$ again gives us a bound of $O(\frac{\poly\log d}{\sqrt{d}}) = o_d(1)$ for each deviation term.   	 Rearranging gives us the desired bound for $u$ as we have \begin{align*}
  	 	 u = -1 + \frac{1}{d}\cdot \eta^\top A^{-1}\eta  &= -1+ \frac{1}{1-\gamma} (\eta^\top \q_0(A) \eta - \eta^\top \q_1(A) \eta  +o_d(1) ) \\&= -1 +\frac{1}{1-\gam }(\gam -\gam^2)\\&=-1+\gam +o_d(1)\,.  	    	 \end{align*}
	\end{proof}

Next, we complete the proofs for the variance calculation for the base inner matrix.
\EtaAInvEtaConcentration*
\WoodburyCorrectionConcentration*

\begin{proof}
	For starters, \cref{clm:concentration-eta-A-eta} follows from the calculation above where we give concentration of $u$. Next, for the second term, it is crucial to recall the following parameters from our previous bound, namely,\[ 
	 r= (1+o_d(1)) \cdot  \frac{m}{d} \cdot \frac{1}{1-\gamma} \,,
	\]
	\[
	s = 1+ o_d(1)\,,
	 \]
	 and \[
	 u = -1 +\gam +o_d(1) \,.
	  \]
	There are three terms that we need to bound here, with our focus on \[ 
	\frac{r}{s^2-ru} \cdot A^{-1} \cdot (\frac{\eta\cdot \eta^T}{d}) A^{-1}\,.
	\] 
Next, we verify that we have
\[ 
	\frac{1}{d}\E [\eta^T ( A^{-1} \cdot (\frac{\eta\cdot \eta^T}{d}) A^{-1}) \eta ] =    \gam^2 +o_d(1)\,.
	\]
	since we have shown concentration of $\frac{1}{d}\cdot \eta^\top A^{-1} \eta =  \gamma +o_d(1) $ in the analysis for $u$.
	Moreover, by our previous bounds, we have
	\[
	\frac{r}{s^2-ru} =  \frac{1}{1/r-u} =  (1+o_d(1)) \cdot \frac{1}{1-\gamma}\,.
	 \]
	 
	 Therefore, \[ 
	 \frac{1}{d}\E\left[\eta^T(  \cdot \frac{r}{s^2-ru} \cdot  A^{-1} \cdot (\frac{\eta\cdot \eta^T}{d}) A^{-1} ) \eta   \right]  = \frac{\gamma^2}{1-\gamma} +o_d(1) \,.
	 \]
	We complete the proof by controlling the remaining two terms. Both are
\(o_d(1)\) with high probability, using the scalar concentration bounds above;
their expectations can be checked directly in the same manner.

For the \(u\cdot J_m\) term, we rewrite its contribution as
\[
    u \cdot
    \frac{\eta^\top A^{-1}\mathbf 1_m}{d}
    \cdot
    \frac{\mathbf 1_m^\top A^{-1}\eta}{d}
    =
    u\cdot (s-1)^2.
\]
Since \(s=1+o_d(1)\) and \(u=O(1)\) with high probability, this term is
\[
    u\cdot (s-1)^2=o_d(1).
\]

For the \(s\cdot(\eta\mathbf 1_m^\top+\mathbf 1_m\eta^\top)\) term, its contribution is
\[
    2s\cdot
    \frac{\eta^\top A^{-1}\eta}{d}
    \cdot
    \frac{\mathbf 1_m^\top A^{-1}\eta}{d}
    =
    2s\cdot
    \frac{\eta^\top A^{-1}\eta}{d}
    \cdot (s-1).
\]
Using \(s=1+o_d(1)\) and
\[
    \frac{\eta^\top A^{-1}\eta}{d}=O(\gamma)
\]
with high probability, we obtain
\[
    2s\cdot
    \frac{\eta^\top A^{-1}\eta}{d}
    \cdot (s-1)
    =
    O(\gamma)\cdot o_d(1)
    =
    o_d(1).
\]
Thus both remaining contributions are negligible.\end{proof}

\subsection{MP Derivations}\label{sec:MP-def-proof}

\begin{lemma}[Full Verification for MP Cancellation]
	\[\cm_{\mathsf{HalfDiamondInt}} =\gam \cdot M_\al + \Delta_{\mathsf{HalfDiamondInt}} \]
where we have \[\|\Delta_{\mathsf{HalfDiamondInt}}\|_{sp} =o_d(1)\,.\]
\end{lemma}

\begin{proof}
	We formally decompose $\Delta_{\mathsf{HalfDiamondInt}} $ into graph matrix basis. Recall that
\[
    M_\alpha[i,j]
    =
    \sum_{a\neq b\in[d]}
    v_i[a]v_i[b]v_j[a]v_j[b]
    =
    2\sum_{a<b\in[d]}v_i[a]v_i[b]v_j[a]v_j[b]
\]
for \(i\neq j\). For such \(i,j\), write
\[
    T_{ab}^{ij}
    :=
    v_i[a]v_i[b]v_j[a]v_j[b].
\]
Then
\[
\begin{aligned}
    \cm_{\mathsf{HalfDiamondInt}}[i,j]
    &=
    2\sum_{a<b\in[d]} T_{ab}^{ij}
    \left(
        \sum_{k\notin\{i,j\}}2v_k[a]^2v_k[b]^2
    \right)                                                     \\
    &=
    2\sum_{a<b\in[d]} T_{ab}^{ij}
    \left(
        \frac{2m}{d^2}
    \right)
    +
    2\sum_{a<b\in[d]} T_{ab}^{ij}
    \left[
        \sum_{k\notin\{i,j\}}2v_k[a]^2v_k[b]^2
        -
        \frac{2m}{d^2}
    \right]                                                     \\
    &=
    \gamma M_\alpha[i,j]
    +
    \Delta_{\mathsf{HalfDiamondInt}}[i,j],
\end{aligned}
\]
where we used \(\gamma=2m/d^2\), and where
\[
    \Delta_{\mathsf{HalfDiamondInt}}[i,j]
    :=
    2\sum_{a<b\in[d]} T_{ab}^{ij}
    \left[
        \sum_{k\notin\{i,j\}}2v_k[a]^2v_k[b]^2
        -
        \frac{2m}{d^2}
    \right].
\]
It remains to show that
\[
    \|\Delta_{\mathsf{HalfDiamondInt}}\|_{\mathrm{sp}}
    =
    o_d(1).
\]

We first rewrite the centered inner term. Since
\[
    v_k[a]^2v_k[b]^2-\frac1{d^2}
    =
    \frac1d\,h_2(v_k[a])
    +
    \frac1d\,h_2(v_k[b])
    +
    h_2(v_k[a])h_2(v_k[b]),
\]
we have
\[
\begin{aligned}
    \sum_{k\notin\{i,j\}}2v_k[a]^2v_k[b]^2-\frac{2m}{d^2}
    &=
    2\sum_{k\notin\{i,j\}}
    \left(
        v_k[a]^2v_k[b]^2-\frac1{d^2}
    \right)
    -
    \frac{4}{d^2}                                               \\
    &=
    \frac{2}{d}
    \sum_{k\notin\{i,j\}}
    \bigl(h_2(v_k[a])+h_2(v_k[b])\bigr)                         \\
    &\qquad
    +
    2\sum_{k\notin\{i,j\}}
    h_2(v_k[a])h_2(v_k[b])
    -
    \frac{4}{d^2}.
\end{aligned}
\]
Therefore
\[
    \Delta_{\mathsf{HalfDiamondInt}}
    =
    \Delta_1+\Delta_2+\Delta_3,
\]
where
\[
\begin{aligned}
    \Delta_1[i,j]
    &:=
    \frac{4}{d}
    \sum_{a<b}
    \sum_{k\notin\{i,j\}}
    T_{ab}^{ij}
    \bigl(h_2(v_k[a])+h_2(v_k[b])\bigr),                       \\
    \Delta_2[i,j]
    &:=
    4
    \sum_{a<b}
    \sum_{k\notin\{i,j\}}
    T_{ab}^{ij}
    h_2(v_k[a])h_2(v_k[b]),                                    \\
    \Delta_3[i,j]
    &:=
    -\frac{8}{d^2}
    \sum_{a<b}T_{ab}^{ij}
    =
    -\frac{4}{d^2}M_\alpha[i,j].
\end{aligned}
\]

For \(\Delta_1\), each summand is a graph matrix obtained from the \(M_\alpha\) shape by adding
one extra square vertex \(k\) and one \(h_2\)-edge from \(k\) to one of the two circle vertices
\(a,b\). In the block-value calculation, the four \(h_1\)-edges from the \(M_\alpha\) part contribute
a factor \(d^{-2}\), the extra \(h_2\)-edge contributes a factor \(d^{-1}\), while the additional internal
square vertex and the two circle vertices contribute at most
\[
    \sqrt m\cdot d.
\]
Thus, before the explicit prefactor \(1/d\), the block value is at most
\[
    \widetilde O\!\left(
        \sqrt m\cdot d\cdot \frac1{d^2}\cdot \frac1d
    \right)
    =
    \widetilde O\!\left(\frac{\sqrt m}{d^2}\right).
\]
Including the prefactor \(1/d\), and using \(m=\Theta(d^2)\), we obtain
\[
    B_q(\Delta_1)
    \le
    \widetilde O\!\left(\frac{\sqrt m}{d^3}\right)
    =
    \widetilde O\!\left(\frac1{d^2}\right)
    =
    o_d(1).
\]
By the trace moment bound for graph matrices, this implies
\(
    \|\Delta_1\|_{\mathrm{sp}}=o_d(1)
\)
with high probability.

Similarly, \(\Delta_2\) is a graph matrix obtained from the \(M_\alpha\) shape by adding one
extra square vertex \(k\) and two \(h_2\)-edges from \(k\) to the two circle vertices \(a,b\). The
four \(h_1\)-edges again contribute \(d^{-2}\), and the two \(h_2\)-edges contribute another
\(d^{-2}\). The vertex factor is again at most \(\sqrt m\cdot d\). Hence
\[
    B_q(\Delta_2)
    \le
    \widetilde O\!\left(
        \sqrt m\cdot d\cdot \frac1{d^2}\cdot \frac1{d^2}
    \right)
    =
    \widetilde O\!\left(\frac{\sqrt m}{d^3}\right)
    =
    \widetilde O\!\left(\frac1{d^2}\right)
    =
    o_d(1).
\]
Therefore
\[
    \|\Delta_2\|_{\mathrm{sp}}=o_d(1)
\]
with high probability.

Combining the three bounds gives
\[
    \|\Delta_{\mathsf{HalfDiamondInt}}\|_{\mathrm{sp}}
    \le
    \|\Delta_1\|_{\mathrm{sp}}
    +
    \|\Delta_2\|_{\mathrm{sp}}
    +
    \|\Delta_3\|_{\mathrm{sp}}
    =
    o_d(1).
\]
\end{proof}

\begin{claim}[Half-flat intersection]
\label{clm:half-flat-int}
We have
\[
    \cm_{\mathsf{HalfFlatInt}}
    =
    \gamma \cm_\beta
    +
    \Delta_{\mathsf{HalfFlatInt}},
\]
where
\[
    \|\Delta_{\mathsf{HalfFlatInt}}\|_{\mathrm{sp}}
    =
    o_d(1).
\]
Equivalently,
\[
    \cm_{\mathsf{HalfFlatInt}}
    =
    \gamma \cm_\beta
    +
    o_d(1)
\]
in spectral norm.
\end{claim}

\begin{proof}
For \(i\neq j\), recall that
\[
    \cm_\beta[i,j]
    =
    \sum_{a\in[d]} h_2(v_i[a])h_2(v_j[a]).
\]
The half-flat intersection has entries
\[
    \cm_{\mathsf{HalfFlatInt}}[i,j]
    =
    \sum_{a\in[d]} h_2(v_i[a])h_2(v_j[a])
    \left(
        \sum_{k\notin\{i,j\}} \frac{2}{d}v_k[a]^2
    \right).
\]
Since \(\E[v_k[a]^2]=1/d\), the deterministic contribution of the inner factor is
\[
    \frac{2(m-2)}{d^2}
    =
    \gamma+o_d(1).
\]
Thus
\[
    \cm_{\mathsf{HalfFlatInt}}[i,j]
    =
    \gamma \cm_\beta[i,j]
    +
    \Delta_{\mathsf{HalfFlatInt}}[i,j],
\]
where
\[
    \Delta_{\mathsf{HalfFlatInt}}[i,j]
    =
    \sum_{a\in[d]} h_2(v_i[a])h_2(v_j[a])
    \left[
        \sum_{k\notin\{i,j\}} \frac{2}{d}v_k[a]^2
        -
        \frac{2m}{d^2}
    \right].
\]
Using \(v_k[a]^2-1/d=h_2(v_k[a])\), we rewrite
\[
    \sum_{k\notin\{i,j\}} \frac{2}{d}v_k[a]^2
    -
    \frac{2m}{d^2}
    =
    \frac{2}{d}
    \sum_{k\notin\{i,j\}} h_2(v_k[a])
    -
    \frac{4}{d^2}.
\]
Hence \(\Delta_{\mathsf{HalfFlatInt}}\) is the sum of two terms: a centered graph matrix obtained
from the \(\cm_\beta\) shape by adding one extra square vertex and one extra centered \(h_2\)-edge,
with prefactor \(1/d\), and a deterministic correction equal to \(O(d^{-2})\cm_\beta\).

The first term has block value \(o_d(1)\), since the additional centered \(h_2\)-edge together with
the prefactor \(1/d\) gives an extra negative power of \(d\). The second term has norm \(o_d(1)\)
because \(\|\cm_\beta\|_{\mathrm{sp}}=O_d(1)\). \end{proof}

\begin{proposition}[Diamond Cancellation]
We have
\[
    \cm_{\mathsf{DiamondInt}}
    =
    \gamma I_m
    +
    \Delta_{\mathsf{DiamondInt}},
\]
where
\[
    \|\Delta_{\mathsf{DiamondInt}}\|_{\mathrm{sp}}
    =
    o_d(1).
\]
Equivalently,
\[
    \cm_{\mathsf{DiamondInt}}
    =
    \gamma I_m
    +
    o_d(1)
\]
in spectral norm.
\end{proposition}

\begin{proof}
A full-diamond intersection identifies the two boundary square vertices,
so \(\cm_{\mathsf{DiamondInt}}\) is diagonal. For \(i\in[m]\),
\[
\cm_{\mathsf{DiamondInt}}[i,i]
=
4\sum_{a<b}
v_i[a]^2v_i[b]^2
\sum_{k\neq i}v_k[a]^2v_k[b]^2.
\]
Set
\[
\rho_i
\coloneqq
2\sum_{a<b}v_i[a]^2v_i[b]^2
=
\|v_i\|_2^4-\sum_a v_i[a]^4
\]
and
\[
S_{ab}^{(-i)}
\coloneqq
2\sum_{k\neq i}v_k[a]^2v_k[b]^2.
\]
Then
\[
\cm_{\mathsf{DiamondInt}}[i,i]
=
2\sum_{a<b}v_i[a]^2v_i[b]^2S_{ab}^{(-i)}.
\]

With high probability,
\[
\max_i|\rho_i-1|=O(\frac{\poly\log d}{\sqrt{d}})= o_d(1),
\qquad
\max_{i,a<b}|S_{ab}^{(-i)}-\gamma|=O(\frac{\poly\log d}{\sqrt{d}})= o_d(1),
\]
where
\[
\E S_{ab}^{(-i)}
=
\frac{2(m-1)}{d^2}
=
\gamma+o_d(1).
\]
Consequently, for any \(i\),
\begin{align*}
\left|
\cm_{\mathsf{DiamondInt}}[i,i]-\gamma
\right|
&\leq
\gamma|\rho_i-1|
+
\rho_i\max_{a<b}|S_{ab}^{(-i)}-\gamma|  \\
&=o_d(1).
\end{align*}

Thus, writing
\[
\Delta_{\mathsf{DiamondInt}}
\coloneqq
\cm_{\mathsf{DiamondInt}}-\gamma I_m,
\]
we have
\[
\|\Delta_{\mathsf{DiamondInt}}\|_{\mathrm{sp}}
=
\max_i
\left|
\cm_{\mathsf{DiamondInt}}[i,i]-\gamma
\right|
=
o_d(1),
\]
which proves the claim.
\end{proof}

\begin{claim}[Proposition B.2 of~\cite{HKPX23}]
For $\cm_D$ defined in~\cref{eq:M_D_def} from the decomposition of $M$, we have
	\[ 
	\|\cm_D\|_{sp} =\tilde{O}(\frac{1}{\sqrt{d}} )= o_d(1)\,.
	\]
	with high probability.
\end{claim}

 \begin{claim}[Coefficient bound in the MP basis (Restatement of~\cref{clm:MP-coef-bnd})]
  Assume \(0<\gamma\le1\). For every \(k\ge0\) and every \(t\ge0\), \[ |c_t(k)| \le \gamma^{-t/2}(2\sqrt{\gamma})^k. \] 
 In particular, for \(k=2q\), \[ |c_t(2q)| \le \gamma^{-t/2}(2\sqrt{\gamma})^{2q}. \] \end{claim} 

\begin{proof}[Proof of~\cref{clm:MP-coef-bnd}]
  Define \( B_k:=\max_{t\ge0}\gamma^{t/2}|c_t(k)|. \) We prove that \[ B_{k+1}\le 2\sqrt{\gamma}\,B_k. \] First consider \(t\ge2\). Using the recurrence, \[ \begin{aligned} \gamma^{t/2}|c_t(k+1)| &\le \gamma^{t/2}|c_{t-1}(k)| + \gamma^{t/2}\gamma |c_{t+1}(k)| \\ &= \sqrt{\gamma}\,\gamma^{(t-1)/2}|c_{t-1}(k)| + \sqrt{\gamma}\,\gamma^{(t+1)/2}|c_{t+1}(k)| \\ &\le 2\sqrt{\gamma}\,B_k. \end{aligned} \] For \(t=1\), we have \[ \begin{aligned} \gamma^{1/2}|c_1(k+1)| &\le \sqrt{\gamma}|c_0(k)| + \sqrt{\gamma}\gamma |c_2(k)| \\ &\le 2\sqrt{\gamma}\,B_k. \end{aligned} \] Finally, for \(t=0\), we use the boundary recurrence: \[ \begin{aligned} |c_0(k+1)| &\le \gamma |c_1(k)|+\gamma |c_0(k)| \\ &= \sqrt{\gamma}\,\gamma^{1/2}|c_1(k)| + \gamma |c_0(k)|. \end{aligned} \] Since \(0<\gamma\le1\), we have \(\gamma\le \sqrt{\gamma}\), and therefore \[ |c_0(k+1)| \le 2\sqrt{\gamma}\,B_k. \] 
Combining the three cases gives \[ B_{k+1}\le 2\sqrt{\gamma}\,B_k. \] Since \(B_0=1\), induction yields \[ B_k\le (2\sqrt{\gamma})^k. \] Equivalently, \[ |c_t(k)| \le \gamma^{-t/2}(2\sqrt{\gamma})^k. \] This proves the claim. \end{proof}

\subsection{Concentration for Matrix Norm Bounds}
\begin{claim}[Trace Moment Bound Implies Spectral Norm Control (Multiplicative)] \label{claim:trace-to-norm-rough}
Let $M\in \mathbb{R}^{n\times n}$ be a random matrix. Suppose that for some integer $q\ge 1$,
\[
\E\!\left[\Tr\!\big((MM^\top)^q\big)\right]
\le
n\left(B_q(M)\right)^{2q}.
\]
Then for any $\eta>0$,
\[
\Pr\!\left(
\|M\|_{\mathrm{sp}} \ge e^{\eta}\, B_q(M)
\right)
\le
n\, e^{-2q\eta}.
\]
In particular, if $q=\Omega(\log^2 n)$, then
\[
\|M\|_{\mathrm{sp}}
\le
(1+o_n(1))\, B_q(M)
\]
with probability $1-o_n(1)$.
\end{claim}
%
%
%
We also make note of the following version that gives an explicit control for the additive deviation.

\begin{corollary}[Quantitative version of \Cref{claim:trace-to-norm-rough}]
\label{cor:trace-to-norm-quant}
Under the assumptions of \Cref{claim:trace-to-norm-rough}, if $q=\log^4 n$, then
\[
\|M\|_{\mathrm{sp}}
\le
\left(1+\frac{1}{\log^2 n}\right) B_q(M)
\]
with probability $1-o_n(1)$.
\end{corollary}

\subsection{Auxiliary Bounds for Refutation}

\begin{proof}[Proof to~\cref{clm:refutation-aux-bounds}]
For convenience, write
\(
    \mathsf B_q\coloneqq B_q(\mathrm{non\text{-}ideal}).
\)
By the parameter verification above,
\[
    \mathsf B_q=d^{-1/2+o(1)}=o(\delta).
\]

Since \(\frac d m\calL^*(\1_m)\in\calS\), the variational
characterization of the orthogonal projection gives
\[
    \|\Pi^\perp I\|_F^2
    \leq
    \left\|I-\frac d m\sum_{i=1}^m v_iv_i^\top\right\|_F^2.
\]
Moreover,
\[
    \E\left\|I-\frac d m\sum_{i=1}^m v_iv_i^\top\right\|_F^2
    =
    \frac{d^2}{m}
    \E\left\|vv^\top-\frac1dI\right\|_F^2
    =
    O(1),
\]
The desired bound then follows by Markov.

Next, every diagonal term on $\calR$ and thereby $Z$ is treated as an error term in our combinatorial analysis in~\cref{lem:error-vertical}, consequently gives
\[
    |\Tr(\calR)|,\ |\Tr(Z)|
    \leq d\,\mathsf B_q
    =
    o(\delta d).
\]

Finally, consider
\[
    \langle Z,\calR\rangle=\Tr(Z\calR).
\]
An ideal closure would require a proper summand of \(Z\) to coincide
with a level-zero summand of \(\calR\). After subtracting
\(c_R\calR\), this cannot occur: horizontal terms contain at least two
backbone pieces, while later correction shapes do not reappear. Thus every mixed closure is non-ideal.
The block-value bound gives \(O(d\mathsf B_q)\), while the aggregate
truncation and termination errors contribute \(o(\delta d)\) by their
\(o(\delta)\) spectral-norm bounds and
\(\|\calR\|_{\mathrm{sp}}=O(1)\). Therefore,
\[
    |\langle Z,\calR\rangle|
    \leq O(d\mathsf B_q)+o(\delta d)
    =
    o(\delta d).
\]
\end{proof}

\section{Deferred Details for Semicircle Polynomials}
\label{app:def-Q}
\begin{proposition}[Backtracking Residual for Semicircle Polynomials]
\label{prop:sc-backtracking-residual}
Define \[ 
	\Delta_{\bti} \coloneqq  \sum _{\tau \in \calB(K)} (c_\tau)^2 \cdot  \cm_{\bti(\tau)} - I_d \,,	\]
	we have \[ 
	\|\Delta_{\bti}\|_{sp} = o_d(1)\,.
	\]
	provided $\Var(K)=1$. Suppose each shape has at most $D_V$ vertices. Then, with high probability,
\[
    \|\Delta_{\bti}\|_{\mathrm{sp}}
    \leq
    \frac{\poly(q,D_V)}{\sqrt d}
    +
    O\left(
        \frac{D_V^2}{d}
        +
        \frac{D_V^2}{m}
    \right).
\]
In particular, if \(m=\Theta(d^2)\) and \(q,D_V=\polylog(d)\), then
\[
    \|\Delta_{\bti}\|_{\mathrm{sp}}
    \leq
    \frac{\polylog(d)}{\sqrt d}
    =
    o_d(1).
\]
\end{proposition}

\begin{proof}
Write
\[
    \Delta_{\bti}=Z_{\bti}+B_{\bti},
\]
where
\[
    Z_{\bti}
    \coloneqq
    \sum_{\tau}
    c(\tau)^2
    \left(
        \cm_{\bti(\tau)}
        -
        \E\cm_{\bti(\tau)}
    \right)
\]
is the centered fluctuation and
\[
    B_{\bti}
    \coloneqq
    \sum_{\tau}
    c(\tau)^2\E\cm_{\bti(\tau)}-I_d
\]
is the injectivity bias.

For a fixed \(\tau\), the backtracking intersection gives
\[
    \E\cm_{\bti(\tau)}
    =
    \vartheta_\tau\Var(\tau)I_d,
\]
where
\[
    \vartheta_\tau
    =
    \frac{(d-1)_{a_\tau}}{d^{a_\tau}}
    \frac{(m)_{b_\tau}}{m^{b_\tau}},
    \qquad
    (x)_r=x(x-1)\cdots(x-r+1).
\]
Indeed, after fixing the boundary circle label, injectivity leaves
\((d-1)_{a_\tau}\) choices for the free circle labels and
\((m)_{b_\tau}\) choices for the square labels. Moreover,
\[
    0\leq 1-\vartheta_\tau
    \leq
    \frac{a_\tau(a_\tau+1)}{2d}
    +
    \frac{b_\tau(b_\tau-1)}{2m}
    =
    O\left(
        \frac{D_V^2}{d}
        +
        \frac{D_V^2}{m}
    \right).
\]
Using
\[
    \sum_{\tau}c(\tau)^2\Var(\tau)=1,
\]
we obtain
\[
    \|B_{\bti}\|_{\mathrm{sp}}
    =
    O\left(
        \frac{D_V^2}{d}
        +
        \frac{D_V^2}{m}
    \right).
\]

Finally, each summand in \(Z_{\bti}\) is a centered well-behaved
backtracking residual. From the subsequent~\cref{lem:centered-sc-backtracking}, \[
    \|Z_{\bti}\|_{\mathrm{sp}}
    \leq
    \frac{\poly(q,D_V)}{\sqrt d}
\]
with high probability. Combining the two estimates proves the claim.
\end{proof}

\begin{lemma}[Centered Backtracking Residual]
\label{lem:centered-sc-backtracking}
Assume \(m=\Theta(d^2)\) and every
\(\tau\in\calB(K)\) has at most \(D_V\) vertices
\[
    \|Z_{\bti}\|_{\mathrm{sp}}
    \leq
    \frac{\poly(q,D_V)}{\sqrt d}.
\]
\end{lemma}

\begin{proof}
For each \(\tau\), Hermite-expand the matched edge-products in the
well-behaved backtracking intersection and remove the all-constant
component. This gives a graph-matrix decomposition
\[
    \cm_{\bti(\tau)}
    -
    \E\cm_{\bti(\tau)}
    =
    \sum_{\rho\in\calR_{\bti}(\tau)}
    a_{\tau,\rho}\cm_\rho,
\]
where every residual shape \(\rho\) retains the decorated record
\(\tau\cdot\tau^\top\) and contains at least one centered,
positive-degree edge factor.

Consider a nonzero trace-walk contribution containing such a residual
block. Its centered edge factor must share its labeled input edge with
another block; otherwise the expectation over that input is zero.
The two matched copies of this edge already occur consecutively inside
the record \(\tau\cdot\tau^\top\). Hence one of its nonboundary
endpoints appears twice in the present block and at least once outside
it, so one of its appearances in the present block is a global middle
appearance.

Thus every residual block has an additional middle appearance relative
to the dominant well-behaved backtracking block. By the existing
factor-assignment bound, this gives a multiplicative gap
\[
    \frac{\poly(q,D_V)}
         {\sqrt{\wt(v)}}
    \leq
    \frac{\poly(q,D_V)}{\sqrt d}.
\]
Therefore,
\[
    B_q\!\left(
        \cm_{\bti(\tau)}
        -
        \E\cm_{\bti(\tau)}
    \right)
    \leq
    \frac{\poly(q,D_V)}{\sqrt d}\Var(\tau).
\]
Applying the linear-combination block-value bound and using
\(\sum_\tau c(\tau)^2\Var(\tau)=1\), we obtain
\[
    B_q(Z_{\bti})
    \leq
    \frac{\poly(q,D_V)}{\sqrt d}.
\]
The spectral-norm bound follows from the block-value-to-norm
implication.
\end{proof}
\section{Deferred Details for Truncation}
\label{sec:defer-details-truncation}

\subsection{Outer-Inner Truncation of the Iterative Process}
\begin{claim}[Coefficient Mass Shift (Restatement of \cref{claim:A-delta-shift})]
For sufficiently small \(\delta>0\),
\[
    A_\delta=1+\Theta(\delta).
\]
\end{claim}

\begin{proof}
Let
\[
    \Delta_\delta(x):=\widetilde F_\delta(x)-F(x).
\]
Since \(F(x)=2x_+\), we have
\[
    \Delta_\delta(x)
    =
    \begin{cases}
        \delta, & x<0,\\
        \delta-2x, & 0\le x<\delta/2,\\
        0, & x\ge \delta/2.
    \end{cases}
\]
In particular,
\[
    |\Delta_\delta(x)|\le \delta
    \qquad
    \text{for all }x\in[-2,2].
\]

We first estimate the shift in the \(P_1\)-coefficient.  Since \(P_1(x)=x\),
\[
    \widetilde a_1-1
    =
    \langle \Delta_\delta,x\rangle_{L^2(\mu_{\mathrm{sc}})}.
\]
On the interval \(x<0\), the contribution is
\[
    \delta\int_{-2}^0 x\,d\mu_{\mathrm{sc}}(x)
    =
    -\Theta(\delta).
\]
On the interval \(0\le x<\delta/2\), we have
\(|\Delta_\delta(x)|\le\delta\) and \(x\le\delta/2\), so the contribution is
\(O(\delta^3)\).  Therefore
\[
    \widetilde a_1=1-\Theta(\delta).
\]

Next we estimate the \(L^2\)-norm of \(\widetilde F_\delta\).  We have
\[
    \|\widetilde F_\delta\|_2^2
    =
    \|F\|_2^2
    +2\langle F,\Delta_\delta\rangle
    +\|\Delta_\delta\|_2^2.
\]
The product \(F\Delta_\delta\) is supported only on
\([0,\delta/2]\), where both \(F(x)=O(\delta)\) and
\(\Delta_\delta(x)=O(\delta)\).  Hence
\[
    \langle F,\Delta_\delta\rangle=O(\delta^3).
\]
Also, since \(|\Delta_\delta|\le \delta\),
\[
    \|\Delta_\delta\|_2^2=O(\delta^2).
\]
Using \(\|F\|_{L^2(\mu_{\mathrm{sc}})}^2=2\), we get
\[
    \|\widetilde F_\delta\|_{L^2(\mu_{\mathrm{sc}})}^2
    =
    2+O(\delta^2).
\]

Now apply Parseval to
\(
    F_\delta=\frac{\widetilde F_\delta}{\widetilde a_1}.
\)
Since the \(P_1\)-coefficient of \(F_\delta\) is normalized to be \(1\),
\[
    C_\delta^2+1+\sum_{j\ge2}(b_j^{(\delta)})^2
    =
    \|F_\delta\|_{L^2(\mu_{\mathrm{sc}})}^2
    =
    \frac{\|\widetilde F_\delta\|_{L^2(\mu_{\mathrm{sc}})}^2}
         {\widetilde a_1^2}.
\]
Using
\[
    \|\widetilde F_\delta\|_{L^2(\mu_{\mathrm{sc}})}^2=2+O(\delta^2),
    \qquad
    \widetilde a_1=1-\Theta(\delta),
\]
we obtain
\[
    C_\delta^2+1+\sum_{j\ge2}(b_j^{(\delta)})^2
    =
    \frac{2+O(\delta^2)}{(1-\Theta(\delta))^2}
    =
    2+\Theta(\delta).
\]
Rearranging gives us \[
    A_\delta
    =
    C_\delta^2+\sum_{j\ge2}(b_j^{(\delta)})^2
    =
    1+\Theta(\delta).
\]
\end{proof}

\subsection{Truncation for \texorpdfstring{$A^{-1}$}{1/A}}
\begin{restatable}[Truncation error for $x^{-1}$ within the MP support]{lemma}{mpinversetruncation} \label{lem:mp-inverse-truncation} Let \(D\) be the truncation parameter for \(A^{-1}\), and define the truncated polynomial as \[ (x^{-1})_{\le D} \coloneqq \frac{1}{1-\gamma} \sum_{t=0}^{D}(-1)^t \q_t(x). \] Then, uniformly for \[ x\in\bigl[(1-\sqrt{\gamma})^2,\,(1+\sqrt{\gamma})^2\bigr], \] we have \[ \left| \frac{1}{x}-(x^{-1})_{\le D} \right| \le O_\gamma\!\left((D+1)(\sqrt{\gamma})^{D+1}\right). \] \end{restatable}\begin{proof}
	
For \(x\in [(1-\sqrt{\gamma})^2, (1+\sqrt{\gam} )^2 ]\), we may parameterize \[ x=1+\gamma+2\sqrt{\gamma}\cos\theta \] for some \(\theta\in[0,\pi]\). We observe that \[ \q_t(x) = \gamma^{t/2} \left( U_t(\cos\theta)+\sqrt{\gamma}\,U_{t-1}(\cos\theta) \right), \] where \(U_t\) is the Chebyshev polynomial of the second kind and \(U_{-1}:=0\). 

Since \( |U_t(\cos\theta)|\le t+1 \) for \(\theta\in[0,\pi]\), we get the uniform bound \[ |\q_t(x)| \le \gamma^{t/2}\bigl((t+1)+\sqrt{\gamma}\,t\bigr) \le 2(t+1)(\sqrt{\gamma})^t. \] Because \(\sqrt{\gamma}<1\), the series \( \sum_{t\ge0}(-1)^t q_t(x) \) converges absolutely and uniformly on \([(1-\sqrt{\gamma})^2, (1+\sqrt{\gam} )^2 ]\). 
Next, we have \[ \frac1x-(x^{-1})_{\le D} = \frac{1}{1-\gamma} \sum_{t>D}(-1)^t \q_t(x). \] Using the bound on \(\q_t(x)\), \[ \left| \frac1x-(x^{-1})_{\le D} \right| \le \frac{2}{1-\gamma} \sum_{t>D}(t+1)(\sqrt{\gamma})^t. \] Letting \(r=\sqrt{\gamma}<1\), we have \[ \sum_{t>D}(t+1)r^t = r^{D+1} \left( \frac{D+2}{1-r} + \frac{r}{(1-r)^2} \right). \] Therefore \[ \left| \frac1x-(x^{-1})_{\le D} \right| \le \frac{2r^{D+1}}{1-\gamma} \left( \frac{D+2}{1-r} + \frac{r}{(1-r)^2} \right) = O_\gamma\!\left((D+1)(\sqrt{\gamma})^{D+1}\right), \] as desired. \end{proof}

\begin{lemma}[Deviation of MP polynomials at the edge]
\label{lem:mp-edge-deviation-detailed}
Let
\[
    G_D(x)=\sum_{t=0}^D c_t\q_t(x),
    \qquad
    B_D:=\sum_{t=0}^D |c_t|.
\]
Then, for every sufficiently small \(\varepsilon\ge0\),
\[
    |G_D(b_\gam+\varepsilon)-G_D(b_\gam)|
    \le
    O_\gam\!\left(
        B_D\,\varepsilon (D+1)^3 e^{O_\gam(D\sqrt{\varepsilon})}
    \right),
\]
and similarly
\[
    |G_D(a_\gam-\varepsilon)-G_D(a_\gam)|
    \le
    O_\gam\!\left(
        B_D\,\varepsilon (D+1)^3 e^{O_\gam(D\sqrt{\varepsilon})}
    \right),
\]
where
\[
    a_\gam=(1-\sqrt{\gam})^2,\qquad
    b_\gam=(1+\sqrt{\gam})^2.
\]
\end{lemma}

\begin{proof}
We prove the upper edge. Write
\[
    b_\gam+\varepsilon
    =
    1+\gam+2\sqrt{\gam}\cosh s,
    \qquad
    s=\arccosh\!\left(1+\frac{\varepsilon}{2\sqrt{\gam}}\right)
    =
    O_\gam(\sqrt{\varepsilon}).
\]
Using
\[
    \q_t(1+\gam+2\sqrt{\gam}y)
    =
    \gam^{t/2}\bigl(U_t(y)+\sqrt{\gam}\,U_{t-1}(y)\bigr),
\]
together with
\[
    |U_j(\cosh s)-U_j(1)|
    \le
    O\!\left((j+1)^3s^2 e^{js}\right),
\]
we get
\[
    |\q_t(b_\gam+\varepsilon)-\q_t(b_\gam)|
    \le
    O_\gam\!\left(
        (t+1)^3\varepsilon e^{O_\gam(t\sqrt{\varepsilon})}
    \right).
\]
Therefore
\[
\begin{aligned}
    |G_D(b_\gam+\varepsilon)-G_D(b_\gam)|
    &\le
    \sum_{t=0}^D |c_t|\,
    |\q_t(b_\gam+\varepsilon)-\q_t(b_\gam)| \\
    &\le
    O_\gam\!\left(
        B_D\,\varepsilon (D+1)^3 e^{O_\gam(D\sqrt{\varepsilon})}
    \right).
\end{aligned}
\]
The lower edge follows identically from \(U_j(-x)=(-1)^jU_j(x)\).
\end{proof}

\mpinverseedgefluctuation*
\begin{proof}
Let
\[
    p_D(x):=(x^{-1})_{\le D}.
\]
On the exact MP support, via~\cref{lem:mp-inverse-truncation}, we have
\[
    \sup_{x\in[a_\gam,b_\gam]}
    \left|
        \frac1x-p_D(x)
    \right|
    \le
    O_\gam\!\left((D+1)(\sqrt{\gam})^{D+1}\right).
\]
Now take \(\lambda\in[a_\gam-\varepsilon_A,b_\gam+\varepsilon_A]\). If \(\lambda\in[a_\gam,b_\gam]\), this already gives the claim. Otherwise, \(\lambda\) lies within distance at most \(\varepsilon_A\) of one of the two edges. For example, if \(\lambda=b_\gam+\delta\), \(0\le\delta\le\varepsilon_A\), then
\[
\begin{aligned}
    \left|\frac1\lambda-p_D(\lambda)\right|
    &\le
    \left|\frac1\lambda-\frac1{b_\gam}\right|
    +
    \left|\frac1{b_\gam}-p_D(b_\gam)\right|
    +
    |p_D(b_\gam)-p_D(\lambda)| \\
    &\le
    O_\gam\!\left((D+1)(\sqrt{\gam})^{D+1}\right)
    +
    O_\gam\!\left(
        \delta (D+1)^4 e^{O_\gam(D\sqrt{\delta})}
    \right),
\end{aligned}
\]
where the last term uses \cref{lem:mp-edge-deviation-detailed} with
\[
    B_D=\frac{1}{1-\gam}\sum_{t=0}^D 1=O_\gam(D+1).
\]
The lower edge is the same, using \(a_\gam-\varepsilon_A>0\). Taking the supremum over eigenvalues of \(A\) gives the spectral-norm bound.
\end{proof}

\end{document}